\documentclass[12pt,english]{article}
\usepackage{amsfonts,babel,amssymb,amsmath,amsthm}
\usepackage{graphicx}
\usepackage{framed}
\usepackage{mathabx}

\newtheorem{proposition}{Proposition}[section]

\newtheorem{lemma}[proposition]{Lemma}

\newcommand{\smallfrac}[2]{{\textstyle\frac{#1}{#2}}}

\newcommand{\qq}{\widehat{\mathbf q}}
\newcommand{\dimm}{\mathrm{dim}\,}
\newcommand{\tildeP}{\widetilde{\mathcal P}}
\newcommand{\tildePP}{\widetilde{\boldsymbol{\mathcal P}}}

\title{From Raviart-Thomas to HDG:\\ a personal voyage}

\author{Francisco-Javier Sayas\\ Department of Mathematical Sciences, University of Delaware}

\date{\today}

\begin{document}

\maketitle

This document has been motivated  by a course entitled {\em From Raviart-Thomas to HDG}, which I have prepared for {\em C\'adiz Num\'erica 2013 -- Course and Encounter on Numerical Analysis} (C\'adiz, Spain -- June 2013). It is an introduction to the techniques for local analysis of classical mixed methods for diffusion problems and how they motivate the Hybridizable Discontinuous Galerkin method. They assume knowledge of basic techniques on Finite Element Analysis, but not of Mixed Finite Element Methods. Many estimates (especially those fitting in the general category of scaling arguments) are carried out with almost excruciating detail. All final bounds are given for solutions with maximal regularity. This being an introductory text, no attempt has been made at dealing with complicated or anisotropic meshes, or at producing estimates for solutions with very low regularity. I have preferred to give most technical results in a theorem-and-proof format, but have kept a more argumentative (while fully rigorous) tone for the main estimates of the three families of methods we will be dealing with. I have also tried to give some precise references to original sources, but have not been extremely thorough in this, partially because of not being able to discern who-was-first in some particular instances.

\begin{framed}
\noindent Corrections and suggestions are welcome.
\end{framed}

\tableofcontents

\section{Getting ready}

This section gives a collection of transformation techniques, equalities, and bounds to relate quantities defined in physical variables to quantities in a reference configuration. They belong to the category of what  the FEM community calls {\em scaling arguments}. I will  not be using anything specific from FEM theory, but the arguments will be familiar to anyone aware of these techniques: proper introductions can be found in classic books as Ciarlet's \cite{Ciarlet:1978}, Brenner and Scott's \cite{BrSc:2008} or Braess's \cite{Braess:2007}.

\subsection{A personal view of Piola transforms}

\paragraph{Reference configurations.}
To fix ideas, let $\widehat K$ be the reference triangle/tetrahedron 
\[
\widehat K:=\{ \widehat{\mathbf x}\in \mathbb R^d\,:\, \widehat x_i\ge 0, \forall i,\quad \mathbf e\cdot\widehat{\mathbf x}\le 1  \}, \qquad \mbox{where $\mathbf e=(1,\ldots,1)^\top$}.
\]
Given a general triangle/tetrahedron, we consider fixed affine invertible maps
\[
\mathrm F:\widehat K \to K, \qquad \mathrm G:=\mathrm F^{-1}:K \to \widehat K
\]
and will denote
\[
\mathrm B:=\mathrm D\mathrm F, \qquad \mathrm B^{-1}=\mathrm D \mathrm G, \qquad |J|:=|\det\mathrm B|.
\]
At the present moment it is not necessary to show dependence on $K$ of all of these quantities. We will also consider a piecewise constant function $|a|:\partial \widehat K\to \mathbb R$ containing the absolute value of the determinant of the tangential derivative matrix of $\mathrm F|_{\partial \widehat K}$. In particular
\[
\int_K f=\int_{\widehat K} f\circ \mathrm F\, |J|, \qquad \int_{\partial K} f =\int_{\partial\widehat K} f\circ \mathrm F|_{\partial\widehat K}\, |a|.
\]
Note that the latter equality is not a typical application of the change of variable formula (that applies on volumes) but the possibility of parametrizing both surface integrals from the same sets of coordinates. Outward pointing unit normal fields will be denoted $\mathbf n:\partial K \to\mathbb R^d$ and $\widehat{\mathbf n}:\partial\widehat K\to\mathbb R^d$.

\paragraph{Changes of variables.} We will deal with three types of fields: scalar fields defined on the volume $u$, vector fields defined on the volume $\mathbf q$, and scalar fields defined on the boundary $\mu$. For reasons we will repeatedly see, inner products are often better understood as duality products. We will then write
\[
(u,u^*)_K:=\int_K\, u\, u^*, \qquad (\mathbf q,\mathbf q^*)_K:=\int_K \mathbf q\cdot\mathbf q^*, \qquad \langle \mu,\mu^*\rangle_{\partial K}:=\int_{\partial K} \mu\, \mu^*
\]
thinking momentarily that starred quantities are dual variables and unstarred quantities are primal. The rules for changes of variables are given as follows
\begin{framed}
\[
\begin{array}{lll}
\mbox{Primal} \hspace{2cm} & u:K\to \mathbb R\hspace{2cm} & \widehat u:=u\circ \mathrm F,\\[1.5ex]
& \mathbf q:K \to \mathbb R^d & \widehat{\mathbf q}:=|J|\mathrm B^{-1} \mathbf q\circ\mathrm F\\[1.5ex]
& \mu:\partial K \to \mathbb R & \widehat\mu:=\mu\circ \mathrm F|_{\partial\widehat K}\\[1.5ex]
\mbox{Dual} & u^*:K\to \mathbb R & \widecheck u:=|J|\, u^*\circ \mathrm F,\\[1.5ex]
& \mathbf q^*:K \to \mathbb R^d & \widecheck{\mathbf q}:=\mathrm B^\top \mathbf q^*\circ\mathrm F\\[1.5ex]
& \mu^*:\partial K \to \mathbb R & \widecheck\mu^*:=|a|\,\mu^*\circ \mathrm F|_{\partial\widehat K}
\end{array}
\]
\end{framed}
\noindent
so that we can change variables in inner/duality products
\begin{subequations}\label{eq:0}
\begin{alignat}{4}
(u,u^*)_K &=(\widehat u,\widecheck u^*)_{\widehat K},\\
(\mathbf q,\mathbf q^*)_K &= (\widehat{\mathbf q},\widecheck{\mathbf q}^*)_{\widehat K},\\
\label{eq:0c}
\langle \mu,\mu^*\rangle_{\partial K} &=\langle \widehat\mu,\widecheck\mu^*\rangle_{\partial\widehat K}.
\end{alignat}
\end{subequations}
These set of rules might be somewhat whimsical, but there are many reasons for them. The best (and deepest) explanations go through $p-$forms, a context where this and much more makes complete sense. The interest reader might want to have a look at the massive work or Douglas Arnold, Richard Falk, and Ragnar Winther \cite{ArFaWi:2006} and \cite{ArFaWi:2010}, where all of this (and considerably more) is given a very general treatment.

\paragraph{A remark on restrictions.} The restriction to $\partial K$ of a function $u:K\to \mathbb R$ (be it a trace restriction or a more classical one) will be denoted either $u|_{\partial K}$ or just $u$. Normal restrictions of vector fields $\mathbf q:K \to \mathbb R^d$, will just be denoted $\mathbf q\cdot\mathbf n$. Note that
\begin{equation}\label{eq:E1}
\widehat{u|_{\partial K}}=\widehat u|_{\partial\widehat K},
\end{equation}
a property that is not satisfied by the check transformations.

\paragraph{Changes of variables and operators.} The following result shows how the gradient and divergence operators, and the normal trace to the boundary transform primal quantities to dual quantities.

\begin{proposition}[Changes of variables]
For smooth enough fields,
\begin{subequations}
\begin{alignat}{4}
\label{eq:1a}
\widehat{\mathrm{div}}\,\widehat{\mathbf q} &= \widecheck{\mathrm{div}\,\mathbf q},\\
\label{eq:1b}
\widehat\nabla \,\widehat u &=\widecheck{\nabla u},\\
\label{eq:1c}
\widehat{\mathbf q}\cdot\widehat{\mathbf n} &= \widecheck{\mathbf q\cdot\mathbf n},
\end{alignat}
\end{subequations}
and therefore
\begin{subequations}
\begin{alignat}{4}
\label{eq:2a}
(\mathrm{div}\,\mathbf q,u)_K &= (\widehat{\mathrm{div}}\,\widehat{\mathbf q},\widehat u)_{\widehat K},\\
\label{eq:2b}
(\mathbf q,\nabla u)_K &= (\widehat{\mathbf q},\widehat\nabla \widehat u)_{\widehat K},\\
\label{eq:2c}
\langle \mathbf q\cdot\mathbf n,\mu\rangle_{\partial K} &= \langle\widehat{\mathbf q}\cdot\widehat{\mathbf n},\widehat\mu\rangle_{\partial\widehat K}.
\end{alignat}
\end{subequations}
\end{proposition}

\begin{proof}
This moment is as good as any other to learn how to use fast index notation: repeated subindices will be automatically added, and a comma followed by one or more indices denotes differentiation w.r.t. the corresponding variable. Note that $(\mathrm D\mathrm F)_{ij}=\mathrm B_{ij}=\mathrm F_{i,j}$. Differentiating in $\widehat u=u\circ \mathrm F$, we have
\[
\widehat u_{,i}=(u_{,j}\circ\mathrm F)\,\mathrm F_{j,i}=\mathrm B_{ji} u_{,j}\circ\mathrm F,
\]
which is \eqref{eq:1b}. All other formulas can be proved using duality arguments. We will do some of them by hand. For instance, the formula
\[
\mathbf q=|J|^{-1} \mathrm B\widehat{\mathbf q}\circ\mathrm G \qquad (\mbox{recall that $\mathrm G=\mathrm F^{-1}$}),
\]
is written componentwise as
\[
q_i=|J|^{-1} \mathrm B_{ij}\widehat q_j\circ\mathrm G
\]
and leads to
\begin{alignat*}{4}
q_{i,i}&=|J|^{-1} \mathrm B_{ij} (\widehat q_{j,k}\circ\mathrm G)\mathrm G_{k,i} & \qquad &\mbox{(chain rule)}\\
&=|J|^{-1} \mathrm B_{ij}\mathrm B_{ki}^{-1} \widehat q_{j,k}\circ\mathrm G & & (\mathrm D\mathrm G=\mathrm B^{-1})\\
&=|J|^{-1} \delta_{kj} \widehat q_{j,k}\circ\mathrm G & &(\mathrm B^{-1}\mathrm B=\mathrm I)\\
&=|J|^{-1} \widehat q_{j,j}\circ\mathrm G,
\end{alignat*}
that is, $|J|q_{i,i}\circ\mathrm F=\widehat q_{i,i}$, which proves \eqref{eq:1a}. The changes of variables \eqref{eq:1a} and \eqref{eq:1b} and the integral rules \eqref{eq:0} --which motivated our notation--, imply \eqref{eq:2a} and \eqref{eq:2b}. The divergence theorem proves then \eqref{eq:2c}. It then follows from \eqref{eq:0} that
\[
\langle \widecheck{\mathbf q\cdot\mathbf n}-\widehat{\mathbf q}\cdot\widehat{\mathbf n},\widehat \mu\rangle_{\partial\widehat K}=0.
\] 
Taking $\widehat\mu=\widecheck{\mathbf q\cdot\mathbf n}-\widehat{\mathbf q}\cdot\widehat{\mathbf n}$, \eqref{eq:1c} follows.
\end{proof}

\subsection{Scaling inequalities}

\paragraph{Some no-brainers.}
We start with quite obvious changes of variables for integrals
\begin{subequations}\label{eq:4}
\begin{alignat}{6}
\label{eq:4a}
\|u\|_K &\le  |J|^{-1/2}\,\|\widehat u\|_{\widehat K},  &  \|\widehat u\|_{\widehat K} &\le  |J|^{1/2} \| u\|_K &\qquad &\mbox{(obviously equal)}\\
\label{eq:4b}
\|\mathbf q\|_K &\le  |J|^{-1/2} \| \mathrm B\|\,\|\widehat{\mathbf q}\|_{\widehat K}, \qquad &  \|\widehat{\mathbf q}\|_K &\le  |J|^{1/2} \| \mathrm B^{-1}\|\,\| \mathbf q\|_{K},\\
\label{eq:4c}
\|\mu\|_{\partial K} &\le \| a\|_{L^\infty}^{1/2} \| \widehat\mu\|_{\widehat K} & \|\widehat\mu\|_{\partial\widehat K} & \le \| a^{-1}\|_{L^\infty}^{1/2} \|\mu\|_{\partial K}.
\end{alignat}
\end{subequations}
At this precise point, we start assuming that there is a collection of triangles/tetrahedra $\mathcal T_h$. The diameter of $K$ is denoted $h_K$. We typically write $h:=\max_{K\in \mathcal T_h} h_K$. The collection $\mathcal T_h$ is called shape-regular when $h_K\le C \rho_K$, where $\rho_K$ is the diameter of the largest ball that we can insert in $K$. This definition includes a constant $C=C(\mathcal T_h)$ that always exists. For it to make sense {\em with $C$ independent of $h$}, we have to assume that there is actually a collection of triangulations $\mathcal T_h$, that are just tagged with this general parameter $h$. Readers are supposed to be in the know of this FEM abuse of notation, and we will not insist on this any longer. Wiggled inequalities will be extremely useful to avoid the introduction of constants that are independent of $h$, possibly different in each occurence:
\[
a_h \lesssim b_h \qquad \mbox{ means }\qquad a_h \le C \, b_h \mbox{ with $C>0$ independent of $h$},
\]
and
\[
a_h\approx b_h\qquad \mbox{means}\qquad a_h\lesssim b_h\lesssim a_h.
\]
Shape-regularity implies
\begin{subequations}\label{eq:5}
\begin{alignat}{6}
& \|\mathrm B_K\|\lesssim h_K, \qquad& & \|\mathrm B_K^{-1}\|\lesssim h_K^{-1}\\
& |J_K|\lesssim h^d_K, & &|J_K|^{-1}\lesssim h^{-d}_K &\qquad  & (|J_K|\approx h^d_K)\\
& \| a_K\|_{L^\infty} \lesssim h^{d-1}_K, & & \| a^{-1}_K\|_{L^\infty} \lesssim h^{1-d}_K,
\end{alignat}
\end{subequations}
and then \eqref{eq:4} ends up being
\begin{equation}\label{eq:6}
\| u\|_K\approx h^{\frac{d}2}_K \|\widehat u\|_{\widehat K}, \qquad 
\|\mathbf q\|_K \approx h^{1-\frac{d}2}_K \| \widehat{\mathbf q}\|_{\widehat K},
 \qquad \|\mu\|_{\partial K}\approx h^{\frac{d-1}2} \|\widehat\mu\|_{\partial\widehat K}.
\end{equation}
\paragraph*{Sobolev seminorms.} When derivatives are introduced (through Sobolev seminorms), the scaling properties for scalar volume fields are well known
\begin{subequations}\label{eq:7}
\begin{alignat}{6}
|u|_{m,K} &\lesssim |J|^{1/2} \|\mathrm B^{-1}\|^m |\widehat u|_{m,\widehat K},\\
|\widehat u|_{m,\widehat K} & \lesssim |J|^{-1/2} \|\mathrm B\|^m |u|_{m,K}.
\end{alignat}
Applying these inequalities to the components of $\widehat{\mathbf q}\circ \mathrm G$, we can prove
\begin{alignat}{6}
|\mathbf q|_{m,K} & \lesssim |J|^{-1/2} \|\mathrm B\|\,\|\mathrm B^{-1}\|^m |\widehat{\mathbf q}|_{m,\widehat K},\\
|\widehat{\mathbf q}|_{m,\widehat K} &\lesssim |J|^{1/2} \|\mathrm B^{-1}\| \,\|\mathrm B\|^m |\mathbf q|_{m,K}.
\end{alignat}
\end{subequations}
This and shape-regularity \eqref{eq:5} yield 
\begin{equation}\label{eq:8}
| u|_{m,K} \approx h^{\frac{d}2-m} |\widehat u|_{m,\widehat K},\qquad
|\mathbf q|_{m,K} \approx h^{1-\frac{d}2-m} |\widehat{\mathbf q}|_{m,\widehat K}.
\end{equation}

\section{The Raviart-Thomas projection}

In this section, we will review some well-known (and some not so well-known) facts about the natural interpolation operator associated to the Raviart-Thomas space. The RT space is named after Pierre-Arnaud Raviart and Jean-Marie Thomas. Their original and very often quoted paper \cite{RaTh:1977} contains a two-dimensional finite element for the $\mathbf H(\mathrm{div},\Omega)$ space, which is slightly different from the one that is now known as the RT space. The three dimensional space is one of the many elements that appears in the first of the two big finite element papers by Jean-Claude N\'ed\'elec \cite{Nedelec:1980}. 

\subsection{Facts you might (not) know about polynomials}

\paragraph{Polynomials.} Polynomials in $d$ variables with (total) degree at most $k$ will be denoted $\mathcal P_k$. It is often convenient to recall the dimension by reminding the reader where the polynomials are defined. To avoid being too wordy, here's some fast notation:
\begin{itemize}
\item $\mathcal P_k(K)$, where $K\in \mathcal T_h$, is the space of polynomials of degree at most $k$ defined on the element $K$.
\item Whenever needed, we will just write $\mathcal P_{-1}(K)=0$, to avoid singling out some particular cases.
\item $\boldsymbol{\mathcal P}_k(K):=\mathcal P_k(K)^d$.
\item $\widetilde{\mathcal P}_k(K)$ are homogeneous polynomials of degree $k$.
\item $\mathbf m\in \mathcal P_1(K)^d$ is the function $\mathbf m(\mathbf x):= \mathbf x$. There is a tradition to call this function just $\mathbf x$, but then $\widehat{\mathbf x}$ has two possible meanings, one as the variable in the reference element, and the other one as the function
\begin{equation}\label{eq:9}
\widehat{\mathbf m}(\widehat{\mathbf x})=|J| (\widehat{\mathbf x}+\mathrm B^{-1}\mathbf b), \qquad \mbox{where } \mathbf b=\mathrm F(\mathbf 0).
\end{equation}
\item $\mathcal E(K)$ is the set of edges of the triangle $K$ or faces of the tetrahedron $K$ (so that $\cup_{e\in \mathcal E(K)} \overline e=\partial K$. I'll be lazy enough to call everything a {\em face}, while using the letter $e$ (as in edge) to refer to these edges/faces.
\item $\mathcal P_k(e)$ with $e\in \mathcal E(K)$  is the space of ($d-1$)-variate polynomials on tangential coordinates.
\item $\mathcal R_k(\partial K)=\prod_{e\in \mathcal E(K)} \mathcal P_k(e)$ are piecewise polynomial functions on $\partial K$.
\end{itemize}
Easy facts about dimensions:
\[
\mathrm{dim}\,\mathcal P_k(K)={k+d\choose d}, \qquad \mathrm{dim}\,\widetilde{\mathcal P}_k(K)=\mathrm{dim}\,\mathcal P_k(e)={k+d-1\choose d-1},
\]
\[ 
\mathrm{dim}\, \mathcal R_k(\partial K)= (d+1)  {k+d-1\choose d-1}.
\]
Two more spaces we will use are
\begin{alignat*}{4}
\mathcal P_k^\bot(K)&:=\{ u\in \mathcal P_k(K)\,:\, (u,v)_K=0 \quad\forall v\in \mathcal P_{k-1}(K)\},\\
\boldsymbol{\mathcal P}_k^\bot(K)&:=\mathcal P_k^\bot(K)^d=\{ \mathbf q\in \boldsymbol{\mathcal P}_k(K)\,:\, (\mathbf q,\mathbf r)_K=0\quad\forall \mathbf r\in \boldsymbol{\mathcal P}_{k-1}(K)\}. 
\end{alignat*}
The following decompositions are direct orthogonal sums
\[
\mathcal P_k(K)=\mathcal P_{k-1}(K)\oplus \mathcal P_k^\bot(K), \qquad \boldsymbol{\mathcal P}_k(K) =\boldsymbol{\mathcal P}_{k-1}(K)\oplus \boldsymbol{\mathcal P}_k^\bot(K).
\]
It is also clear that
\begin{equation}\label{eq:10}
\mathrm{dim}\,\mathcal P_k^\bot(K)=\mathrm{dim}\,\widetilde{\mathcal P}_k(K)=\mathrm{dim}\,\mathcal P_k(e), \qquad e\in \mathcal E(K).
\end{equation}

\begin{lemma}\label{lemma:2.1}\
\begin{itemize}
\item[{\rm (a)}] If $u\in \mathcal P_k^\bot(K)$ satisfies $u|_e=0$ on some $e\in \mathcal E(K)$, then $u=0$.
\item[{\rm (b)}] If $\mathbf q\in \boldsymbol{\mathcal P}_k^\bot(K)$ satisfies $\mathbf q\cdot\mathbf n=0$ on $\partial K$, then $\mathbf q=\mathbf 0$. 
\end{itemize}
\end{lemma}

\begin{proof}
Part (a) is quite simple. The face $e$ is contained in the hyperplane $p(\mathbf x)=\mathbf x\cdot\mathbf n_e-c=0$, and then $u=p\, v$, where $v\in \mathcal P_{k-1}(K)$. But then,
\[
0=(u,v)_K=(p\, v,v)_K=(p, v^2)_K\qquad \mbox{while $p<0$ in $K$},
\]
so $v=0$. To prove part (b), we use (a) applied to the polynomial $\mathbf q\cdot\mathbf n_e$ for each $e\in \mathcal E(K)$. Then $\mathbf q\cdot\mathbf n_e=0$ in $K$ for every $e\in \mathcal E(K)$. This shows that $\mathbf q=\mathbf 0$. Note that the result also holds if $\mathbf q\cdot\mathbf n=0$ on $\partial K\setminus e$, for any $e\in \mathcal E(K)$, since only $d$ normal vectors are needed to have a basis of $\mathbb  R^d$.
\end{proof}

\begin{lemma}\label{lemma:2.2}
The following decomposition is a direct orthogonal sum:
\begin{equation}\label{eq:11}
\mathcal R_k(\partial K)=\{ u|_{\partial K}\,:\, u\in \mathcal P_k^\bot(K)\}\oplus\{\mathbf q\cdot\mathbf n\,:\, \mathbf q\in \boldsymbol{\mathcal P}_k^\bot(K)\}.
\end{equation}
\end{lemma}

\begin{proof}
By Lemma \ref{lemma:2.1}, the operators $R_1:\mathcal P_k^\bot(K)\to \mathcal R_k(\partial K)$ and $R_2:\boldsymbol{\mathcal P}_k^\bot(K)\to \mathcal R_k(\partial K)$, given by
\[
R_1u:=u|_{\partial K}, \qquad R_2\mathbf q:=\mathbf q\cdot\mathbf n,
\]
are one-to-one. On the other hand
\[
\langle R_1u,R_2\mathbf q\rangle_{\partial K}=\langle u,\mathbf q\cdot\mathbf n\rangle_{\partial K}=(\nabla u,\mathbf q)_K+(u,\mathrm{div}\,\mathbf q)_K=0,
\]
since $\nabla u\in \boldsymbol{\mathcal P}_{k-1}(K)$ and $\mathrm{div}\,\mathbf q\in \mathcal P_{k-1}(K)$. This means that the sum $\mathrm{Range}\,R_1\oplus\mathrm{Range}\, R_2$ (the right-hand side of \eqref{eq:11}) is orthogonal. The result follows from an easy dimension count:
\begin{alignat*}{4}
\mathrm{dim} (\mathrm{Range}\,R_1\,\oplus\,\mathrm{Range}\, R_2) &=\mathrm{dim}\,\mathrm{Range}\, R_1 \,+\,\mathrm{dim}\,\mathrm{Range}\, R_2 &\qquad &\mbox{(direct sum)}\\
&=\mathrm{dim}\,\mathcal P_k^\bot(K)\,+\,\mathrm{dim}\,\boldsymbol{\mathcal P}_k^\bot(K) &&\mbox{($R_1$ and $R_2$ are 1-1)}\\
&=(d+1)\mathrm{dim}\, \mathcal P_k^\bot(K) \\
&=\mathrm{dim} \,\mathcal R_k(\partial K). & &\mbox{(by \eqref{eq:10})}
\end{alignat*}
This finishes the proof. (This simple lemma appears in \cite{CoSa:TA}.)
\end{proof}

\paragraph{Polynomials and Piola transforms.}
It is also easy to note that polynomials are preserved by the changes of variables, in both possible roles of primal and dual functions:
\begin{alignat*}{6}
\mbox{Primal} &\qquad & 
u\in \mathcal P_k(K) &\qquad \Longleftrightarrow\qquad \widehat u\in \mathcal P_k(\widehat K),\\ 
&  &\mathbf q \in \boldsymbol{\mathcal P}_k(K)  &\qquad\Longleftrightarrow\qquad   \widehat{\mathbf q}\in \boldsymbol{\mathcal P}_k(\widehat K), \\
& & \mu \in \mathcal R_k(\partial K) &\qquad\Longleftrightarrow\qquad \widehat\mu \in \mathcal R_k(\partial\widehat K),\\
\mbox{Dual} & &
u^*\in \mathcal P_k(K) &\qquad \Longleftrightarrow\qquad \widecheck u^*\in \mathcal P_k(\widehat K),\\ 
& & \mathbf q^* \in \boldsymbol{\mathcal P}_k(K)  &\qquad\Longleftrightarrow\qquad   \widecheck{\mathbf q}^*\in \boldsymbol{\mathcal P}_k(\widehat K), \\
& & \mu^* \in \mathcal R_k(\partial K) &\qquad\Longleftrightarrow\qquad \widecheck\mu^* \in \mathcal R_k(\partial\widehat K).
\end{alignat*}
These relations and \eqref{eq:0} show that the spaces $\mathcal P_k^\bot(K)$ and $\boldsymbol{\mathcal P}_k^\bot(K)$ are also preserved with the changes of variables:
\begin{alignat*}{6}
\mbox{Primal} &\qquad & 
u\in \mathcal P_k^\bot(K) &\qquad \Longleftrightarrow\qquad \widehat u\in \mathcal P_k^\bot(\widehat K),\\ 
&  &\mathbf q \in \boldsymbol{\mathcal P}_k^\bot(K)  &\qquad\Longleftrightarrow\qquad   \widehat{\mathbf q}\in \boldsymbol{\mathcal P}_k^\bot(\widehat K), \\
\mbox{Dual} & &
u^*\in \mathcal P_k^\bot(K) &\qquad \Longleftrightarrow\qquad \widecheck u^*\in \mathcal P_k^\bot(\widehat K),\\ 
& & \mathbf q^* \in \boldsymbol{\mathcal P}_k^\bot(K)  &\qquad\Longleftrightarrow\qquad   \widecheck{\mathbf q}^*\in \boldsymbol{\mathcal P}_k^\bot(\widehat K). \\
\end{alignat*}

\subsection{The space and the projection}

\paragraph{The Raviart-Thomas space.} The RT space in $K$ is defined as
\[
\mathcal{RT}_k(K):=\boldsymbol{\mathcal P}_k(K)\oplus \mathbf m\,\widetilde{\mathcal P}_k(K) \qquad \mbox{(recall that $\mathbf m(\mathbf x)=\mathbf x$)}.
\]
It is quite obvious that
\[
\boldsymbol{\mathcal P}_k(K)\subset \mathcal{RT}_k(K)\subset \boldsymbol{\mathcal P}_{k+1}(K),
\]
both inclusions being proper,
and
\begin{equation}\label{eq:12}
\mathrm{dim}\, \mathcal{RT}_k(K)=d\,{k+d\choose d}+{k+d-1\choose d-1}=\mathrm{dim}\, \boldsymbol{\mathcal P}_{k-1}(K)+\mathrm{dim}\,\mathcal R_k(\partial K).
\end{equation}
(The last equality takes one minute to prove.) Slightly less obvious facts are collected in the next proposition.

\begin{proposition}\label{prop:2.3}\
\begin{itemize}
\item[{\rm (a)}] $\mathbf q\cdot\mathbf n\in \mathcal R_k(\partial K)$ for all $\mathbf q\in \mathcal{RT}_k(K)$.
\item[{\rm (b)}] $\mathbf q\in \mathcal{RT}_k(K)$ if and only if $\widehat{\mathbf q}\in \mathcal{RT}_k(\widehat K)$.
\item[{\rm (c)}] If $\mathrm{div}\,\mathbf q=0$ with $\mathbf q\in \mathcal{RT}_k(K)$, then $\mathbf q\in \boldsymbol{\mathcal P}_k(K)$.
\item[{\rm (d)}] $\mathrm{div}\,\mathcal{RT}_k(K)=\mathcal P_k(K)$.
\end{itemize}
\end{proposition}

\begin{proof} 
It is clear that to prove (a)-(b) we only need to worry about functions $\mathbf m\, p$, where $p\in \widetilde{\mathcal P}_k(K)$. It is also clear that $\mathbf m\cdot\mathbf n \in \mathcal R_0(\partial K)$ (the faces are parts of planes with normal vector $\mathbf n$, so $\mathbf x\cdot\mathbf n=c$). Then $(\mathbf m\,p)|_{\partial K}\cdot\mathbf n\in \mathcal R_0(\partial K)\cdot \mathcal R_{k}(\partial K)\subset \mathcal R_k(\partial K)$, which proves (a). Part (b) follows from \eqref{eq:9}, that is from the fact that $\widehat{\mathbf m}(\widehat{\mathbf x})=|J|\widehat{\mathbf x}+\mathbf c$.

If $\mathbf q=\mathbf p+\mathbf m\,p$, with $\mathbf p\in \boldsymbol{\mathcal P}_k(K)$ and $p\in \widetilde{\mathcal P}_k(K)$,
then by Euler's homogeneous function theorem:
\begin{equation}\label{eq:13A}
\mathrm{div}(\mathbf p+\mathbf m\, p)=\mathrm{div}\,\mathbf p+\mathbf m\cdot\nabla p+ (\mathrm{div}\,\mathbf m)\, p =\mathrm{div}\,\mathbf p+ (k+d)\, p\in \mathcal P_{k-1}(K)\oplus\widetilde{\mathcal P}_k(K),
\end{equation}
and therefore $p=0$. This proves (c).

Since $\mathcal{RT}_k(K)\subset\boldsymbol{\mathcal P}_{k+1}(K)$, it is obvious that $\mathrm{div}\,\mathcal{RT}_k(K)\subseteq\mathcal P_k(K)$. Given now $u\in \mathcal P_k(K)$, we write
\[
u=u_0+u_1+\ldots+u_k, \qquad u_j\in \widetilde{\mathcal P}_j(K)\quad \forall j,
\]
and then use Euler's homogeneous function theorem and the computation in \eqref{eq:13A} to guess
\[
\mathbf p = \bigg(\sum_{j=0}^k \smallfrac1{j+d} u_j\bigg)\,\mathbf m \in \mathbf m \mathcal P_k(K)\subset \mathcal{RT}_k(K).
\]
A simple computation shows then that $\mathrm{div}\,\mathbf p=u$.
\end{proof}

\begin{framed}
\noindent{\bf The Raviart-Thomas projection.} Let $\mathbf q:K\to\mathbb R^d$ be sufficiently smooth. The RT projection is $\boldsymbol\Pi^{\mathrm{RT}}\mathbf q\in \mathcal{RT}_k(K)$ characterized by the equations
\begin{subequations}\label{eq:RT}
\begin{alignat}{4}
\label{eq:RTa}
(\boldsymbol\Pi^{\mathrm{RT}}\mathbf q,\mathbf r)_K &=(\mathbf q,\mathbf r)_K &\qquad &\forall \mathbf r\in \boldsymbol{\mathcal P}_{k-1}(K),\\
\langle \boldsymbol\Pi^{\mathrm{RT}}\mathbf q\cdot\mathbf n,\mu\rangle_{\partial K}&=\langle \mathbf q\cdot\mathbf n,\mu\rangle_{\partial K} & &\forall \mu \in \mathcal R_k(\partial K).
\end{alignat}
Attached to this projection, there is a scalar field projection, $\Pi_k$, which is just the $L^2(K)$-projection onto $\mathcal P_k(K)$:
\begin{equation}
(\Pi_k u,v)_K=(u,v)_K \qquad \forall v\in \mathcal P_k(K).
\end{equation}
\end{subequations}
Note that as $\mathcal P_{-1}(K)=0$, equations \eqref{eq:RTa} are void for $k=0$.
\end{framed}

\begin{proposition}[Definition of the RT projection]
Equations \eqref{eq:RT} are uniquely solvable and therefore define a projection onto $\mathcal{RT}_k(K)$.
\end{proposition}

\begin{proof}
Note that \eqref{eq:12} implies that \eqref{eq:RT} is equivalent to a square system of linear equations, so we only need to prove uniqueness of solution. Let then $\mathbf q\in \mathcal{RT}_k(K)$ satisfy
\begin{subequations}
\begin{alignat}{4}\label{eq:14a}
(\mathbf q,\mathbf r)_K& =0 &\qquad  &\forall \mathbf r\in \boldsymbol{\mathcal P}_{k-1}(K),\\
\label{eq:14b}
\langle \mathbf q\cdot\mathbf n,\mu\rangle_{\partial K}&=0 & &\forall \mu \in \mathcal R_k(\partial K).
\end{alignat}
\end{subequations}
Then
\[
\|\mathrm{div}\,\mathbf q\|_K^2 =\langle \mathbf q\cdot\mathbf n,(\mathrm{div}\,\mathbf q)|_{\partial K}\rangle_{\partial K}-(\mathbf q,\nabla (\mathrm{div}\mathbf q))_K=0
\]
by \eqref{eq:14a} and \eqref{eq:14b}. This implies that $\mathrm{div}\,\mathbf q=0$ and then, by Proposition \ref{prop:2.3}(d), it follows that $\mathbf q\in \boldsymbol{\mathcal P}_k(K)$. Then \eqref{eq:14a} means that $\mathbf q\in \boldsymbol{\mathcal P}_k^\bot(K)$, while \eqref{eq:14b} implies that $\mathbf q\cdot\mathbf n=0$. Using now Lemma \ref{lemma:2.2}(b), it follows that $\mathbf q=\mathbf 0$.
\end{proof}

\paragraph{The commutativity property.} Note that for all $u\in \mathcal P_k(K)$,
\begin{alignat*}{4}
(\mathrm{div}\,\boldsymbol\Pi^{\mathrm{RT}}\mathbf q,u)_K &=\langle\boldsymbol\Pi^{\mathrm{RT}}\mathbf q\cdot\mathbf n,u\rangle_{\partial K}-(\mathbf \Pi^{\mathrm{RT}}\mathbf q,\nabla u)_K\\
&=\langle \mathbf q\cdot\mathbf n,u\rangle_{\partial K}-(\mathbf q,\nabla u)_K\\
&=(\mathrm{div}\,\mathbf q,u)_K,
\end{alignat*}
i.e.
\begin{equation}\label{eq:commRT}
\mathrm{div}\,\boldsymbol\Pi^{\mathrm{RT}}\mathbf q =\Pi_k \mathrm{div}\,\mathbf q.
\end{equation}

\paragraph{Invariance by Piola transforms.} Our next goal is to relate the RT projection in the physical element \eqref{eq:RT} with the one defined in the reference element: given $\widehat{\mathbf q}$ we look for $\widehat{\boldsymbol\Pi}^{\mathrm{RT}}\widehat{\mathbf q}\in \mathcal{RT}_k(\widehat K)$, satisfying
\begin{subequations}\label{eq:17}
\begin{alignat}{4}
(\widehat{\boldsymbol\Pi}^{\mathrm{RT}}\widehat{\mathbf q},\mathbf r)_{\widehat K} &=(\widehat{\mathbf q},\mathbf r)_{\widehat K} &\qquad &\forall \mathbf r\in \boldsymbol{\mathcal P}_{k-1}(\widehat K),\\
\langle \widehat{\boldsymbol\Pi}^{\mathrm{RT}}\widehat{\mathbf q}\cdot\widehat{\mathbf n},\mu\rangle_{\partial\widehat K}&=\langle \widehat{\mathbf q}\cdot\widehat{\mathbf n},\mu\rangle_{\partial\widehat K} & &\forall \mu \in \mathcal R_k(\partial\widehat K).
\end{alignat}
\end{subequations}
Note that by \eqref{eq:0} 
\[
(\widehat{\boldsymbol\Pi^{\mathrm{RT}}\mathbf q},\widecheck{\mathbf r})_{\widehat K} =(\boldsymbol\Pi^{\mathrm{RT}}\mathbf q,\mathbf r)_K =(\mathbf q,\mathbf r)_K=(\widehat{\mathbf q},\widecheck{\mathbf r})_{\widehat K}\qquad \forall \mathbf r\in \boldsymbol{\mathcal P}_{k-1}(K),
\]
and by \eqref{eq:2c}
\[
\langle\widehat{\boldsymbol\Pi^{\mathrm{RT}}\mathbf q}\cdot\widehat{\mathbf n},\widecheck\mu\rangle_{\partial\widehat K}=\langle \boldsymbol\Pi^{\mathrm{RT}}\mathbf q\cdot\mathbf n,\mu\rangle_{\partial K}=\langle \mathbf q\cdot\mathbf n,\mu\rangle_{\partial K}=\langle\widehat{\mathbf q}\cdot\widehat{\mathbf n},\widecheck\mu\rangle_{\partial\widehat K}\qquad \forall \mu \in \mathcal R_k(\partial K).
\]
However, since the test spaces transform well under the check rules and so does the RT space w.r.t. the hat rule (Proposition \ref{prop:2.3}(c)), it follows that
\begin{equation}\label{eq:18}
\widehat{\boldsymbol\Pi}^{\mathrm{RT}}\widehat{\mathbf q}=\widehat{\boldsymbol\Pi^{\mathrm{RT}}\mathbf q}.
\end{equation}

\subsection{Estimates and liftings}\label{sec:2.3}

By looking at the equations on the reference element \eqref{eq:17}, and using a basis of the space $\mathcal{RT}_k(K)$, it is easy to see how
\begin{equation}\label{eq:19}
\| \widehat{\boldsymbol\Pi}^{\mathrm{RT}}\widehat{\mathbf q}\|_{\widehat K}\lesssim \|\widehat{\mathbf q}\|_{\widehat K}+\|\widehat{\mathbf q}\cdot\widehat{\mathbf n}\|_{\partial\widehat K}\lesssim \|\widehat{\mathbf q}\|_{1,\widehat K}\qquad \forall \widehat{\mathbf q}\in \mathbf H^1(\widehat K):=H^1(\widehat K)^d.
\end{equation}
This inequality actually shows how the RT projection is well defined on $\mathbf H^{\frac12+\varepsilon}(\widehat K)$, which is a space that guarantees the existence of a classical trace operator, so that $\widehat{\mathbf q}\cdot\widehat{\mathbf n}\in L^2(\partial \widehat K)$. (We will not deal with these low regularity cases in these notes though.) Another easy fact follows from a compactness argument (a.k.a. the Bramble-Hilbert lemma): since $\widehat\Pi^{\mathrm{RT}}$ preserves the space $\boldsymbol{\mathcal P}_k(\widehat K)\subset \mathcal {RT}_k(\widehat K)$, then
\begin{equation}\label{eq:20}
\|\qq-\widehat{\boldsymbol\Pi}^{\mathrm{RT}}\qq\|_{\widehat K}\lesssim |\qq|_{k+1,\widehat K}\qquad \forall \qq\in\mathbf H^{k+1}(\widehat K).
\end{equation}

\begin{proposition}[Estimates for the RT projection]\label{prop:2.4}
On shape-regular triangulations and for sufficiently smooth $\mathbf q$,
\begin{itemize}
\item[{\rm (a)}]
$
\| \boldsymbol\Pi^{\mathrm{RT}}\mathbf q\|_K \lesssim \|\mathbf q\|_K+ h_K|\mathbf q|_{1,K},
$
\item[{\rm (b)}]
$
\|\mathbf q-\boldsymbol\Pi^{\mathrm{RT}}\mathbf q\|_K \lesssim h^{k+1}_K|\mathbf q|_{k+1,K},
$
\item[{\rm (c)}]
$
\|\mathrm{div}\,\mathbf q-\mathrm{div}\,\boldsymbol\Pi^{\mathrm{RT}}\mathbf q\|_K\lesssim h^{k+1}_K|\mathrm{div}\,\mathbf q|_{k+1,K}.
$
\end{itemize}
\end{proposition}

\begin{proof}
The results follow from the estimates in the reference element \eqref{eq:19}-\eqref{eq:20}, the relation between the projection and the projection on the reference element \eqref{eq:18}, and scaling arguments \eqref{eq:6} and \eqref{eq:8} (or their more primitive forms in \eqref{eq:4} and \eqref{eq:7}). For instance
\begin{alignat*}{4}
\|\boldsymbol\Pi^{\mathrm{RT}}\mathbf q\|_K & \le |J_K|^{-1/2} \|\mathrm B_K\|\|\widehat{\boldsymbol\Pi^{\mathrm{RT}}\mathbf q}\|_{\widehat K} &\qquad & \mbox{(by \eqref{eq:4})}\\
& = |J_K|^{-1/2} \|\mathrm B_K\|\|\widehat{\boldsymbol\Pi}^{\mathrm{RT}}\qq\|_{\widehat K} &\qquad & \mbox{(by \eqref{eq:18})}\\
& \lesssim |J_K|^{-1/2} \|\mathrm B_K\| \|\qq\|_{1,\widehat K} & & \mbox{(by \eqref{eq:19})}\\
& \lesssim \| \mathrm B_K\|\,\|\mathrm B_K^{-1}\|\,(\|\mathbf q\|_K+ \|\mathrm B_K\|\,|\mathbf q|_{1,K}). & & \mbox{(by \eqref{eq:4} and \eqref{eq:7})}
\end{alignat*} 
Similarly 
\begin{alignat*}{4}
\|\mathbf q-\boldsymbol\Pi^{\mathrm{RT}}\mathbf q\|_K &\le |J|^{-1/2} \|\mathrm B_K\| \,\|\qq-\widehat{\boldsymbol\Pi^{\mathrm{RT}}\mathbf q}\|_{\widehat K} &\qquad &\mbox{(by \eqref{eq:4})}\\
& = |J_K|^{-1/2} \|\mathrm B_K\| \,\|\qq-\widehat{\boldsymbol\Pi}^{\mathrm{RT}}\qq\|_{\widehat K} &\qquad &\mbox{(by \eqref{eq:18})}\\
& \lesssim |J_K|^{-1/2} \|\mathrm B_K\|\, |\qq|_{k+1,\widehat K} & &\mbox{(by \eqref{eq:20}, i.e., Bramble-Hilbert)}\\
& \le \| \mathrm B_K^{-1}\|\,\|\mathrm B_K\|^{k+2} |\mathbf q|_{k+1,K}. & & \mbox{(by \eqref{eq:7})}
\end{alignat*}
From these inequalities to (a) and (b) we only need to use the shape-regularity bounds \eqref{eq:5}. To prove (c) we use the commutation property \eqref{eq:commRT} and a bunch of scaling arguments:
\begin{alignat*}{4}
\|\mathrm{div}\,\mathbf q-\mathrm{div} \,\boldsymbol\Pi^{\mathrm{RT}}\mathbf q\|_K &= \|\mathrm{div}\,\mathbf q-\Pi_k\mathrm{div}\,\mathbf q\|_K &\qquad & \mbox{(by commutativity \eqref{eq:commRT})}\\
& = |J_K|^{-1/2} \| \widehat{\mathrm{div}\,\mathbf q} - \widehat{\Pi_k\mathrm{div}\,\mathbf q}\|_{\widehat K} & & \mbox{(by \eqref{eq:4})}\\
& = |J_K|^{-1/2} \|\widehat{\mathrm{div}\,\mathbf q}-\widehat\Pi_k \widehat{\mathrm{div}\,\mathbf q}\|_{\widehat K} & &\mbox{(easy argument)}\\
& \lesssim |J_K|^{-1/2} |\widehat{\mathrm{div}\,\mathbf q}|_{k+1,\widehat K} & &\mbox{(compactness-Bramble-Hilbert)}\\
& \le \|\mathrm B_K\|^{k+1} |\mathrm{div}\,\mathbf q|_{k+1,K}. & & \mbox{(by \eqref{eq:7})}
\end{alignat*}
The result now follows readily.
\end{proof}

\begin{proposition}[RT local lifting of the normal trace]\label{prop:2.5}
There exists a linear operator $\mathbf L^{\mathrm{RT}}:\mathcal R_k(\partial K)\to \mathcal{RT}_k(K)$ such that
\[
(\mathbf L^{\mathrm{RT}}\mu)\cdot\mathbf n=\mu \quad \mbox{and}\quad \|\mathbf L^{\mathrm{RT}}\mu\|_K \lesssim h_K^{1/2}\|\mu\|_{\partial K}\qquad \forall \mu \in \mathcal R_k(\partial K).
\]
\end{proposition}

\begin{proof}
Let $\mathbf q=\mathbf L^{\mathrm{RT}} \mu\in \mathcal{RT}_k(K)$ be defined as
\[
\mathbf q:=|J_K|^{-1}\mathrm B_K\qq\circ\mathrm G_K,
\]
where $\qq\in \mathcal{RT}_k(\widehat K)$ is the solution of the discrete equations in the reference domain:
\begin{subequations}\label{eq:21}
\begin{alignat}{4}
(\qq,\mathbf r)_{\widehat K} &=0 &\qquad &\forall\mathbf r\in \boldsymbol{\mathcal P}_{k-1}(\widehat K),\\
\langle\qq\cdot\widehat{\mathbf n},\xi\rangle_{\partial\widehat K}&=\langle\widecheck\mu,\xi\rangle_{\partial\widehat K} & & \forall \xi\in \mathcal R_k(\partial\widehat K).
\end{alignat}
\end{subequations}
Note that by \eqref{eq:2c}
\[
\langle \mathbf q\cdot\mathbf n,\xi\rangle_{\partial K} =\langle \qq\cdot\widehat{\mathbf n},\widehat\xi\rangle_{\partial\widehat K}=\langle \widecheck\mu,\widehat\xi\rangle_{\partial\widehat K}=\langle\mu,\xi\rangle_{\partial K}\qquad \forall \xi\in \mathcal R_k(\partial K),
\]
and therefore $\mathbf q\cdot\mathbf n=\mu$. Also
\begin{alignat*}{4}
\|\mathbf q\|_K &\le |J_K|^{-1/2} \|\mathrm B_K\| \|\qq\|_{\widehat K} &\qquad &\mbox{(by \eqref{eq:4b})}\\
&\lesssim |J_K|^{-1/2}\|\mathrm B_K\|\,\|\widecheck\mu\|_{\partial\widehat K} & &\mbox{(finite dimensional argument on \eqref{eq:21})}\\
&\lesssim |J_K|^{-1/2} \| \mathrm B_K\|\, \| a\|_{L^\infty}^{1/2} \|\mu\|_{\partial K}, & & \mbox{(simple argument based on \eqref{eq:4c})}
\end{alignat*}
and the bound follows from estimating all the above geometric quantities using \eqref{eq:5}
\[
\lesssim |J_K|^{-1/2} \| \mathrm B_K\|\, \| a\|_{L^\infty}^{1/2} \lesssim h_K^{-\frac{d}2} h_K h_K^{\frac{d-1}2}=h_K^{1/2}.
\]
This finishes the proof.
\end{proof}

\paragraph{More jargon.} When in the middle of a Finite Element argument, we use that we are dealing with polynomials of a fixed degree (or any finite dimensional space) on the reference domain, it is common to refer to the argument as a {\em finite dimensional argument.} This leads to inequalities with constants depending on polynomial degrees and dimension, but on nothing else.

\section{Projection-based analysis of RT}

In this section we are going to develop a fully detailed analysis of the RT approximation of the system
\begin{subequations}\label{eq:22}
\begin{alignat}{4}
\label{eq:22a}
\kappa^{-1}\mathbf q+\nabla u &=0 & \qquad &\mbox{in $\Omega$},\\
\mathrm{div}\,\mathbf q &=f & & \mbox{in $\Omega$},\\
\label{eq:22c}
u&=u_0 & & \mbox{on $\Gamma:=\partial\Omega$},
\end{alignat}
\end{subequations}
where $\Omega$ is a polygonal/polyhedral domain, $f\in L^2(\Omega)$, $\kappa\in L^\infty(\Omega)$ is strictly positive (so that $\kappa^{-1}\in L^\infty(\Omega)$ as well). We are not going to use any of the results of the Brezzi theory of mixed problems \cite{Brezzi:1974}. Our approach is going to be more local and much less elegant. (It has to be noted, that Brezzi's theory and Fortin's 
inversion of the discrete divergence \cite{Fortin:1977} will be constantly in the background, and we will just be repeating ideas that can be expressed in more abstract terms.)

It is very simple to see (no need of mixed variational formulation) that problem \eqref{eq:22} has a unique solution $(\mathbf q,u)\in \mathbf H(\mathrm{div},\Omega)\times H^1(\Omega)$, where
\[
\mathbf H(\mathrm{div},\Omega):=\{ \mathbf q\in \mathbf L^2(\Omega):=L^2(\Omega)^d\,:\, \mathrm{div}\,\mathbf q\in L^2(\Omega)\}.
\]

\paragraph{Discretization.} 
For discretization let us consider a conforming partition $\mathcal T_h$ of $\Omega$ into triangles/tetrahedra, and the discrete spaces
\begin{subequations}
\begin{alignat}{4}
\mathbf V_h&:=\prod_{K\in \mathcal T_h} \mathcal{RT}_h(K)=\{ \mathbf q_h:\Omega\to \mathbb R^d\,:\, \mathbf q_h|_K \in \mathcal{RT}_k(K) \quad\forall K \in \mathcal T_h\},\\
W_h &:=\prod_{K\in\mathcal T_h} \mathcal P_k(K) =\{ u_h :\Omega\to \mathbb R\,:\, u_h|_K\in \mathcal P_k(K)\quad\forall K\in \mathcal T_h\},\\
\mathbf V_h^{\mathrm{div}}&:=\mathbf V_h\cap \mathbf H(\mathrm{div},\Omega).
\end{alignat}
\end{subequations}
The RT approximation of \eqref{eq:22} is a simple Galerkin scheme for a variational formulation of \eqref{eq:22} obtained after integrating by parts in \eqref{eq:22a}, which naturally incorporates the BC \eqref{eq:22c}: we look for
\begin{subequations}\label{eq:24}
\begin{equation}
(\mathbf q_h,u_h)\in \mathbf V_h^{\mathrm{div}}\times W_h
\end{equation}
satisfying
\begin{alignat}{6}
\label{eq:24b}
& (\kappa^{-1}\mathbf q_h,\mathbf r)_\Omega -(u_h,\mathrm{div}\,\mathbf r)_\Omega &&=-\langle u_0,\mathbf r\cdot\mathbf n\rangle_\Gamma &\qquad &\forall \mathbf r\in \mathbf V_h^{\mathrm{div}},\\
\label{eq:24c}
& ( \mathrm{div}\,\mathbf q_h,v)_\Omega &&=(f,v)_\Omega & &\forall v\in W_h.
\end{alignat}
\end{subequations}
Existence and uniqueness of solution of \eqref{eq:24} will follow from the arguments in the next section. Instead of using this Galerkin formulation we will insert Lagrange multipliers to handle the continuity of the normal components of $\mathbf q_h$ across element interfaces: this leads to a formulation with three fields due to Douglas Arnold and Franco Brezzi \cite{ArBr:1985}. In Section 
\ref{sec:4} we will show how to eliminate then interior fields in order to to build a discrete system that has lost the saddle point structure and only contains degrees of freedom on the faces.

\subsection{The Arnold-Brezzi formulation}

\paragraph{Imposing continuity of the normal components.}
The key idea leading to the next equivalent presentation of equations \eqref{eq:24} is an observation about what conditions functions in $\mathbf V_h$ must satisfy in order to belong to the space $\mathbf V_h^{\mathrm{div}}$. Let $K_1, K_2\in \mathcal T_h$ meet in one face $\overline{K_1}\cap \overline{K_2}=\overline e$, with $e\in \mathcal E_h$. Given $\mathbf q_h \in \mathbf V_h$, it is easy to prove (based on Proposition \ref{prop:2.3}(a)) that
\begin{equation}\label{eq:255}
\mathbf q_h\in \mathbf V_h^{\mathrm{div}} \qquad \Longrightarrow \qquad \langle\mathbf q_h|_{K_1}\cdot\mathbf n_1+\mathbf q_h|_{K_2}\cdot\mathbf n_2,\mu\rangle_e=0 \quad\forall \mu\in \mathcal P_k(e). 
\end{equation}
Instead of writing the matching condition in the right-hand side of \eqref{eq:255} for each interior $e\in \mathcal E_h$ looking for the elements on both sides of $e$, we can do as follows. For $\mathbf q:\Omega\to\mathbb R^d$ and $\mu:\cup\{ \overline e\,:\, e\in\mathcal E_h\} \to\mathbb R$, we write
\begin{equation}\label{eq:26}
\langle \mathbf q\cdot\mathbf n,\mu\rangle_{\partial \mathcal T_h\setminus\Gamma}:=\sum_{K\in \mathcal T_h} \langle\mathbf q|_K\cdot\mathbf n_K,\mu\rangle_{\partial K\setminus\Gamma}.
\end{equation}
We then consider the space
\begin{equation}\label{eq:27}
M_h:=\prod_{e\in \mathcal E_h} \mathcal P_k(e) =\{ \mu:\cup\{ \overline e\,:\, e\in\mathcal E_h\} \to\mathbb R\,:\, \mu|_e\in \mathcal P_k(e) \quad \forall e\in \mathcal E_h\},
\end{equation}
and finally group all conditions in \eqref{eq:255} as
\begin{equation}\label{eq:28}
\langle\mathbf q_h\cdot\mathbf n,\mu\rangle_{\partial\mathcal T_h\setminus\Gamma} =0 \qquad \forall \mu\in M_h.
\end{equation}
This condition is then not only necessary but sufficient, that is, given $\mathbf q_h\in \mathbf V_h$, condition \eqref{eq:28} is equivalent to the property $\mathbf q_h \in \mathbf V_h^{\mathrm{div}}$. It is to be noticed that condition \eqref{eq:28} does not use the entire space $M_h$ but only the subspace
\[
M_h^\circ:=\{\mu \in M_h\,:\, \mu|_\Gamma =0\}.
\]
The remaining part is the space 
\[
M_h^\Gamma := \{\mu \in M_h\,:\,\mu|_e=0\quad \forall e\in \mathcal E_h^\circ\}\equiv \prod_{e\in \mathcal E_h\, e\subset\Gamma} \mathcal P_k(e),
\]
where $\mathcal E_h^\circ$ is the set of interior faces.

\paragraph{Reaching the formulation.} The side condition \eqref{eq:26} will be compensated with the inclusion of a Lagrange multiplier, which will end up being an approximation of $u$ on the skeleton of the triangulation (on the union of all faces of all the elements). Equation \eqref{eq:24c} is naturally local, since the space $W_h$ is a product space of polynomial spaces on the elements. Instead of using \eqref{eq:24b}, we will consider a similar equation based on each element. Note that we will not do any passage through the reference element in the remainder of this section, which will allow us to use the hat symbol to refer to a particular unknown of the discrete system. We then look for
\begin{subequations}\label{eq:25}
\begin{equation}
(\mathbf q_h,u_h,\widehat u_h)\in \mathbf V_h \times W_h \times M_h
\end{equation}
satisfying for all $K\in \mathcal T_h$
\begin{alignat}{6}
\label{eq:25b}
&(\kappa^{-1}\mathbf q_h,\mathbf r)_K -(u_h,\mathrm{div}\,\mathbf r)_K -\langle \widehat u_h,\mathbf r\cdot\mathbf n\rangle_{\partial K}&&=0 &\qquad &\forall \mathbf r \in \mathcal{RT}_k(K),\\
\label{eq:25c}
&(\mathrm{div}\,\mathbf q_h,w)_K &&=(f,w)_K && \forall w\in \mathcal P_k(K),
\end{alignat}
as well as
\begin{alignat}{6}
\langle \mathbf q_h\cdot\mathbf n,\mu\rangle_{\partial\mathcal T_h\setminus\Gamma} &=0 &\qquad& \forall \mu \in M_h^\circ,\\
\langle \widehat u_h,\mu\rangle_\Gamma &=\langle u_0,\mu\rangle_\Gamma &&\forall \mu \in M_h^\Gamma.
\end{alignat}
\end{subequations}
These equations can be written in global form using the following notation
\[
(u,v)_{\mathcal T_h}=\sum_{K\in \mathcal T_h}(u,v)_K, \qquad \langle \mathbf q\cdot\mathbf n,\mu\rangle_{\partial\mathcal T_h}=\sum_{K\in \mathcal T_h}\langle \mathbf q\cdot\mathbf n,\mu\rangle_{\partial K}
\]
(compare with \eqref{eq:26}) and adding the contributions of all the elements in the local equations \eqref{eq:25b} and \eqref{eq:25c}. The $\mathcal T_h$-subscripted bracket will emphasize the fact that differential operators are applied element by element.

\begin{framed}
\noindent
We look for
\begin{subequations}\label{eq:RTeq}
\begin{equation}
(\mathbf q_h,u_h,\widehat u_h)\in \mathbf V_h \times W_h \times M_h
\end{equation}
satisfying
\begin{alignat}{6}
\label{eq:RTeqb}
&(\kappa^{-1}\mathbf q_h,\mathbf r)_{\mathcal T_h}-(u_h,\mathrm{div}\,\mathbf r)_{\mathcal T_h}+\langle \widehat u_h,\mathbf r\cdot\mathbf n\rangle_{\partial\mathcal T_h} &&=0 &\qquad &\forall\mathbf r \in \mathbf V_h,\\
\label{eq:RTeqc}
&(\mathrm{div}\,\mathbf q_h,w)_{\mathcal T_h}& &=(f,w)_{\mathcal T_h} & &\forall w\in W_h,\\
\label{eq:RTeqd}
&\langle\mathbf q_h\cdot\mathbf n,\mu\rangle_{\partial\mathcal T_h\setminus\Gamma} &&=0 & &\forall\mu \in M_h^\circ,\\
\label{eq:RTeqe}
&\langle \widehat u_h,\mu\rangle_\Gamma &&=\langle u_0,\mu\rangle_\Gamma & &\forall \mu \in M_h^\Gamma.
\end{alignat}
\end{subequations}
\end{framed}

\begin{proposition}[Unique solvability]\label{prop:3.1}\
\begin{itemize}
\item[{\rm (a)}] Equations \eqref{eq:RTeq} are uniquely solvable.
\item[{\rm (b)}] The solution of \eqref{eq:RTeq} solves \eqref{eq:24}.
\item[{\rm (c)}] A solution of \eqref{eq:24} can be added a field $\widehat u_h\in M_h$ to be a solution of \eqref{eq:RTeq}. 
\end{itemize}
\end{proposition}

\begin{proof}
Since $M_h\equiv M_h^\circ\oplus M_h^\Gamma$, existence of solution of \eqref{eq:RTeq} follows from uniqueness. Let then $(\mathbf q_h,u_h,\widehat u_h)\in \mathbf V_h\times W_h\times M_h$ be a solution of 
\begin{subequations}
\begin{alignat}{6}
\label{eq:31a}
&(\kappa^{-1}\mathbf q_h,\mathbf r)_{\mathcal T_h}-(u_h,\mathrm{div}\,\mathbf r)_{\mathcal T_h}+\langle \widehat u_h,\mathbf r\cdot\mathbf n\rangle_{\partial\mathcal T_h} &&=0 &\qquad &\forall\mathbf r \in \mathbf V_h,\\
\label{eq:31b}
&(\mathrm{div}\,\mathbf q_h,w)_{\mathcal T_h}& &=0 & &\forall w\in W_h,\\
\label{eq:31c}
&\langle\mathbf q_h\cdot\mathbf n,\mu\rangle_{\partial\mathcal T_h\setminus\Gamma} &&=0 & &\forall\mu \in M_h^\circ,\\
\label{eq:31d}
&\langle \widehat u_h,\mu\rangle_\Gamma &&=0 & &\forall \mu \in M_h^\Gamma.
\end{alignat}
\end{subequations}
Testing these equations with $(\mathbf q_h,u_h,-\widehat u_h,-\mathbf q_h\cdot\mathbf n)$ and adding the results, we show that $(\kappa^{-1}\mathbf q_h,\mathbf q_h)_{\mathcal T_h}=0$ and hence $\mathbf q_h=\mathbf 0$. Let us now go back to \eqref{eq:31a}, which after integration by parts and localization on a single element yields for all $K\in \mathcal T_h$
\begin{equation}\label{eq:32}
(\nabla u_h,\mathbf r)_K+\langle u_h-\widehat u_h,\mathbf r\cdot\mathbf n\rangle_{\partial K}=0\qquad \forall \mathbf r\in \mathcal {RT}_k(K).
\end{equation}
We now construct $\mathbf p\in \mathcal{RT}_k(K)$ satisfying
\begin{subequations}
\begin{alignat}{4}
(\mathbf p,\mathbf r)_K&=0 &\qquad &\forall \mathbf r\in \boldsymbol{\mathcal P}_{k-1}(K),\\
\label{eq:33b}
\langle \mathbf p\cdot\mathbf n,\mu\rangle_{\partial K}&=\langle u_h-\widehat u_h,\mu\rangle_{\partial K} &&\forall \mu \in \mathcal R_k(\partial K).
\end{alignat}
\end{subequations}
(Note that these are the same equations that define the RT projection \eqref{eq:RT} and the local RT lifting of Section \ref{sec:2.3}.) Using this function as the test function in \eqref{eq:32}, we prove that
\[
0 =(\nabla u_h,\mathbf p)_K +\langle u_h-\widehat u_h,\mathbf p\cdot\mathbf n\rangle_{\partial K}=\langle u_h-\widehat u_h,u_h-\widehat u_h\rangle_{\partial K},
\]
where we have used $\mu=u_h-\widehat u_h\in \mathcal R_k(\partial K)$ in \eqref{eq:33b}. Hence $u_h-\widehat u_h=0$ on $\partial K$ and \eqref{eq:32} shows then (take $\mathbf r=\nabla u_h$) that $u_h\equiv c_K$ in $K$ and $\widehat u_h=u_h\equiv c_K$ on $\partial K$. Since each interior face value of $\widehat u_h$ is reached from different elements, it is easy to see that we have proved that $u_h\equiv c$ and $\widehat u_h\equiv c$. However, equation \eqref{eq:31d} implies that $\widehat u_h=0$ on $\Gamma$, and the proof of uniqueness of solution of \eqref{eq:RTeq} is thus finished.

To prove (b), note first that \eqref{eq:RTeqd} implies that $\mathbf q_h\in \mathbf V_h^{\mathrm{div}}$. On the other hand, if we test equation \eqref{eq:RTeqb} with $\mathbf r \in \mathbf V_h^{\mathrm{div}}\subset \mathbf V_h$, it follows that
\begin{alignat*}{4}
0 = & (\kappa^{-1}\mathbf q_h,\mathbf r)_{\mathcal T_h}-(u_h,\mathrm{div}\,\mathbf r)_{\mathcal T_h} +\langle \widehat u_h,\mathbf r\cdot\mathbf n\rangle_{\partial \mathcal T_h} \\
&(\kappa^{-1}\mathbf q_h,\mathbf r)_\Omega-(u_h,\mathrm{div}\,\mathbf r)_\Omega +\langle \widehat u_h,\mathbf r\cdot\mathbf n\rangle_\Gamma &\qquad &\mbox{($\mathbf r\in \mathbf V_h^{\mathrm{div}}$)}\\
&(\kappa^{-1}\mathbf q_h,\mathbf r)_\Omega-(u_h,\mathrm{div}\,\mathbf r)_\Omega +\langle u_0,\mathbf r\cdot\mathbf n\rangle_\Gamma. &&\mbox{(by \eqref{eq:RTeqe})}
\end{alignat*}
This easily shows that any solution of \eqref{eq:RTeq} solves the traditional RT equations \eqref{eq:24}.

Let now $(\mathbf q_h,u_h)$ solve \eqref{eq:24}. It is clear that equations \eqref{eq:RTeqc} and \eqref{eq:RTeqd} are satisfied. We now look for $\widehat u_h \in M_h$ such that
\begin{equation}\label{eq:E2}
\langle \widehat u_h,\mathbf r\cdot\mathbf n\rangle_{\partial \mathcal T_h}=-
 (\kappa^{-1}\mathbf q_h,\mathbf r)_{\partial\mathcal T_h}+(u_h,\mathrm{div}\,\mathbf r)_{\mathcal T_h}\qquad \forall \mathbf r\in \mathbf V_h.
\end{equation}
The argument to show that \eqref{eq:E2} has a unique solution is simple. Uniqueness follows from the fact that if $\mu \in M_h$, there exists $\mathbf r\in \mathbf V_h$ such that $\langle\mu,\mathbf r\cdot\mathbf n\rangle_{\partial\mathcal T_h}=\langle\mu,\mu\rangle_{\partial\mathcal T_h}$ (this is done by using local liftings of the normal trace). To prove existence of solution, note that if $\mathbf r\in \mathbf V_h$ is such that $\langle\mu,\mathbf r\cdot\mathbf n\rangle_{\partial \mathcal T_h}=0$ for all $\mu \in M_h$, then $\mathbf r\in \mathbf V_h^{\mathrm{div}}$ and $\mathbf r\cdot\mathbf n=0$ on $\Gamma$, but in that case, by \eqref{eq:24b}, the right hand side of \eqref{eq:E2} vanishes. This means that the right-hand side is orthogonal to the kernel of the transpose system. Then, by construction, \eqref{eq:RTeqb} is satisfied. Finally, if $\mu \in M_h^\Gamma$, we can easily find $\mathbf r\in \mathbf V_h^{\mathrm{div}}$ such that $\mathbf r\cdot\mathbf n=\mu$ on $\Gamma$. Then
\begin{alignat*}{4}
\langle \widehat u_h,\mu\rangle_\Gamma &=\langle\widehat u_h,\mathbf r\cdot\mathbf n\rangle_\Gamma &\qquad &\mbox{(by construction)}\\
&=\langle\widehat u_h,\mathbf r\cdot\mathbf n\rangle_{\partial\mathcal T_h} && (\mathbf r\in \mathbf V_h^{\mathrm{div}})\\
&=-(\kappa^{-1}\mathbf q_h,\mathbf r)_{\mathcal T_h}+(u_h,\mathrm{div}\,\mathbf r)_{\mathcal T_h} &&\mbox{(by \eqref{eq:E2})}\\
&=\langle u_0,\mathbf r\cdot\mathbf n\rangle_\Gamma, && \mbox{(by \eqref{eq:24b}, since $\mathbf r\in \mathbf V_h^{\mathrm{div}}$)}
\end{alignat*}
which is the missing equation in the decoupled formulation \eqref{eq:RTeq}.
\end{proof}

\subsection{Energy estimates}\label{sec:3.2}

\paragraph{The error equations.}
Let us first recall the RT equations
\begin{subequations}\label{eq:34}
\begin{alignat}{6}
\label{eq:34a}
&(\kappa^{-1}\mathbf q_h,\mathbf r)_{\mathcal T_h}-(u_h,\mathrm{div}\,\mathbf r)_{\mathcal T_h}+\langle \widehat u_h,\mathbf r\cdot\mathbf n\rangle_{\partial\mathcal T_h} &&=0 &\qquad &\forall\mathbf r \in \mathbf V_h,\\
\label{eq:34b}
&(\mathrm{div}\,\mathbf q_h,w)_{\mathcal T_h}& &=(f,w)_{\mathcal T_h} & &\forall w\in W_h,\\
&\langle\mathbf q_h\cdot\mathbf n,\mu\rangle_{\partial\mathcal T_h\setminus\Gamma} &&=0 & &\forall\mu \in M_h^\circ,\\
&\langle \widehat u_h,\mu\rangle_\Gamma &&=\langle u_0,\mu\rangle_\Gamma & &\forall \mu \in M_h^\Gamma,
\end{alignat}
\end{subequations}
and let us note that these equations correspond to a consistent method:
\begin{subequations}\label{eq:35}
\begin{alignat}{6}
&(\kappa^{-1}\mathbf q,\mathbf r)_{\mathcal T_h}-(u,\mathrm{div}\,\mathbf r)_{\mathcal T_h}+\langle  u,\mathbf r\cdot\mathbf n\rangle_{\partial\mathcal T_h} &&=0 &\qquad &\forall\mathbf r \in \mathbf V_h,\\
&(\mathrm{div}\,\mathbf q,w)_{\mathcal T_h}& &=(f,w)_{\mathcal T_h} & &\forall w\in W_h,\\
&\langle\mathbf q\cdot\mathbf n,\mu\rangle_{\partial\mathcal T_h\setminus\Gamma} &&=0 & &\forall\mu \in M_h^\circ,\\
&\langle  u,\mu\rangle_\Gamma &&=\langle u_0,\mu\rangle_\Gamma & &\forall \mu \in M_h^\Gamma.
\end{alignat}
\end{subequations}
We then consider the projections $(\boldsymbol\Pi\mathbf q,\Pi u,\mathrm P u)\in \mathbf V_h\times W_h\times M_h$ 
defined by $\boldsymbol\Pi \mathbf q|_K:=\boldsymbol\Pi^{\mathrm{RT}}\mathbf q$, $\Pi u|_K:=\Pi_k u$ and 
\[
\langle \mathrm Pu,\mu\rangle_e=\langle u,\mu\rangle_e\qquad \forall \mu \in \mathcal P_k(e)\quad \forall e \in \mathcal E_h.
\]
Next, we substitute these projections in as many instances of \eqref{eq:35} as possible:
\begin{subequations}\label{eq:36}
\begin{alignat}{6}
\label{eq:36a}
&(\kappa^{-1}\mathbf q,\mathbf r)_{\mathcal T_h}-(\Pi  u,\mathrm{div}\,\mathbf r)_{\mathcal T_h}+\langle \mathrm P  u,\mathbf r\cdot\mathbf n\rangle_{\partial\mathcal T_h} &&=0 &\qquad &\forall\mathbf r \in \mathbf V_h,\\
\label{eq:36b}
&(\mathrm{div}\,\boldsymbol\Pi\mathbf q,w)_{\mathcal T_h}& &=(f,w)_{\mathcal T_h} & &\forall w\in W_h,\\
&\langle\boldsymbol\Pi \mathbf q\cdot\mathbf n,\mu\rangle_{\partial\mathcal T_h\setminus\Gamma} &&=0 & &\forall\mu \in M_h^\circ,\\
&\langle  \mathrm P u,\mu\rangle_\Gamma &&=\langle u_0,\mu\rangle_\Gamma & &\forall \mu \in M_h^\Gamma.
\end{alignat}
\end{subequations}
(Note that we have used the commutativity property \eqref{eq:18} of the RT projection.)
We then think in terms of the following quantities:
\begin{equation}\label{eq:37}
\boldsymbol\varepsilon_h^q:=\boldsymbol\Pi \mathbf q-\mathbf q_h\in \mathbf V_h, \qquad \varepsilon_h^u:=\Pi u-u_h\in W_h,\qquad \widehat\varepsilon_h^u:=\mathrm Pu-\widehat u_h\in M_h.
\end{equation}
Subtracting the discrete equations \eqref{eq:34} from \eqref{eq:36}, we get to the {\bf error equations}
\begin{subequations}\label{eq:38}
\begin{alignat}{6}
\label{eq:38a}
&(\kappa^{-1}\boldsymbol\varepsilon_h^q,\mathbf r)_{\mathcal T_h}-(\varepsilon_h^u,\mathrm{div}\,\mathbf r)_{\mathcal T_h}+\langle \widehat\varepsilon_h^u,\mathbf r\cdot\mathbf n\rangle_{\partial\mathcal T_h} &&=(\kappa^{-1}(\boldsymbol\Pi\mathbf q-\mathbf q),\mathbf r)_{\mathcal T_h} &\qquad &\forall\mathbf r \in \mathbf V_h,\\
\label{eq:38b}
&(\mathrm{div}\,\boldsymbol\varepsilon_h^q,w)_{\mathcal T_h}& &=0 & &\forall w\in W_h,\\
&\langle\boldsymbol\varepsilon_h^q\cdot\mathbf n,\mu\rangle_{\partial\mathcal T_h\setminus\Gamma} &&=0 & &\forall\mu \in M_h^\circ,\\
\label{eq:38d}
&\langle  \widehat\varepsilon_h^u,\mu\rangle_\Gamma &&=0 & &\forall \mu \in M_h^\Gamma.
\end{alignat}
\end{subequations}
Testing now equations \eqref{eq:38} with $(\boldsymbol\varepsilon_h^q,\varepsilon_h^u,-\widehat\varepsilon_h^u,-\boldsymbol\varepsilon_h^q\cdot\mathbf n)$, and adding the results, we obtain the {\bf energy identity}
\begin{equation}\label{eq:39}
(\kappa^{-1} \boldsymbol\varepsilon_h^q,\boldsymbol\varepsilon_h^q)_{\mathcal T_h}=(\kappa^{-1}(\boldsymbol\Pi\mathbf q-\mathbf q),\boldsymbol\varepsilon_h^q)_{\mathcal T_h}.
\end{equation}
The Cauchy-Schwarz inequality w.r.t. the following norm
\[
\| \mathbf q\|_{\kappa^{-1}}=\|\kappa^{-1/2} \mathbf q\|_\Omega=(\kappa^{-1}\mathbf q,\mathbf q)_\Omega^{1/2},
\]
provides our first convergence estimate
\begin{framed}
\begin{equation}\label{eq:40}
\|\boldsymbol\Pi\mathbf q-\mathbf q_h\|_{\kappa^{-1}}=\|\boldsymbol\varepsilon_h^q\|_{\kappa^{-1}}\le \| \boldsymbol\Pi \mathbf q-\mathbf q\|_{\kappa^{-1}}.
\end{equation}
\end{framed}

\paragraph{On the decoupling in energy estimates.} The estimate \eqref{eq:40} is decoupled: convergence properties for $\mathbf q$ depend on approximation properties provided by the space $\mathbf V_h$, but not on those of the space $W_h$. This is not the case for equations of the form
\[
\mathrm{div}\,\mathbf q+c\, u=f,\qquad \mbox{where}\, c\ge 0.
\]
We will study this effect in Section \ref{sec:4.5}.

\paragraph{A flux estimate.} Let $\mathbf p\in \boldsymbol{\mathcal P}_{k+1}(K)$. Then
\begin{alignat*}{4}
h_K \langle \mathbf p\cdot\mathbf n,\mathbf p\cdot\mathbf n\rangle_{\partial K}&=h_K \langle \widecheck{\mathbf p\cdot\mathbf n},\widehat{\mathbf p\cdot\mathbf n}\rangle_{\partial\widehat K} &\qquad &\mbox{(by \eqref{eq:0c})} \\
&=h_K \langle \widehat{\mathbf p\cdot\mathbf n}, |a_K|^{-1} \widehat{\mathbf p\cdot\mathbf n}\rangle_{\partial\widehat K}\\
&=h_K \langle \widehat{\mathbf p}\cdot\widehat{\mathbf n},|a_K|^{-1},\widehat{\mathbf p}\cdot\widehat{\mathbf n}\rangle_{\partial\widehat K} &&\mbox{(by \eqref{eq:1c})}\\
&\lesssim h_K^{2-d} \|\widehat{\mathbf p}\cdot\widehat{\mathbf n}\|_{\partial\widehat K}^2 &&\mbox{(by \eqref{eq:5})}\\
&\lesssim h_K^{2-d} \| \widehat{\mathbf p}\|_{\widehat K}^2 &&\mbox{(finite dimension)}\\
&\approx \|\mathbf p\|_K^2. &&\mbox{(by \eqref{eq:6})}
\end{alignat*}
If we then use the norm
\[
\|\mu\|_h:=\bigg( \sum_{K\in \mathcal T_h} h_K \| \mu\|_{\partial K}^2\bigg)^{1/2} \approx \bigg( \sum_{e\in \mathcal E_h} h_e \|\mu\|_e^2\bigg)^{1/2},
\]
we have a bound 
\begin{framed}
\begin{equation}\label{eq:HH1}
\|\boldsymbol\varepsilon_h^q\cdot\mathbf n\|_h\lesssim \|\boldsymbol\varepsilon_h^q\|_\Omega.
\end{equation}
\end{framed}

\subsection{More convergence estimates}\label{sec:3.3}

\paragraph{Lifting $\varepsilon_h^u$.} The key to an error analysis of $u_h$ is a lifting to $\varepsilon_h^u$ to be the divergence of a continuous vector field. Consider then an operator $\mathbf L:L^2(\Omega)\to \mathbf H^1(\Omega)$ 
\[
 \mathrm{div}\,\mathbf Lv=v, \qquad \|\mathbf L\,v\|_{1,\Omega}\lesssim  \| v\|_\Omega.
\]
This can be done by solving a Stokes-like problem in $\Omega$ or by using an extension operator and inverting the divergence operator in free space. Let  then $\boldsymbol\xi:=\mathbf L \varepsilon_h^u$, so that
\begin{equation}\label{eq:42}
\mathrm{div}\,\boldsymbol\xi=\varepsilon_h^u, \qquad \|\boldsymbol\xi\|_{1,\Omega}\lesssim \|\varepsilon_h^u\|_\Omega.
\end{equation}
The error analysis is then based on using $\boldsymbol\Pi\boldsymbol\xi$ as test function in \eqref{eq:36a} and \eqref{eq:34a} and subtracting the result:
\begin{equation}\label{eq:43}
(\kappa^{-1}(\mathbf q_h-\mathbf q),\boldsymbol\Pi\boldsymbol\xi)_{\mathcal T_h}-(\varepsilon_h^u,\mathrm{div}\,\boldsymbol\Pi\boldsymbol\xi)_{\mathcal T_h}+\langle \widehat\varepsilon_h^u,(\boldsymbol\Pi\boldsymbol\xi)\cdot\mathbf n\rangle_{\partial\mathcal T_h}=0.
\end{equation}
We then go ahead and study what is in \eqref{eq:43}. First of all
\begin{alignat*}{4}
\langle \widehat\varepsilon_h^u,(\boldsymbol\Pi\boldsymbol\xi)\cdot\mathbf n\rangle_{\partial\mathcal T_h}&=\langle \widehat\varepsilon_h^u,\boldsymbol\xi\cdot\mathbf n\rangle_{\partial\mathcal T_h} &\qquad & \mbox{(definition of the RT projection)}\\
&=\langle \widehat\varepsilon_h^u,\boldsymbol\xi\cdot\mathbf n\rangle_{\partial\mathcal T_h\setminus\Gamma} & &\mbox{($\widehat\varepsilon_h^u=0$ on $\Gamma$ by \eqref{eq:38d})}\\
&=0. & &\mbox{($\boldsymbol\xi\cdot\mathbf n$ changes sign on internal faces)} 
\end{alignat*}
Then
\begin{alignat*}{4}
\| \varepsilon_h^u\|_\Omega^2 &=(\varepsilon_h^u,\mathrm{div}\,\boldsymbol\xi)_{\mathcal T_h} &\qquad &\mbox{(by \eqref{eq:42})}\\
&=(\varepsilon_h^u,\Pi\,\mathrm{div}\,\boldsymbol\xi)_{\mathcal T_h} &\qquad &\mbox{($\varepsilon_h^u\in W_h$)}\\
&=(\varepsilon_h^u,\mathrm{div}\,\boldsymbol\Pi\boldsymbol\xi)_{\mathcal T_h} &&\mbox{(commutativity property \eqref{eq:commRT})}\\
&=(\kappa^{-1}(\mathbf q_h-\mathbf q),\boldsymbol\Pi \boldsymbol\xi)_{\mathcal T_h} & & \mbox{(by the error equation \eqref{eq:43})}\\
& \le \|\mathbf q-\mathbf q_h\|_{\kappa^{-1}} \|\kappa^{-1/2}\|_{L^\infty}\|\boldsymbol\Pi\boldsymbol\xi\|_{\Omega}\\
& \lesssim  \|\kappa^{-1/2}\|_{L^\infty} \|\mathbf q-\mathbf q_h\|_{\kappa^{-1}}\|\boldsymbol\xi\|_{1,\Omega} & &\mbox{(by Proposition \ref{prop:2.4})}\\
&\lesssim \|\mathbf q-\mathbf q_h\|_{\kappa^{-1}}\| \varepsilon_h^u\|_\Omega. & &\mbox{(by \eqref{eq:42})}
\end{alignat*}
This takes us to our second error estimate
\begin{framed}
\begin{equation}\label{eq:44}
\| \Pi u-u_h\|_\Omega=\|\varepsilon_h^u\|_\Omega \lesssim \| \mathbf q-\mathbf q_h\|_{\kappa^{-1}}.
\end{equation}
\end{framed}

\paragraph{Lifting $\widehat\varepsilon_h^u$ locally.} The error analysis for $\widehat u_h$ is carried out on an element-by-element basis using the lifting operator of Proposition \ref{prop:2.5}. Let then $\mathbf r:=\mathbf L^{\mathrm{RT}}\widehat\varepsilon_h^u|_{\partial K}\in \mathcal{RT}_k(K)$, so that
\begin{equation}\label{eq:45}
\mathbf r\cdot\mathbf n=\widehat\varepsilon_h^u, \qquad \|\mathbf r\|_K \lesssim  h_K^{1/2}\|\widehat\varepsilon_h^u\|_{\partial K}.
\end{equation}
Also, using a scaling argument (and the fact that $\mathbf r$ is in a polynomial space), we show that
\begin{equation}\label{eq:46}
h_K |\mathbf r|_{1,K} \lesssim \|\mathbf r\|_K\lesssim  h_K^{1/2}\|\widehat\varepsilon_h^u\|_{\partial K}.
\end{equation}
Then
\begin{alignat*}{4}
\| \widehat\varepsilon_h^u\|_{\partial K}^2 &=\langle \widehat\varepsilon_h^u,\mathbf r\cdot\mathbf n\rangle_{\partial K} &\qquad &\mbox{(by \eqref{eq:45})}\\
&=(\varepsilon_h^u,\mathrm{div}\,\mathbf r)_K-(\kappa^{-1} (\mathbf q-\mathbf q_h),\mathbf r)_K & & \mbox{(by the error eqn \eqref{eq:38a})}\\
& \le \| \varepsilon_h^u\|_K \|\mathrm{div}\,\mathbf r\|_K + \|\kappa^{-1/2} (\mathbf q-\mathbf q_h)\|_K \| \kappa^{-1/2} \mathbf r\|_K\\
&\lesssim  h_K^{-1} \bigg(\|\varepsilon_h^u\|_K + h_K   \|\kappa^{-1/2} (\mathbf q-\mathbf q_h)\|_K \bigg)\, h_K^{1/2}\|\widehat\varepsilon_h^u\|_{\partial K}. &&\mbox{(by \eqref{eq:45} and \eqref{eq:46})}
\end{alignat*}
A complete estimate can now be proved by adding the previous inequalities over all triangles. 
We have then essentially proved that
\begin{framed}
\begin{equation}\label{eq:47}
\|\mathrm P u-\widehat u_h\|_h=
\|\widehat\varepsilon_h^u\|_h \lesssim \|\varepsilon_h^u\|_\Omega + h \|\mathbf q-\mathbf q_h\|_{\kappa^{-1}}.
\end{equation}
\end{framed}

\subsection{Superconvergence estimates by duality}\label{sec:3.4}

\paragraph{Another inverse of the divergence.} A superconvergence analysis for $u_h$ can be carried out by using a more demanding form of writing $\mathrm{div}\,\boldsymbol\xi=\varepsilon_h^u$ than the one used in
\eqref{eq:42}. In particular, we will be using also the second of the error equations \eqref{eq:38} in the arguments that follow. We start by considering a dual problem 
\begin{subequations}\label{eq:48}
\begin{alignat}{4}
\label{eq:48a}
\kappa^{-1} \boldsymbol\xi -\nabla \theta &=0 &\qquad &\mbox{in $\Omega$},\\
\mathrm{div}\,\boldsymbol\xi &=\varepsilon_h^u &&\mbox{in $\Omega$},\\
\label{eq:48c}
\theta &=0 & &\mbox{on $\Gamma$}.
\end{alignat}
\end{subequations}
We assume the following {\bf regularity hypothesis}: there exists $C_{\mathrm{reg}}>0$ such that
\begin{equation}\label{eq:49}
\| \boldsymbol\xi\|_{1,\Omega}+ \|\theta\|_{2,\Omega}\le C_{\mathrm{reg}} \|\varepsilon_h^u\|_\Omega.
\end{equation}
This estimate holds for convex domains with smooth diffusion coefficient $\kappa$.

\paragraph{The duality estimate.} The beginning of the argument can be copied verbatim from what we did in Section \ref{sec:3.3}:
\begin{alignat*}{4}
\|\varepsilon_h^u\|_\Omega^2 &=(\kappa^{-1} (\mathbf q_h-\mathbf q),\boldsymbol\Pi\boldsymbol\xi)_{\mathcal T_h}\\
&=(\kappa^{-1} (\mathbf q_h-\mathbf q),\boldsymbol\Pi\boldsymbol\xi-\boldsymbol\xi)_{\mathcal T_h}+(\mathbf q_h-\mathbf q,\kappa^{-1}\boldsymbol\xi)_{\mathcal T_h}\\
&=(\kappa^{-1} (\mathbf q_h-\mathbf q),\boldsymbol\Pi\boldsymbol\xi-\boldsymbol\xi)_{\mathcal T_h}+(\mathbf q_h-\mathbf q,\nabla \theta)_{\mathcal T_h}. &\qquad &\mbox{(by \eqref{eq:48a})}
\end{alignat*}
The next part of the argument consists of working on the rightmost term in the previous inequality. Then
\begin{alignat*}{4}
(\mathbf q_h-\mathbf q,\nabla \theta)_{\mathcal T_h} &=-(\mathrm{div}\,(\mathbf q_h-\mathbf q),\theta)_{\mathcal T_h}+\langle (\mathbf q_h-\mathbf q)\cdot\mathbf n,\theta\rangle_{\partial\mathcal T_h\setminus\Gamma}&\qquad &\mbox{($\theta=0$ on $\Gamma$)}\\
&=(\mathrm{div}\,(\mathbf q-\mathbf q_h),\theta)_{\mathcal T_h} & &\mbox{(single-valued on $\partial\mathcal T_h$)}\\
&=(f-\Pi f,\theta)_\Omega & & \mbox{($\mathrm{div}\mathbf q_h=\Pi f$ is \eqref{eq:34b})}\\
&=(f-\Pi f,\theta-\Pi \theta)_\Omega.
\end{alignat*}
We end up by putting everything together and using estimates of the projections
\begin{alignat*}{4}
\|\varepsilon_h^u\|_\Omega^2 &
=(\kappa^{-1} (\mathbf q_h-\mathbf q),\boldsymbol\Pi\boldsymbol\xi-\boldsymbol\xi)_{\mathcal T_h}+(\mathbf q_h-\mathbf q,\nabla \theta)_{\mathcal T_h}\\
&=(\kappa^{-1} (\mathbf q_h-\mathbf q),\boldsymbol\Pi\boldsymbol\xi-\boldsymbol\xi)_{\mathcal T_h}+(f-\Pi f,\theta-\Pi \theta)_\Omega\\
&\lesssim h \|\mathbf q_h-\mathbf q\|_{\kappa^{-1}} |\boldsymbol\xi|_{1,\Omega}+ h \| f-\Pi f\|_\Omega |\theta|_{1,\Omega}\\
& \lesssim h\, (  \|\mathbf q_h-\mathbf q\|_{\kappa^{-1}}  + \| f-\Pi f\|_\Omega)\|\varepsilon_h^u\|_\Omega. &&\mbox{(reg. hypothesis \eqref{eq:49})}
\end{alignat*}
(The argument uses that $\|\mathbf p-\boldsymbol\Pi^{\mathrm{RT}}\mathbf p\|_K \lesssim h_K |\mathbf p|_{1,K}$. This can be easily proved using the same arguments as in Proposition \ref{prop:2.4}(b).)
We have thus proved that, under the regularity hypothesis \eqref{eq:49},
\begin{framed}
\begin{equation}\label{eq:50}
\|\varepsilon_h^u\|_\Omega \lesssim h\, \big(  \|\mathbf q_h-\mathbf q\|_{\kappa^{-1}}  + \| f-\Pi f\|_\Omega\big).
\end{equation}
\end{framed}
This bound can then be used in the right-hand side of \eqref{eq:47} to show that
\[
\|\widehat\varepsilon_h^u\|_h \lesssim  h\, \big(  \|\mathbf q_h-\mathbf q\|_{\kappa^{-1}}  + \| f-\Pi f\|_\Omega\big).
\]

\subsection{Summary of estimates}

\paragraph{Approximation properties.} Let start by recalling that
\[
\|\mathbf q-\boldsymbol\Pi\mathbf q\|_\Omega \lesssim h^{k+1} |\mathbf q|_{k+1,\Omega}\qquad\mbox{and}\qquad \| u-\Pi u\|_\Omega \lesssim h^{k+1} |u|_{k+1,\Omega}.
\]
Also
\begin{alignat*}{4}
h_K^{\frac12} \| u-\mathrm P u\|_{\partial K} &\lesssim h_K^{\frac{d}2}\|\widehat u-\widehat{\mathrm Pu}\|_{\partial\widehat K} &\qquad& \mbox{(change of variables \eqref{eq:6})}\\
&= h_K^{\frac{d}2} \|\widehat u-\widehat{\mathrm P}\widehat u\|_{\partial\widehat K} &&\mbox{(easy argument)}\\
& \le h_K^{\frac{d}2} \|\widehat u-\widehat\Pi_k \widehat u\|_{\partial\widehat K} && \mbox{($\widehat{\mathrm P}$ gives the best aprox)}\\
& \lesssim h_K^{\frac{d}2} \|\widehat u-\widehat\Pi_k \widehat u\|_{1,\widehat K} && \mbox{(trace theorem)}\\
& \lesssim h_K^{\frac{d}2} | \widehat u|_{k+1,\widehat K} &&\mbox{(compactness)}\\
& \lesssim h_K^{k+1} |u|_{k+1,K}, &&\mbox{(change of variables \eqref{eq:6})} 
\end{alignat*}
which can be collected in the estimate
\[
\| u-\mathrm P u\|_h \lesssim h^{k+1} | u|_{k+1,\Omega}.
\]
It is also easy to see that
\[
\|\mathbf q\cdot\mathbf n-\boldsymbol\Pi\mathbf q\cdot\mathbf n\|_h\lesssim h^{k+1}|\mathbf q|_{k+1,\Omega}.
\]
This is done element-by-element, face-by-face, using the fact that $\boldsymbol\Pi\mathbf q\cdot\mathbf n|_{\partial K}$ is the best approximation of $\mathbf q\cdot\mathbf n$ on $\mathcal R_k(\partial K)$ and, therefore, we can use the previous estimate applied to $u=\mathbf q\cdot\mathbf n_e$ for every $e\in \mathcal E(K)$.

\paragraph{Optimal convergence.} Assuming that everything is going the best way it can (solutions are smooth, the regularity hypotheses holds), we can summarize the convergence results in the following table:
\begin{framed}
\begin{alignat*}{7}
\|\mathbf q-\mathbf q_h\|_\Omega &\lesssim h^{k+1} & \qquad  \|\boldsymbol\Pi\mathbf q-\mathbf q_h\|_\Omega&\lesssim h^{k+1} &\qquad && \mbox{(see \eqref{eq:40})}\\
\| u-u_h\|_\Omega &\lesssim h^{k+1} &\| \Pi u-u_h\|_\Omega &\lesssim h^{k+2} &\qquad &&\mbox{(see \eqref{eq:50})}\\
\| u-\widehat u_h\|_h &\lesssim h^{k+1} & \|\mathrm P u-\widehat u_h\|_h &\lesssim h^{k+2} &&& \mbox{(see \eqref{eq:47} and \eqref{eq:50})}\\
\|\mathbf q\cdot\mathbf n-\mathbf q_h\cdot\mathbf n\|_h &\lesssim h^{k+1} \qquad &\|\boldsymbol\Pi\mathbf q\cdot\mathbf n-\mathbf q_h\cdot\mathbf n\|_h &\lesssim h^{k+1} &&&\mbox{(see \eqref{eq:HH1} and \eqref{eq:40})}
\end{alignat*}
\end{framed}

\section{Additional topics}\label{sec:4}

The following section explains some topics that are related to the RT method, or more especifically to the Arnold-Brezzi formulation of the RT method. These are general ideas that will apply with minimal changes to the other two methods (BDM and HDG) that we will introduce in these notes. For reasons of notation, we will write the diffusion problem as
\begin{subequations}\label{eq:A0}
\begin{alignat}{4}
\kappa^{-1}\mathbf q+\nabla u &=0 &\quad &\mbox{in $\Omega$},\\
\mathrm{div}\,\mathbf q &=f&\quad &\mbox{in $\Omega$},\\
u&=g &&\mbox{on $\Gamma$.}
\end{alignat}
\end{subequations}
The term hybridization makes reference to the not that popular hybrid methods, where the variational formulation is taken directly on the interfaces of the elements. (Yes, some of them can be understood as domain decomposition methods, and yes, the ultra-weak variational formulation UWVF is also related.) For more about hybrid methods --that we will not touch here--, the reader is referred to Brezzi and Fortin's book \cite{BrFo:1991}.

\subsection{Hybridization}

\paragraph{What is hybridization.} The goal of hybridization is the reduction of the system \eqref{eq:RTeq} to a linear system where only $\widehat u_h$ shows up. The remaining two variables will be reconstructed after solving for $\widehat u_h$, in an element-by-element fashion, easy to realize due to the fact that equations \eqref{eq:RTeqb} and \eqref{eq:RTeqc} are local or, in other words, the spaces $\mathbf V_h$ and $W_h$ are completely discontinuous. For some forthcoming arguments, it'll be practical to deal with the space
\[
B_h:=\prod_{K\in \mathcal T_h}\mathcal R_k(\partial K),
\]
and to note that $M_h$ is the subset of $B_h$ of functions that are single valued.

\paragraph{Flux due to sources.} Given $f:\Omega \to\mathbb R$, we look for
\begin{subequations}\label{eq:A1}
\begin{equation}
(\mathbf q_h^f,u_h^f)\in \mathbf V_h\times W_h,
\end{equation}
satisfying
\begin{alignat}{6}
& (\kappa^{-1}\mathbf q_h^f,\mathbf r)_{\mathcal T_h} -(u_h^f,\mathrm{div}\,\mathbf r)_{\mathcal T_h} &&=0 &\qquad &\forall\mathbf r \in \mathbf V_h,\\
& (\mathrm{div}\,\mathbf q_h^f,w)_{\mathcal T_h} &&=(f,w)_{\mathcal T_h} &&\forall w\in W_h.
\end{alignat}
\end{subequations}
(Existence and uniqueness of solution of \eqref{eq:A1} is straightforward to prove.) 
We then define
\begin{equation}\label{eq:A2}
\phi_h^f:=-\mathbf q_h^f\cdot\mathbf n \in B_h.
\end{equation}

\paragraph{Local solvers and flux operators.} Consider now the operator
\begin{subequations}
\begin{equation}\label{eq:A3a}
M_h \ni \widehat u_h \quad\longmapsto\quad (\mathrm L^q(\widehat u_h),\mathrm L^u(\widehat u_h))=(\mathbf q_h,u_h) \in \mathbf V_h\times W_h
\end{equation}
where
\begin{alignat}{6}
\label{eq:A3b}
& (\kappa^{-1}\mathbf q_h,\mathbf r)_{\mathcal T_h} -(u_h,\mathrm{div}\,\mathbf r)_{\mathcal T_h}+\langle\widehat u_h,\mathbf r\cdot\mathbf n\rangle_{\partial T_h} &&=0 &\qquad &\forall\mathbf r \in \mathbf V_h,\\
\label{eq:A3c}
& (\mathrm{div}\,\mathbf q_h,w)_{\mathcal T_h} &&=0 &&\forall w\in W_h.
\end{alignat}
We then consider the flux operator $\phi_h:M_h \to B_h$ given by
\begin{equation}
\phi_h(\widehat u_h):=-\mathbf q_h\cdot\mathbf n.
\end{equation}
\end{subequations}
Note that equations \eqref{eq:A3b}-\eqref{eq:A3c} are uniquely solvable and can be solved element by element.

\begin{framed}
\noindent{\bf The hybridized system.} We look for
\begin{subequations}
\begin{equation}
\widehat u_h \in M_h
\end{equation}
satisfying
\begin{alignat}{6}
\langle \phi_h(\widehat u_h)+\phi_h^f,\mu\rangle_{\partial\mathcal T_h\setminus\Gamma} &=0 &\qquad &\forall\mu \in M_h^\circ,\\
\langle \widehat u_h,\mu\rangle_\Gamma &=\langle g,\mu\rangle_\Gamma &&\forall \mu \in M_h^\Gamma.
\end{alignat}
We then define
\begin{equation}
\mathbf q_h=\mathrm L^q(\widehat u_h)+\mathbf q_h^f, \qquad u_h=\mathrm L^u(\widehat u_h)+u_h^f.
\end{equation}
\end{subequations}
\end{framed}

\noindent
Note that if we subtract 
\[
\widehat u_h^q\in M_h^\Gamma \qquad \mbox{satisfying} \qquad \langle \widehat u_h^g,\mu\rangle_\Gamma =\langle g,\mu\rangle_\Gamma \qquad\forall \mu \in M_h^\Gamma,
\]
then the hybridized system can be written as
\[
\widehat u_h^\circ\in M_h^\circ \qquad \mbox{such that}\qquad \langle \phi_h(\widehat u_h^\circ),\mu\rangle_{\partial\mathcal T_h\setminus\Gamma}=-\langle\phi_h^f+\phi_h(\widehat u_h^g),\mu\rangle_{\partial\mathcal T_h\setminus\Gamma}\qquad\forall\mu\in M_h^\circ.
\]

\paragraph{The hybridized bilinear form.} We next focus on the bilinear form
\begin{equation}\label{eq:A5}
M_h^\circ\times M_h^\circ \ni (\lambda,\mu)\quad \longmapsto\quad \langle\phi_h(\lambda),\mu\rangle_{\partial\mathcal T_h\setminus\Gamma}.
\end{equation}
Let then $(\mathbf q_h,u_h)=(\mathrm L^q(\lambda),\mathrm L^u(\lambda))$ and $(\mathbf v_h,v_h)=(\mathrm L^q(\mu),\mathrm L^u(\mu))$. Note that
\begin{alignat*}{6}
&(\kappa^{-1}\mathbf v_h,\mathbf r)_{\mathcal T_h}-(v_h,\mathrm{div}\,\mathbf r)_{\mathcal T_h}+\langle \mu,\mathbf r\cdot\mathbf n\rangle_{\partial\mathcal T_h} &&=0&\qquad &\forall\mathbf r\in \mathbf V_h,\\
&(\mathrm{div}\,\mathbf q_h,w)_{\mathcal T_h}&&=0 &&\forall w\in W_h,
\end{alignat*}
and therefore
\begin{alignat*}{6}
&(\kappa^{-1}\mathbf v_h,\mathbf q_h)_{\mathcal T_h}-(v_h,\mathrm{div}\,\mathbf q_h)_{\mathcal T_h}&&=-\langle \mu,\mathbf q_h\cdot\mathbf n\rangle_{\partial\mathcal T_h} \\
&(\mathrm{div}\,\mathbf q_h,v_h)_{\mathcal T_h}&&=0,
\end{alignat*}
which implies that
\begin{alignat*}{6}
\langle \phi_h(\lambda),\mu\rangle_{\partial\mathcal T_h\setminus\Gamma} &=\langle \phi_h(\lambda),\mu\rangle_{\partial\mathcal T_h} &\qquad &(\mu \in M_h^\circ),\\
&=-\langle \mathbf q_h\cdot\mathbf n,\mu\rangle_{\partial\mathcal T_h} &&\mbox{(definition of $\phi_h$)}\\
&=(\kappa^{-1}\mathbf v_h,\mathbf q_h)_{\mathcal T_h}.
\end{alignat*}
It is clear from this expression that the bilinear form is {\bf symmetric} and positive semidefinite. On the other hand, if $\lambda\in M_h^\circ$ and $\langle\phi_h(\lambda),\lambda\rangle_{\partial\mathcal T_h\setminus\Gamma}=0$, it is a simple exercise to observe that $(L^q(\lambda),L^u(\lambda),\lambda)$ is a solution of the discrete equations \eqref{eq:RTeq} with zero right-hand side and therefore has to vanish. This is proof of {\bf positive definiteness} of the bilinear form \eqref{eq:A5}.

\subsection{A discrete Dirichlet form}

\paragraph{Towards a primal form.} The goal of this section is the proof that the system \eqref{eq:RTeq} can be written in the variable $u_h$ only. This is not useful from the practical point of view, but helps in arguments related to RT discretization of evolutionary partial differential equations. 

\paragraph{Lifting of Dirichlet conditions.} Given $g:\Gamma\to\mathbb R$, we consider the pair
\begin{subequations}\label{eq:A6}
\begin{equation}
(\mathbf q_h^g,\widehat u_h^g)\in \mathbf V_h\times M_h
\end{equation}
satisfying
\begin{alignat}{6}
&(\kappa^{-1}\mathbf q_h^g,\mathbf r)_{\mathcal T_h}+\langle \widehat u_h^g,\mathbf r\cdot\mathbf n\rangle_{\partial\mathcal T_h} &&=0 &\qquad &\forall\mathbf r\in \mathbf V_h,\\
&\langle\mathbf q_h^g\cdot\mathbf n,\mu\rangle_{\partial\mathcal T_h\setminus\Gamma}&&=0&&\forall \mu\in M_h^\circ,\\
&\langle\widehat u_h^g,\mu\rangle_\Gamma &&=\langle g,\mu\rangle_\Gamma &&\forall \mu \in M_h^\Gamma.
\end{alignat}
(Existence and uniqueness of solutions to this problem is an easy exercise.) We then define
\begin{equation}
W_h\ni w \quad\longmapsto\quad \ell_g(w):=(\mathrm{div}\,\mathbf q_h^g,w)_{\mathcal T_h}.
\end{equation}
\end{subequations} 

\paragraph{The RT gradient.} We now consider the map
\begin{subequations}\label{eq:A7}
\begin{equation}
W_h\ni w_h \quad\longmapsto\quad (\mathrm G^q(u_h),\mathrm G^{\hat u}(u_h))=(\mathbf q_h,\widehat u_h)\in \mathbf V_h\times M_h,
\end{equation}
where
\begin{alignat}{6}
&(\kappa^{-1}\mathbf q_h,\mathbf r)_{\mathcal T_h}-(u_h,\mathrm{div}\,\mathbf r)_{\mathcal T_h}+\langle \widehat u_h,\mathbf r\cdot\mathbf n\rangle_{\partial\mathcal T_h} &&=0 &\qquad &\forall\mathbf r\in \mathbf V_h,\\
&\langle\mathbf q_h\cdot\mathbf n,\mu\rangle_{\partial\mathcal T_h\setminus\Gamma}&&=0&&\forall \mu\in M_h^\circ,\\
&\langle\widehat u_h,\mu\rangle_\Gamma &&=0 &&\forall \mu \in M_h^\Gamma.
\end{alignat}
We can thus think of the bilinear form (the discrete Dirichlet form)
\begin{equation}\label{eq:A7e}
(u_h,v_h)\ni W_h\times W_h \quad\longmapsto\quad D_h(u_h,w_h):=(\mathrm{div}\,\mathbf q_h,v_h)_{\mathcal T_h}=(\mathrm{div}\,\mathrm G^q(u_h),v_h)_{\mathcal T_h}.
\end{equation}
\end{subequations}
Note that $\mathrm G^q(u_h)$ is a minus gradient operator, instead of a gradient operator.

\paragraph{The primal form.} Given $f:\Omega\to\mathbb R$ and $g:\Gamma\to\mathbb R$, we look for
\[
u_h\in W_h \qquad\mbox{satisfying}\qquad D_h(u_h,w)=(f,w)_{\mathcal T_h}-\ell_g(w)\quad\forall w\in W_h.
\]
Then, 
\[
\mathbf q_h=\mathbf q_h^g+\mathrm G^q(u_h), \qquad \widehat u_h=\widehat u_h^g+\mathrm G^{\hat u}(u_h),
\]
and $u_h$ constitute the solution of \eqref{eq:RTeq}. It is not difficult to figure out that the primal form is just the Schur complement form of the traditional RT formulation \eqref{eq:24}.

\paragraph{Properties of the Dirichlet form.} Given $(u_h,v_h)\in W_h\times W_h$, we consider $(\mathbf q_h,\widehat u_h)=(\mathrm G^q(u_h),\mathrm G^{\hat u}(u_h))$ and $(\mathbf v_h,\widehat v_h)=(\mathrm G^q(v_h),\mathrm G^{\hat u}(v_h))$. Note that
\begin{alignat*}{4}
&(\kappa^{-1}\mathbf v_h,\mathbf r)_{\mathcal T_h}-(v_h,\mathrm{div}\,\mathbf r)_{\mathcal T_h}+\langle\widehat v_h,\mathbf r\cdot\mathbf n\rangle_{\partial\mathcal T_h\setminus\Gamma} &&=0&\qquad &\forall\mathbf r\in \mathbf V_h,\\
&\langle \mathbf q_h\cdot\mathbf n,\mu\rangle_{\partial\mathcal T_h\setminus\Gamma}&&=0&&\forall\mu\in M_h^\circ,
\end{alignat*}
and therefore
\begin{alignat*}{4}
&(\kappa^{-1}\mathbf v_h,\mathbf q_h)_{\mathcal T_h}-(v_h,\mathrm{div}\,\mathbf q_h)_{\mathcal T_h}+\langle\widehat v_h,\mathbf q_h\cdot\mathbf n\rangle_{\partial\mathcal T_h\setminus\Gamma} &&=0,&\qquad &\\
&\langle \mathbf q_h\cdot\mathbf n,\widehat v_h\rangle_{\partial\mathcal T_h\setminus\Gamma}&&=0.
\end{alignat*}
Then
\[
D_h(u_h,v_h) =(\mathrm{div}\,\mathbf q_h,v_h)_{\mathcal T_h}
=(\kappa^{-1}\mathbf v_h,\mathbf q_h)_{\mathcal T_h},
\]
which proves that the discrete Dirichlet form is {\bf symmetric} and positive semidefinite. Now, if $D(u_h,u_h)=0$, it is easy to see how $(\mathrm G^q(u_h),u_h,\mathrm G^{\hat u}(u_h))$ is a solution of \eqref{eq:RTeq} with zero right-hand side, and therefore it has to vanish, which proves that the discrete Dirichlet form is {\bf positive definite}.

\subsection{Stenberg postprocessing}

\paragraph{The local postprocessing step.} Assume that we have solved the RT equations \eqref{eq:RTeq}. We look for
\begin{subequations}\label{eq:A8}
\begin{equation}
u_h^\star\in \prod_{K\in \mathcal T_h} \mathcal P_{k+1}(K),
\end{equation}
satisfying for all $K\in\mathcal T_h$
\begin{alignat}{6}
(\kappa \nabla u_h^\star,\nabla v)_K &=(f,  v)_K-\langle\mathbf q_h\cdot\mathbf n,v\rangle_{\partial K} &\qquad & \forall v\in \mathcal P_{k+1}(K),\\
(u_h^\star,1)_K &=(u_h,1)_K. &&
\end{alignat}
\end{subequations}
This postprocessing method was first proposed by Rolf Stenberg in \cite{Stenberg:1991}. Note that a simple computation shows
\begin{alignat}{4}
\nonumber
(\kappa\nabla u_h^\star,\nabla v)_K &=(f,v)_K-\langle\mathbf q_h\cdot\mathbf n,v\rangle_{\partial K}\\
\nonumber
&=(\mathrm{div}\,\mathbf q,v)_K -\langle\mathbf q_h\cdot\mathbf n,v\rangle_{\partial K}\\
\nonumber
&=-(\mathbf q,\nabla v)_K+\langle\mathbf q\cdot\mathbf n-\mathbf q_h\cdot\mathbf n,v\rangle_{\partial K}\\
&=(\kappa \nabla u,\nabla v)_K+\langle\mathbf q\cdot\mathbf n-\mathbf q_h\cdot\mathbf n,v\rangle_{\partial K} \qquad\forall v\in \mathcal P_{k+1}(K).
\label{eq:A9o}
\end{alignat}

\paragraph{Some preliminary comments.} Before we start anayzing this, let us introduce the space
\[
\mathcal P_{k+1}^0(K):=\{ v\in \mathcal P_{k+1}(K)\,:\, (v,1)_K=0\},
\]
and note that
\begin{equation}\label{eq:A9}
(f,1)_K=(\mathrm{div}\,\mathbf q_h,1)_K=\langle\mathbf q_h\cdot\mathbf n,1\rangle_{\partial K},
\end{equation}
which means that we can decompose in an orthogonal sum
\begin{subequations}\label{eq:A10}
\begin{equation}
u_h^\star=c_h+\omega_h, \qquad c_h \in \prod_{K\in \mathcal T_h}\mathcal P_0(K), \qquad \omega_h\in \prod_{K\in \mathcal T_h}\mathcal P_{k+1}^0(K)
\end{equation}
and compute separately for all $K\in \mathcal T_h$
\begin{alignat}{4}
(\kappa \nabla w_h,\nabla v)_K&=(f,v)_K-\langle\mathbf q_h\cdot\mathbf n,v\rangle_{\partial K}&\qquad & \forall v\in \mathcal P_{k+1}^0(K),\\
(c_h,1)_K&=(u_h,1)_K 
\end{alignat}
\end{subequations}
It is clear that due to \eqref{eq:A9}, problems \eqref{eq:A10} and \eqref{eq:A8} are equivalent, while it is quite obvious that problem \eqref{eq:A10} has a unique solution.

\begin{lemma}\label{lemma:4.1}
The following inequalities hold
\[
\| v\|_{\partial K}\lesssim h_K^{1/2} | v|_{1,K}, \qquad \| v\|_{K}\approx h_K | v|_{1,K} \qquad \forall v\in \mathcal P_{k+1}^0(K)
\]
\end{lemma}

\begin{proof}
Both inequalities follow from scaling arguments, and the following facts:
\begin{equation}\label{eq:A11}
v\in \mathcal P_{k+1}^0(K) \qquad \Longleftrightarrow \qquad \widehat v\in \mathcal P_{k+1}^0(\widehat K),
\end{equation}
\begin{equation}\label{eq:A12}
\|v\|_{\partial\widehat K}\lesssim \| v\|_{\widehat K}\approx | v|_{1,\widehat K} \qquad \forall v \in \mathcal P_{k+1}^0(\widehat K).
\end{equation}
It is also important to keep in mind \eqref{eq:E1}, which says that the hat symbol is not ambiguous when applied in the interior domain or on the boundary.
Then, the scaling argument is reduced to noticing that for all $v\in \mathcal P_{k+1}^0(K)$,
\begin{alignat*}{4}
\|v \|_{\partial K} &\approx h_K^ {\frac{d-1}2}\|\widehat v\|_{\partial\widehat K} &\qquad&\mbox{(scaling \eqref{eq:6} and meaning of $\widehat v$)}\\
&\lesssim h_K^{\frac{d-1}2} | \widehat v|_{1,\widehat K} &&\mbox{(finite dimensional bound \eqref{eq:A12})}\\
&\approx h_K^{\frac12} |v|_{1,K}, &&\mbox{(scaling \eqref{eq:8})}
\end{alignat*}
and
\begin{alignat*}{4}
\| v\|_K &\approx h_K^{\frac{d}2} \|\widehat v\|_{\widehat K} &\qquad &\mbox{(scaling \eqref{eq:6})}\\
&\approx h_K^{\frac{d}2} |\widehat v|_{1,\widehat K} && \mbox{(finite dimensional bound \eqref{eq:A12})}\\
&\approx h_K |v|_{1,K}. && \mbox{(scaling \eqref{eq:8})}
\end{alignat*}
This finishes the proof.
\end{proof}

\begin{proposition}[Postprocessing]\label{prop:4.2}
Let $(u_h,\mathbf q_h)$ be any approximation of the solution of \eqref{eq:A0} satisfying $(f,1)_K=\langle\mathbf q_h\cdot\mathbf n,1\rangle_{\partial K}$. Then the Stenberg postprocessing \eqref{eq:A8} satisfies:
\begin{alignat*}{4}
\| u-u_h^\star\|_\Omega \lesssim & \| u-\Pi_{k+1} u\|_\Omega+\bigg(\sum_{K\in\mathcal T_h} h_K^2| u-\Pi_{k+1} u|_{1,K}^2\bigg)^{1/2}\\
&+\| u_h-\Pi_k u\|_\Omega+ h \|\mathbf q\cdot\mathbf n-\mathbf q_h\cdot\mathbf n\|_h.
\end{alignat*}
If the discrete conservation property $(f,1)_K=\langle\mathbf q_h\cdot\mathbf n,1\rangle_{\partial K}$ does not hold, then the same bound is satisfied by the solution of \eqref{eq:A10}. 
\end{proposition}

\begin{proof}
Let
\begin{alignat*}{4}
v &:= u_h^\star-\Pi_{k+1} u-\Pi_0(u_h^\star-\Pi_{k+1} u)\\
&=u_h^\star-\Pi_{k+1} u - \Pi_0 (u_h-\Pi_k u) &\qquad &\mbox{(by \eqref{eq:A8} and $\Pi_0\Pi_k=\Pi_0=\Pi_0\Pi_{k+1}$)}
\end{alignat*}
and note that $v|_K\in \mathcal P_{k+1}^0(K)$. Then
\begin{alignat*}{6}
\|\kappa^{1/2}\nabla v\|_K^2=&(\kappa \nabla (u_h^\star-\Pi_{k+1} u),\nabla v)_K &\qquad & (\nabla \Pi_0=0)\\
=&(\kappa \nabla (u-\Pi_{k+1} u),\nabla v)_L+\langle\mathbf q\cdot\mathbf n -\mathbf q_h\cdot\mathbf n,v\rangle_{\partial K}&&\mbox{(by \eqref{eq:A9o})}\\
\le  & |u-\Pi_{k+1} u|_{1,K}\|\kappa^{1/2}\nabla v\|_K\|\kappa^{1/2}\|_{L^\infty}\\
& + h_K^{1/2} \| \mathbf q\cdot\mathbf n-\mathbf q_h\cdot\mathbf n\|_{\partial K} h_K^{-1/2} \|v\|_{\partial K}\\
\lesssim & \big( |u-\Pi_{k+1} u|_{1,K}+h_K^{1/2} \| \mathbf q\cdot\mathbf n-\mathbf q_h\cdot\mathbf n\|_{\partial K}\big)\|\kappa^{1/2}\nabla v\|_K, &&\mbox{(by Lemma \ref{lemma:4.1})}
\end{alignat*}
or, in other words,
\begin{equation}\label{eq:A13}
|v|_{1,K}^2 \lesssim |u-\Pi_{k+1} u|_{1,K}^2+ h_K \|\mathbf q\cdot\mathbf n-\mathbf q_h\cdot\mathbf n\|_{\partial K}^2.
\end{equation}
Therefore
\begin{alignat*}{4}
\| u_h^\star-\Pi_{k+1} u\|_K^2 =&\|\Pi_0(u_h^\star-\Pi_{k+1} u)\|_K^2 +\| v\|_K^2 &\qquad &\mbox{(orthogonal decomp)}\\
=&\|\Pi_0(u_h-\Pi_k u)\|_K^2+\| v\|_K^2 && \mbox{(see definition of $v$)}\\
\lesssim &\| u_h-\Pi_k u\|_K^2 + h_K^2 | v|_{1,K}^2 && \mbox{(by Lemma \ref{lemma:4.1})}\\
\lesssim & \| u_h-\Pi_k u\|_K^2 \\
&+ h_K^2 | u-\Pi_{k+1}u|_{1,K}^2 + h_K^3 \|\mathbf q\cdot\mathbf n-\mathbf q_h\cdot\mathbf n\|_{\partial K}^2, &&\mbox{(by \eqref{eq:A13})}
\end{alignat*}
and to prove the result we only need to collect the contributions of all the elements.
\end{proof}

\noindent
For the RT discretization, assuming superconvergence, the Stenberg postprocessing \eqref{eq:A8} satisfies
\begin{framed}
\[
\| u-u_h^\star\|_\Omega \lesssim h^{k+2}.
\]
\end{framed}

\subsection{A second postprocessing scheme}

\paragraph{Another way of getting a good gradient.} Since $\nabla u=-\kappa^{-1}\mathbf q$, we can use the approximation $\mathbf q_h$ as a way of getting an improvec gradient, using $u_h$ to determine the average of the postprocessed on each element. We then look for
\begin{subequations}\label{eq:C1}
\begin{equation}
u_h^\star\in \prod_{K\in \mathcal T_h} \mathcal P_{k+1}(K),
\end{equation}
satisfying for all $K\in \mathcal T_h$
\begin{alignat}{4}
(\nabla u_h^\star,\nabla v)_K &=-(\kappa^{-1}\mathbf q_h,\nabla v)_K &\qquad &\forall v\in \mathcal P_{k+1}^0(K),\\
(u_h^\star,q)_K &=(u_h,1)_K.
\label{eq:C1c}
\end{alignat}
\end{subequations}

\paragraph{Its analysis.}
Note now that
\[
(\nabla u_h^\star,\nabla v)_K=(\nabla u,\nabla v)_K+(\kappa^{-1}(\mathbf q-\mathbf q_h),\nabla v)_K \qquad \forall v\in \mathcal P_{k+1}^0(K).
\]
Like in the proof of Proposition \ref{prop:4.2}, we consider
\begin{alignat*}{4}
\prod_{K\in \mathcal T_h}\mathcal P_{k+1}^0(K)\ni v &:= u_h^\star-\Pi_{k+1} u-\Pi_0(u_h^\star-\Pi_{k+1} u)\\
&=u_h^\star-\Pi_{k+1} u - \Pi_0 (u_h-\Pi_k u), &\qquad &\mbox{(by \eqref{eq:C1c} and $\Pi_0(\Pi_k-\Pi_{k+1})=0$)}
\end{alignat*}
and write
\begin{alignat*}{4}
|v|_{1,K}^2 &=(\nabla (u_h^\star-\Pi_{k+1} u),\nabla v)_K\\
&=(\nabla (u-\Pi_{k+1} u),\nabla v)_K+(\kappa^{-1} (\mathbf q-\mathbf q_h),\nabla v)_K,
\end{alignat*}
so that, using Lemma \ref{lemma:4.1}, we have bounded
\[
h_K^{-1}\|v\|_K\lesssim | v|_{1,K} \le |u-\Pi_{k+1} u|_{1,K}+\|\kappa^{-1}(\mathbf q-\mathbf q_h\|_K.
\]
What is left follows the final steps of the arguments in Proposition \ref{prop:4.2}, leading to
\[
\| u-u_h^\star\|_\Omega \lesssim \| u-\Pi_{k+1}u \|_\Omega + \|u_h-\Pi_k u\|_\Omega + \bigg( \sum_{K\in \mathcal T_h} h_K^2|u-\Pi_{k+1}u|_{1,K}^2\bigg) + h \|\mathbf q-\mathbf q_h\|_\Omega,
\]
and therefore to superconvergence. Once again, nothing particular about how $(\mathbf q_h,u_h)$ has been produced is used in this argument. However, to reach superconvergence, we obviously need that $\| u_h-\Pi_k u\|_\Omega$, superconverges, as is the case with the RT method.

\subsection{The influence of reaction terms}\label{sec:4.5}

\paragraph{Reaction-diffusion problems.} In this section we will have a look at how the analysis of RT discretization is adapted for the following simple modification of our equations:
\begin{subequations}\label{eq:B1}
\begin{alignat}{4}
\kappa^{-1}\mathbf q+\nabla u &=0 &\quad &\mbox{in $\Omega$},\\
\mathrm{div}\,\mathbf q +c\, u&=f&\quad &\mbox{in $\Omega$},\\
u&=g &&\mbox{on $\Gamma$,}
\end{alignat}
\end{subequations}
where $c:\Omega\to \mathbb R$ is a non-negative function. The seminorm
\[
| u|_c:=(c\,u,u)_\Omega^{1/2}=\| c^{1/2} u\|_\Omega,
\]
will play an important role in the energy analysis of this problem.

\paragraph{Discretization and error equations.}
The RT equations for problem \eqref{eq:B1} are
\begin{subequations}\label{eq:B2}
\begin{alignat}{6}
\label{eq:B2a}
&(\kappa^{-1}\mathbf q_h,\mathbf r)_{\mathcal T_h}-(u_h,\mathrm{div}\,\mathbf r)_{\mathcal T_h}+\langle \widehat u_h,\mathbf r\cdot\mathbf n\rangle_{\partial\mathcal T_h} &&=0 &\qquad &\forall\mathbf r \in \mathbf V_h,\\
\label{eq:B2b}
&(\mathrm{div}\,\mathbf q_h,w)_{\mathcal T_h}+(c\,u_h,w)_{\mathcal T_h}& &=(f,w)_{\mathcal T_h} & &\forall w\in W_h,\\
&\langle\mathbf q_h\cdot\mathbf n,\mu\rangle_{\partial\mathcal T_h\setminus\Gamma} &&=0 & &\forall\mu \in M_h^\circ,\\
&\langle \widehat u_h,\mu\rangle_\Gamma &&=\langle u_0,\mu\rangle_\Gamma & &\forall \mu \in M_h^\Gamma,
\end{alignat}
\end{subequations}
while projections satisfy the following discrete equations
\begin{subequations}\label{eq:B3}
\begin{alignat}{6}
\label{eq:B3a}
&(\kappa^{-1}\boldsymbol\Pi\mathbf q,\mathbf r)_{\mathcal T_h}-(\Pi  u,\mathrm{div}\,\mathbf r)_{\mathcal T_h}+\langle \mathrm P  u,\mathbf r\cdot\mathbf n\rangle_{\partial\mathcal T_h} &&=(\kappa^{-1}\boldsymbol\Pi\mathbf q-\mathbf q,\mathbf r)_{\mathcal T_h} &\quad &\forall\mathbf r \in \mathbf V_h,\\
\label{eq:B3b}
&(\mathrm{div}\,\boldsymbol\Pi\mathbf q,w)_{\mathcal T_h}+(c\,\Pi u,w)_{\mathcal T_h}& &=(f,w)_{\mathcal T_h}\\
&&& \phantom{=}+(c\,(\Pi u-u),w)_{\mathcal T_h} & &\forall w\in W_h,\\
&\langle\boldsymbol\Pi \mathbf q\cdot\mathbf n,\mu\rangle_{\partial\mathcal T_h\setminus\Gamma} &&=0 & &\forall\mu \in M_h^\circ,\\
&\langle  \mathrm P u,\mu\rangle_\Gamma &&=\langle u_0,\mu\rangle_\Gamma & &\forall \mu \in M_h^\Gamma.
\end{alignat}
\end{subequations}
Subtracting the discrete equations \eqref{eq:B2} from \eqref{eq:B3}, we get  the  error equations
\begin{subequations}\label{eq:B4}
\begin{alignat}{6}
\label{eq:B4a}
&(\kappa^{-1}\boldsymbol\varepsilon_h^q,\mathbf r)_{\mathcal T_h}-(\varepsilon_h^u,\mathrm{div}\,\mathbf r)_{\mathcal T_h}+\langle \widehat\varepsilon_h^u,\mathbf r\cdot\mathbf n\rangle_{\partial\mathcal T_h} &&=(\kappa^{-1}(\boldsymbol\Pi\mathbf q-\mathbf q),\mathbf r)_{\mathcal T_h} &\qquad &\forall\mathbf r \in \mathbf V_h,\\
\label{eq:B4b}
&(\mathrm{div}\,\boldsymbol\varepsilon_h^q,w)_{\mathcal T_h}+(c\,\varepsilon_h^u,w)_{\mathcal T_h}& &=(c\,(\Pi u-u),w)_{\mathcal T_h} & &\forall w\in W_h,\\
&\langle\boldsymbol\varepsilon_h^q\cdot\mathbf n,\mu\rangle_{\partial\mathcal T_h\setminus\Gamma} &&=0 & &\forall\mu \in M_h^\circ,\\
\label{eq:B4d}
&\langle  \widehat\varepsilon_h^u,\mu\rangle_\Gamma &&=0 & &\forall \mu \in M_h^\Gamma.
\end{alignat}
\end{subequations}
Testing the error equations \eqref{eq:B4} with $(\boldsymbol\varepsilon_h^q,\varepsilon_h^u,-\widehat\varepsilon_h^u,-\boldsymbol\varepsilon_h^q\cdot\mathbf n)$ and adding the result, we reach the
 new energy identity
\[
(\kappa^{-1}\boldsymbol\varepsilon_h^q,\boldsymbol\varepsilon_h^q)_{\mathcal T_h}+(c\, \varepsilon_h^u,\varepsilon_h^u)_{\mathcal T_h}=(\kappa^{-1}(\boldsymbol\Pi\mathbf q-\mathbf q),\boldsymbol\varepsilon_h^q)_{\mathcal T_h}+(c\,(\Pi u-u),\varepsilon_h^u)_{\mathcal T_h},
\]
thus proving the estimate
\begin{equation}\label{eq:B5}
\|\boldsymbol\varepsilon_h^q\|_{\kappa^{-1}}^2+|\varepsilon_h^u|_c^2\le \|\boldsymbol\Pi\mathbf q-\mathbf q\|_{\kappa^{-1}}^2+|\Pi u-u|_c^2.
\end{equation}
As can be seen from \eqref{eq:B5}, this couples back the estimates for the variable $\mathbf q$ with the approximation properties of $W_h$. The estimate (see \eqref{eq:HH1})
\begin{equation}\label{eq:B6}
\| \boldsymbol\varepsilon_h^q\cdot\mathbf n\|_h \lesssim \|\boldsymbol\varepsilon_h^q\|_{\Omega}
\end{equation}
is a purely finite dimensional one, independent of the equations satisfied by the discrete quantities. In a similar spirit, we can prove \eqref{eq:47} again, i.e., we obtain
\begin{equation}\label{eq:B7}
\|\mathrm P u-\widehat u_h\|_h=
\|\widehat\varepsilon_h^u\|_h \lesssim \|\varepsilon_h^u\|_\Omega + h \|\mathbf q-\mathbf q_h\|_{\kappa^{-1}}.
\end{equation}
This happens because this estimate depends only on the first error equation (the discretization of the equation $\kappa^{-1}\mathbf q+\nabla u=0$) which does not depend on the particular equilibrium equation. The argument to prove that
\begin{equation}\label{eq:B88}
\| \Pi u-u_h\|_\Omega=\|\varepsilon_h^u\|_\Omega \lesssim \| \mathbf q-\mathbf q_h\|_{\kappa^{-1}},
\end{equation}
was especifically based on the commutativity property for the projection and on the first error equation, so nothing has to be changed.

\paragraph{The duality estimate.} The duality argument becomes more complicated as we deal with more complex model problems. Instead of adapting the proof of the superconvergence estimate of the diffusion problem, we are going to show a more systematic way of proving estimates, a methodology that will be extremely useful in HDG analysis. We start with the dual problem
\begin{subequations}\label{eq:B8}
\begin{alignat}{4}
\kappa^{-1} \boldsymbol\xi -\nabla \theta &=0 &\qquad &\mbox{in $\Omega$},\\
-\mathrm{div}\,\boldsymbol\xi +c \,\theta &=\varepsilon_h^u &&\mbox{in $\Omega$},\\
\theta &=0 & &\mbox{on $\Gamma$}.
\end{alignat}
\end{subequations}
Note that this time we have changed signs in both first order operators.
We assume a  regularity hypothesis
\begin{equation}\label{eq:B9}
\| \boldsymbol\xi\|_{1,\Omega}+ \|c\,\theta\|_{1,\Omega}\le C_{\mathrm{reg}} \|\varepsilon_h^u\|_\Omega.
\end{equation}
We start by writing down the discrete equations satisfied by the projections $(\boldsymbol\Pi\boldsymbol\xi,\Pi\theta,\mathrm P\theta)$:
\begin{subequations}\label{eq:B10}
\begin{alignat}{6}
&(\kappa^{-1}\boldsymbol\Pi\boldsymbol\xi,\mathbf r)_{\mathcal T_h}+(\Pi  \theta,\mathrm{div}\,\mathbf r)_{\mathcal T_h}-\langle \mathrm P  \theta,\mathbf r\cdot\mathbf n\rangle_{\partial\mathcal T_h} &&=(\kappa^{-1}\boldsymbol\Pi\boldsymbol\xi-\boldsymbol\xi,\mathbf r)_{\mathcal T_h} &\quad &\forall\mathbf r \in \mathbf V_h,\\
&-(\mathrm{div}\,\boldsymbol\Pi\boldsymbol\xi,w)_{\mathcal T_h}+(c\,\Pi \theta,w)_{\mathcal T_h}& &=(\varepsilon_h^u,w)_{\mathcal T_h}\\
&&& \phantom{=}+(c\,(\Pi \theta-\theta),w)_{\mathcal T_h} & &\forall w\in W_h,\\
&\langle\boldsymbol\Pi \boldsymbol\xi\cdot\mathbf n,\mu\rangle_{\partial\mathcal T_h\setminus\Gamma} &&=0 & &\forall\mu \in M_h^\circ,\\
&\langle  \mathrm P \theta,\mu\rangle_\Gamma &&=\langle u_0,\mu\rangle_\Gamma & &\forall \mu \in M_h^\Gamma.
\end{alignat}
\end{subequations}
We now test the first three equations with $(\boldsymbol\varepsilon_h^q,\varepsilon_h^u,\widehat\varepsilon_h^u)$ and align terms carefully:
\begin{subequations}\label{eq:B11}
\begin{alignat}{8}
(\boldsymbol\Pi\boldsymbol\xi,\kappa^{-1}\boldsymbol\varepsilon_h^q)_{\mathcal T_h} &+(\Pi\theta,\mathrm{div}\,\boldsymbol\varepsilon_h^q)_{\mathcal T_h}&&-\langle \mathrm P\theta,\widehat\varepsilon_h^u\rangle_{\partial\mathcal T_h\setminus\Gamma}&&=(\boldsymbol\Pi\boldsymbol\xi-\boldsymbol\xi,\kappa^{-1}\varepsilon_h^q)_{\mathcal T_h},
\\
-(\mathrm{div}\,\boldsymbol\Pi\boldsymbol\xi,\varepsilon_h^u)_{\mathcal T_h}&+(\Pi\theta,c\,\varepsilon_h^u)_{\mathcal T_h}&& &&=\|\varepsilon_h^u\|_\Omega^2+(\Pi\theta-\theta,c\,\varepsilon_h^u)_{\mathcal T_h},\\
\langle\boldsymbol\Pi\boldsymbol\xi\cdot\mathbf n,\widehat\varepsilon_h^u\rangle_{\mathcal T_h} && && & =0.
\end{alignat}
\end{subequations}
Note that we have used twice that $\widehat\varepsilon_h^u=0$ on $\Gamma$ (this is the fourth of the error equations \eqref{eq:B4}. The next course of action is the addition of equations \eqref{eq:B11}. Close inspection of the columns of the tabulated system \eqref{eq:B11} shows the error equations \eqref{eq:B4} tested with $(\boldsymbol\Pi,\boldsymbol\xi,\Pi\theta,\mathrm P\theta)$. Therefore
\begin{eqnarray*}
&& \hspace{-2cm}(\kappa^{-1} (\boldsymbol\Pi\mathbf q-\mathbf q),\boldsymbol\Pi\boldsymbol\xi)_{\mathcal T_h}+(c(\Pi u-u),\Pi\theta)_{\mathcal T_h}\\
&& = \| \varepsilon_h^u\|_\Omega^2+(\kappa^{-1} (\boldsymbol\Pi\mathbf q-\mathbf q_h),\boldsymbol\Pi\boldsymbol\xi-\boldsymbol\xi)_{\mathcal T_h}+(c\,(\Pi u-u_h),\Pi\theta-\theta)_{\mathcal T_h}.
\end{eqnarray*}
We just reorganize this equality to get
\begin{alignat*}{4}
\| \varepsilon_h^u\|_\Omega^2=&(\kappa^{-1} (\boldsymbol\Pi\mathbf q-\mathbf q),\boldsymbol\Pi\boldsymbol\xi)_{\mathcal T_h}-
(\kappa^{-1} (\boldsymbol\Pi\mathbf q-\mathbf q_h),\boldsymbol\Pi\boldsymbol\xi-\boldsymbol\xi)_{\mathcal T_h}\\
&+(c(\Pi u-u),\Pi\theta)_{\mathcal T_h}-(c\,(\Pi u-u_h),\Pi\theta-\theta)_{\mathcal T_h}\\
=& (\kappa^{-1} (\mathbf q_h-\mathbf q),\boldsymbol\Pi\boldsymbol\xi-\boldsymbol\xi)_{\mathcal T_h}+(\kappa^{-1}(\boldsymbol\Pi\mathbf q-\mathbf q),\boldsymbol\xi)_{\mathcal T_h} &\quad &\mbox{(add and subtract $\boldsymbol\xi$)}\\
&+(c\,(u_h-u),\Pi\theta-\theta)_{\mathcal T_h}+(c\,(\Pi u-u),\theta)_{\mathcal T_h} &&\mbox{(add and subtract $\theta$)}\\
=&(\kappa^{-1} (\mathbf q_h-\mathbf q),\boldsymbol\Pi\boldsymbol\xi-\boldsymbol\xi)_{\mathcal T_h}+(c\,(u_h-u),\Pi\theta-\theta)_{\mathcal T_h} &&\mbox{(these are fine)}\\
&+(\boldsymbol\Pi\mathbf q-\mathbf q,\nabla\theta)_{\mathcal T_h} + (\Pi u-u,c\,\theta)_{\mathcal T_h}. && (\kappa^{-1}\boldsymbol\xi=\nabla\theta)
\end{alignat*}
The last two terms need some additional work. The second one is easy:
\[
(\Pi u-u,c\,\theta)_{\mathcal T_h}=(\Pi u-u,c\,\theta-\Pi (c\,\theta))_{\mathcal T_h}.
\]
In the first one we start with integration by parts
\begin{alignat*}{4}
(\boldsymbol\Pi\mathbf q-\mathbf q,\nabla\theta)_{\Omega} &=-(\mathrm{div}\,(\boldsymbol\Pi\mathbf q-\mathbf q),\theta)_\Omega +\langle (\boldsymbol\Pi\mathbf q-\mathbf q)\cdot\mathbf n,\theta\rangle_\Gamma &\quad&\mbox{(as $\boldsymbol\Pi\mathbf q-\mathbf q\in \mathbf H(\mathrm{div},\Omega)$)}\\
&=-(\mathrm{div}\,(\boldsymbol\Pi\mathbf q-\mathbf q),\theta)_\Omega &&\mbox{(BC for dual problem)}
\\
&=-(\Pi (\mathrm{div}\,\mathbf q)-\mathrm{div}\,\mathbf q,\theta)_\Omega &&\mbox{(commutativity prop)}\\
&=(\Pi (\mathrm{div}\,\mathbf q)-\mathrm{div}\,\mathbf q,\Pi\theta-\theta)_\Omega.
\end{alignat*}
Collecting these equalities and applying bounds on low order estimates for the projections we get
\begin{alignat*}{4}
\|\varepsilon_h^u\|_\Omega^2 \lesssim & h \|\mathbf q-\mathbf q_h\|_\Omega\,|\boldsymbol\xi|_{1,\Omega}+h|u-u_h|_c |\theta|_{1,\Omega}\\
&+ h\|\Pi(\mathrm{div}\,\mathbf q)-\mathrm{div}\,\mathbf q\|_\Omega |\theta|_{1,\Omega}+ h \|\Pi u-u\|_\Omega |c\,\theta|_{1,\Omega},
\end{alignat*}
which together with the regularity assumption \eqref{eq:B9} and the energy estimate \eqref{eq:B5} proves superconvergence:
\begin{equation}
\|\varepsilon_h^u\|_\Omega \lesssim h(\|\Pi\mathbf q-\mathbf q\|_\Omega+\|\Pi u-u\|_\Omega+\|\Pi(\mathrm{div}\,\mathbf q)-\mathrm{div}\,\mathbf q\|_\Omega).
\end{equation}

\paragraph{Some notes.}
As can be seen from these arguments, duality estimates are not a smooth ride, but they follow quite predictable patterns. The reader can wonder how it was the case that the duality estimate when $c=0$ seemed so much simpler. There is a simple reason: when $c=0$, then $\mathrm{div}\,\boldsymbol\varepsilon_h^q=0$, and it is simple to show (from the first error equation) that
\[
(\kappa^{-1}\boldsymbol\xi,\boldsymbol\varepsilon_h^q)_{\mathcal T_h}=0,
\]
which takes us back to some of the simpler estimates of Section \ref{sec:3.4}. Note also that when $k\ge 1$, we can write
\[
(\boldsymbol\Pi\mathbf q-\mathbf q,\nabla\theta)_\Omega=(\boldsymbol\Pi\mathbf q-\mathbf q,\nabla\theta-\boldsymbol\Pi_0(\nabla\theta))_\Omega\lesssim h \|\boldsymbol\Pi\mathbf q-\mathbf q\|_\Omega |\theta|_{2,\Omega},
\]
which leads to a slightly different regularity assumption and does not require integration by parts to make the additional power of $h$. This argument does not hold in the lower order case $k=0$, because the projection does not include any internal degrees of freedom.

\section{Introducing BDM}

In this section we go over all the needed changes to modify the projection-based analysis of RT elements to a similar analysis of a loosely called Brezzi-Douglas-Marini BDM element (we'll deal with names later on). For purposes of comparison, we will stick to the following table, lining up the boundary d.o.f. and not the space that we used for the variable $u_h$. (The definition of the N\'ed\'elec space $\mathcal N_{k-2}$ is given in Section \ref{sec:5.1}.)
\[
\begin{tabular}{|c|c|c|c|c|}
\hline 
degree$\phantom{\Big|}$ & $\mathbf q_h$ & $u_h$ & boundary d.o.f. & internal d.o.f.\\
\hline
$k\ge 0\phantom{\Big|}$ & $\mathcal{RT}_k(K)=\boldsymbol{\mathcal P}_k(K)+\mathbf m\tildeP_k(K)$ &  $\mathcal P_k(K)$& $\mathcal R_k(\partial K)$& $\boldsymbol{\mathcal P}_{k-1}(K)$\\
$k\ge 1\phantom{\Big|}$ & $\boldsymbol{\mathcal P}_k(K)$ & $\mathcal P_{k-1}(K)$ & $\mathcal R_k(\partial K)$ & $\mathcal N_{k-2}(K)$\\
\hline
\end{tabular}
\]

\subsection{The N\'ed\'elec space}\label{sec:5.1}

Consider the spaces
\[
\mathcal N_k(K):=\boldsymbol{\mathcal P}_k(K)\oplus \{ \mathbf q\in \widetilde{\boldsymbol{\mathcal P}}_{k+1}(K)\,:\, \mathbf q\cdot\mathbf m=0\},
\]
which obviously satisfies
\[
\boldsymbol{\mathcal P}_k(K)\subset \mathcal N_k(K)\subset \boldsymbol{\mathcal P}_{k+1}(K).
\]

\begin{proposition}\label{prop:5.1}\
\begin{itemize}
\item[{\rm (a)}] 
$
\mathrm{dim}\, \mathcal N_k(K)=d\,\mathrm{dim}\,\mathcal P_{k+1}(K)-\mathrm{dim} \,\widetilde{\mathcal P}_{k+2}(K).
$
\item[{\rm (b)}] 
$
\mathrm{dim}\,\mathcal N_{k-1}(K)+\mathrm{dim}\,\mathcal R_{k+1}(\partial K)=\mathrm{dim}\,\boldsymbol{\mathcal P}_{k+1}(K).
$
\item[{\rm (c)}]
$
\mathcal N_k(K)\oplus \nabla \widetilde{\mathcal P}_{k+2}(K)=\boldsymbol{\mathcal P}_{k+1}(K).
$
\item[{\rm (d)}]
$
\mathbf q\in \mathcal N_k(K) \, \Longleftrightarrow\,\widecheck{\mathbf q}\in \mathcal N_k(\widehat K).
$
\end{itemize}
\end{proposition}

\begin{proof}
It is easy to see that the linear operator
\[
\widetilde{\boldsymbol{\mathcal P}}_{k+1}(K) \ni \mathbf p \longmapsto T\mathbf p:=\mathbf p\cdot\mathbf m \in \widetilde{\mathcal P}_{k+2}(K) 
\]
is onto. Hence, 
\begin{alignat*}{4}
\dimm \mathcal N_k(K) &=\dimm \boldsymbol{\mathcal P}_k(K)+ \dimm \mathrm{ker}\,T&\quad& (\mathcal N_k=\boldsymbol{\mathcal P}_k\oplus\mathrm{ker}\,T)\\
&= d \, \dimm \mathcal P_k(K) + \dimm \tildePP_{k+1}(K) -\dimm \tildeP_{k+2}(K) &&\mbox{($T$ is onto)}\\
&=d (\dimm \mathcal P_k(K)+\dimm \tildeP_{k+1}(K))-\dimm \tildeP_{k+2}(K),
\end{alignat*}
which proves (a). To prove (b), note that by (a)
\begin{alignat*}{4}
\dimm \mathcal N_{k-1}(K)+\dimm \mathcal R_{k+1}(K) &=d\, \dimm \mathcal P_k(K)-\dimm \tildeP_{k+1}(K)+ (d+1) \dimm \mathcal P_{k+1}(e) \\
&=d\,\dimm \mathcal P_k(K)+ d\,\dimm \tildeP_{k+1}(K),
\end{alignat*}
where $e$ denotes any of the faces of $K$. Let now
\[
\mathcal S_{k+1}=\{ \mathbf q\in \widetilde{\boldsymbol{\mathcal P}}_{k+1}(K)\,:\, \mathbf q\cdot\mathbf m=0\}.
\]
 On the one hand $\mathcal S_{k+1}+\nabla \tildeP_{k+2}\subseteq \tildePP_{k+1}$ and this sum is direct, since if $p\in \tildeP_{k+2}$, then $\nabla p\cdot\mathbf m=(k+2) p$ by the Euler homogeneous function theorem. Finally,
\begin{alignat*}{4}
\dimm (\mathcal S_{k+1}\oplus\nabla \tildeP_{k+2}) &= \dimm \mathcal S_{k+2}+\dimm \tildeP_{k+2} &\quad & \mbox{($\nabla$ is 1-1)}\\
&= d\, \dimm \tildeP_{k+1}-\dimm \tildeP_{k+2}+\dimm \tildeP_{k+2} & & \mbox{(computation to prove (a))}\\
&=\dimm \tildePP_{k+1},
\end{alignat*} 
and therefore $\mathcal S_{k+1}\oplus\nabla \tildeP_{k+2}= \tildePP_{k+1}$ and (c) holds.

To prove (d), we just need to show that if $\mathbf q \in \widetilde S_{k+1}$, then $\widecheck{\mathbf q}\in \mathcal N_k(\widehat K)$ (note that the transformation $\mathbf q\mapsto \widecheck{\mathbf q}$ is a bijection. Let then $\mathbf q\in \mathcal S_{k+1}$ and $\mathrm F_K(\widehat{\mathbf x})=\mathrm B_K\widehat{\mathbf x}+\mathbf b_K$. Then
\begin{alignat*}{4}
\widecheck{\mathbf q}(\widehat{\mathbf x})\cdot (\widehat{\mathbf x}+\mathrm B_K^{-1}\mathbf b_K) &=\mathrm B_K^\top \mathbf q(\mathrm F_K(\widehat{\mathbf x}))\cdot  (\widehat{\mathbf x}+\mathrm B_K^{-1}\mathbf b_K) &\qquad &\mbox{(definition of $\widecheck{\mathbf q}$)}\\
&=\mathbf q(\mathrm F_K(\widehat{\mathbf x}))\cdot (\mathrm B_K\widehat{\mathbf x}+\mathbf b_K)\\
&=\mathbf q(\mathrm F_K(\widehat{\mathbf x}))\cdot\mathrm F_K(\widehat{\mathbf x})=0. & &\mbox{(since $\mathbf q\in \mathcal S_{k+1}$)}
\end{alignat*}
If we now decompose $\widecheck{\mathbf q}\in \boldsymbol{\mathcal P}_{k+1}(\widehat K)=\widetilde{\mathbf q}+\mathbf q_k$, where $\widetilde{\mathbf q}\in \tildePP_{k+1}(\widehat K)$ and $\mathbf q_k\in \boldsymbol{\mathcal P}_k(\widehat K)$, then we have (with $\mathbf c:=\mathrm B_K^{-1}\mathbf b_K$)
\[
0=\widecheck{\mathbf q}(\widehat{\mathbf x})\cdot(\widehat{\mathbf x}+\mathbf c)=\underbrace{\widetilde{\mathbf q}(\widehat{\mathbf x})\cdot\widehat{\mathbf x}}_{\in \tildeP_{k+2}(\widehat K)}+\underbrace{\widetilde{\mathbf q}(\widehat{\mathbf x})\cdot\mathbf c+\mathbf q_k(\widehat{\mathbf x})\cdot(\widehat{\mathbf x}+\mathbf c)}_{\in \boldsymbol{\mathcal P}_{k+1}(\widehat K)},
\]
and therefore $\widetilde{\mathbf q}\cdot\mathbf m=0$, which implies that $\widetilde{\mathbf q}\in \mathcal S_{k+1}$ and therefore $\widecheck{\mathbf q}\in \mathcal N_k(\widehat K)$.
\end{proof}

\paragraph{The two dimensional spaces.} When $d=2$, it is easy to see that
\[
(q_1,q_2)\in \mathcal{RT}_k(K)\qquad\Longrightarrow\qquad (-q_2,q_1)\in \mathcal N_k(K)
\]
and
\[
\mathrm{dim}\,\mathcal{RT}_k(K)=\mathrm{dim}\,\mathcal N_k(K),
\]
and therefore
\[
(q_1,q_2)\in \mathcal{RT}_k(K)\qquad\Longleftrightarrow\qquad (-q_2,q_1)\in \mathcal N_k(K),
\]
which means that the N\'ed\'elec space follows from a $\pi/2$-rotation of the Raviart-Thomas space in two space dimensions.

\subsection{The BDM projection}

The BDM projection is the interpolation operator associated to a mixed finite element named after Franco Brezzi, Jim Douglas Jr, and Donatella Marini. The BDM element was first introduced in the two dimensional case, by Brezzi, Douglas, and Marini \cite{BrDoMa:1985}, with slightly different internal degrees of freedom from those we are going to see here. The three dimensional version that we will see here is due to Jean-Claude N\'ed\'elec \cite{Nedelec:1986}. There is another variant of this three dimensional BDM element due to Brezzi, Douglas, Ricardo Dur\'an and Michel Fortin \cite{BrDoDuFo:1987}.

\begin{framed}
\noindent{\bf The BDM projection.} Let $\mathbf q:K\to\mathbb R^d$ be sufficiently smooth. For $k\ge 1$, the BDM projection is $\boldsymbol\Pi^{\mathrm{BDM}}\mathbf q\in \mathcal{P}_k(K)$ characterized by the equations
\begin{subequations}\label{eq:BDM}
\begin{alignat}{4}
\label{eq:BDMa}
(\boldsymbol\Pi^{\mathrm{BDM}}\mathbf q,\mathbf r)_K &=(\mathbf q,\mathbf r)_K &\qquad &\forall \mathbf r\in \mathcal N_{k-2}(K),\\
\label{eq:BDMb}
\langle \boldsymbol\Pi^{\mathrm{BDM}}\mathbf q\cdot\mathbf n,\mu\rangle_{\partial K}&=\langle \mathbf q\cdot\mathbf n,\mu\rangle_{\partial K} & &\forall \mu \in \mathcal R_k(\partial K).
\end{alignat}
The associated scalar projection is $\Pi_{k-1} u\in \mathcal P_{k-1}(K)$
\begin{equation}
(\Pi_{k-1} u,v)_K=(u,v)_K \qquad \forall v\in \mathcal P_{k-1}(K).
\end{equation}
In the case $k=1$, equations \eqref{eq:BDMa} are void.
\end{subequations}
\end{framed}

\begin{proposition}[Definition of the BDM projection]
Equations \eqref{eq:BDMa} and \eqref{eq:BDMb} are uniquely solvable.
\end{proposition}

\begin{proof}
By Proposition \ref{prop:5.1}(b), these equations make up a square system of linear equations, so we only need to prove uniqueness of solution of the homogeneous problem. Let thus $\mathbf q\in \boldsymbol{\mathcal P}_k(K)$ satisfy
\begin{subequations}\label{eq:100}
\begin{alignat}{4}
\label{eq:100a}
(\mathbf q,\mathbf r)_K &=0 &\qquad &\forall \mathbf r\in \mathcal N_{k-2}(K),\\
\label{eq:100b}
\langle \mathbf q\cdot\mathbf n,\mu\rangle_{\partial K}&=0 & &\forall \mu \in \mathcal R_k(\partial K).
\end{alignat}
\end{subequations}
Equation \eqref{eq:100b} implies that $\mathbf q\cdot\mathbf n=0$ on $\partial K$.
Take now $u\in \tildeP_k(K)$ and note that
\begin{alignat*}{4}
(\mathbf q,\nabla u)_K &=-(\mathrm{div}\,\mathbf q,u)_K &\qquad & \mbox{(integration by parts and $\mathbf q\cdot\mathbf n=0$)}\\
&=-(\mathrm{div}\,\mathbf q,\Pi_{k-1}u)_K & & \mbox{($\mathrm{div}\,\mathbf q\in \mathcal P_{k-1}(K)$)}\\
&=(\mathbf q,\nabla \Pi_{k-1} u)_K & & \mbox{(integration by parts and $\mathbf q\cdot\mathbf n=0$)}\\
&=0. & & \mbox{($\nabla \Pi_{k-1} u\in \boldsymbol{\mathcal P}_{k-2}(K)\subset \mathcal N_{k-2}(K)$ and \eqref{eq:100a})}
\end{alignat*}
Therefore $(\mathbf q,\mathbf r)_K=0$ for all $\mathbf r\in \mathcal N_{k-2}(K) + \nabla\tildeP_{k-1}(K)=\boldsymbol{\mathcal P}_{k-1}(K)$ (this was proved in Proposition \ref{prop:5.1}(c)). This means that $\mathbf q\in \boldsymbol{\mathcal P}_k^\bot(K)$ and $\mathbf q\cdot\mathbf n=0$ on $\partial K$, which implies (by Lemma \ref{lemma:2.1}(b)) that $\mathbf q=\mathbf 0$.
\end{proof}

\paragraph{The commutativity property.} For all $\mathbf q$ and $u\in \mathcal P_{k-1}(K)$,
\begin{alignat*}{4}
(\mathrm{div}\,\boldsymbol\Pi^{\mathrm{BDM}}\mathbf q,u)_K &=\langle\boldsymbol\Pi^{\mathrm{BDM}}\mathbf q\cdot\mathbf n,u\rangle_{\partial K}-(\boldsymbol\Pi^{\mathrm{BDM}}\mathbf q,\nabla u)_K\\
&=\langle\mathbf q\cdot\mathbf n, u\rangle_{\partial K}-(\mathbf q,\nabla u)_K & &\mbox{(by \eqref{eq:BDM} and $\nabla u\in \mathcal N_{k-2}(K)$)}\\
&=(\mathrm{div}\,\mathbf q,u)_K
\end{alignat*}
and therefore
\begin{equation}\label{eq:53}
\mathrm{div}\,\boldsymbol\Pi^{\mathrm{BDM}}\mathbf q=\Pi_{k-1}(\mathrm{div}\,\mathbf q).
\end{equation}

\paragraph{Invariance by Piola transforms.} Let $\widehat{\boldsymbol\Pi}^{\mathrm{BDM}}$ be the BDM projection on the reference triangle. Then, using the formulas for change to the reference domain,
\[
(\widehat{\boldsymbol\Pi^{\mathrm{BDM}}\mathbf q},\widecheck{\mathbf r})_{\widehat K}=(\boldsymbol\Pi^{\mathrm{BDM}}\mathbf q,\mathbf r)_K=(\mathbf q,\mathbf r)_K=(\qq,\widecheck{\mathbf r})_{\widehat K} \qquad \forall\mathbf r\in \mathcal N_{k-2}(K),
\]
(see Proposition \ref{prop:5.1}(d)), and
\[
\langle\widehat{\boldsymbol\Pi^{\mathrm{BDM}}\mathbf q}\cdot\widehat{\mathbf n},\widehat\mu\rangle_{\partial\widehat K}=\langle\boldsymbol\Pi^{\mathrm{BDM}}{\mathbf q}\cdot\mathbf n,\mu\rangle_{\partial K}=\langle\mathbf q\cdot\mathbf n,\mu\rangle_{\partial K}=\langle \qq\cdot\widehat{\mathbf n},\widehat\mu\rangle_{\partial\widehat K}\qquad\forall \mu \in \mathcal R_k(\partial K),
\]
which proves that
\begin{equation}\label{eq:commBDM}
\widehat{\boldsymbol\Pi^{\mathrm{BDM}}\mathbf q}=\widehat{\boldsymbol\Pi}^{\mathrm{BDM}}\qq.
\end{equation}

\begin{proposition}[Estimates for the BDM projection]\label{prop:5.2}
On shape-regular triangulations and for sufficiently smooth $\mathbf q$,
\begin{itemize}
\item[{\rm (a)}]
$
\| \boldsymbol\Pi^{\mathrm{BDM}}\mathbf q\|_K \lesssim \|\mathbf q\|_K+ h_K|\mathbf q|_{1,K},
$
\item[{\rm (b)}]
$
\|\mathbf q-\boldsymbol\Pi^{\mathrm{BDM}}\mathbf q\|_K \lesssim h^{k+1}_K|\mathbf q|_{k+1,K},
$
\item[{\rm (c)}]
$
\|\mathrm{div}\,\mathbf q-\mathrm{div}\,\boldsymbol\Pi^{\mathrm{BDM}}\mathbf q\|_K\lesssim h^{k}_K|\mathrm{div}\,\mathbf q|_{k,K}.
$
\end{itemize}
\end{proposition}

\begin{proof}
The proof is almost identical to that of Proposition \ref{prop:2.4}. We first need to show that
\[
\| \widehat{\boldsymbol\Pi}^{\mathrm{BDM}}\widehat{\mathbf q}\|_{\widehat K}\lesssim \|\widehat{\mathbf q}\|_{\widehat K}+\|\widehat{\mathbf q}\cdot\widehat{\mathbf n}\|_{\partial\widehat K}\lesssim \|\widehat{\mathbf q}\|_{1,\widehat K}\qquad \forall \widehat{\mathbf q}\in \mathbf H^1(\widehat K),
\]
and then use a scaling argument, taking advantage of \eqref{eq:commBDM}, to move to the reference element, to prove (a).
Since $\widehat\Pi^{\mathrm{BDM}}$ preserves the space $\boldsymbol{\mathcal P}_k$, then
\[
\|\qq-\widehat{\boldsymbol\Pi}^{\mathrm{BDM}}\qq\|_{\widehat K}\lesssim |\qq|_{k+1,\widehat K}\qquad \forall \qq\in\mathbf H^{k+1}(\widehat K).
\]
Another scaling argument proves then (b). (Note that the details of these scaling arguments are the same as in Proposition \ref{prop:2.4}.) Finally (c) follows from \eqref{eq:53}.
\end{proof}

\begin{proposition}[BDM local lifting of the normal trace]\label{prop:5.3}
For $k\ge 1$, there exists a linear operator $\mathbf L^{\mathrm{BDM}}:\mathcal R_k(\partial K)\to \boldsymbol{\mathcal P}_k(K)$ such that
\[
(\mathbf L^{\mathrm{BDM}}\mu)\cdot\mathbf n=\mu \quad \mbox{and}\quad \|\mathbf L^{\mathrm{BDM}}\mu\|_K \lesssim h_K^{1/2}\|\mu\|_{\partial K}\qquad \forall \mu \in \mathcal R_k(\partial K).
\]
\end{proposition}

\begin{proof}
Let $\mathbf q=\mathbf L^{\mathrm{BDM}} \mu\in \boldsymbol{\mathcal P}_k(K)$ be defined as
$
\mathbf q=|J_K|^{-1}\mathrm B_K\qq\circ\mathrm G_K,
$
where $\qq\in \boldsymbol{\mathcal P}_k(\widehat K)$ is the solution of the discrete equations in the reference domain:
\begin{alignat*}{4}
(\qq,\mathbf r)_{\widehat K} &=0 &\qquad &\forall\mathbf r\in \mathcal N_{k-2}(\widehat K),\\
\langle\qq\cdot\widehat{\mathbf n},\xi\rangle_{\partial\widehat K}&=\langle\widecheck\mu,\xi\rangle_{\partial\widehat K} & & \forall \xi\in \mathcal R_k(\partial\widehat K).
\end{alignat*}
The remainder of the proof of Proposition \ref{prop:2.5} (essentially a scaling argument) can be used word by word to prove the result.
\end{proof}

\subsection{The BDM method}

\paragraph{Spaces and equations.} We start by redefining the discrete spaces
\[
\mathbf V_h:=\prod_{K\in \mathcal T_h}\boldsymbol{\mathcal P}_k(K), \qquad W_h:=\prod_{K\in \mathcal T_h} \mathcal P_{k-1}(K), \qquad M_h:=\prod_{e\in \mathcal E_h}\mathcal P_k(e),
\]
and similarly $M_h^\circ$ and $M_h^\Gamma$.

\begin{framed}
\noindent
We look for
\begin{subequations}\label{eq:BDMeq}
\begin{equation}
(\mathbf q_h,u_h,\widehat u_h)\in \mathbf V_h \times W_h \times M_h,
\end{equation}
satisfying
\begin{alignat}{6}
\label{eq:BDMeqb}
&(\kappa^{-1}\mathbf q_h,\mathbf r)_{\mathcal T_h}-(u_h,\mathrm{div}\,\mathbf r)_{\mathcal T_h}+\langle \widehat u_h,\mathbf r\cdot\mathbf n\rangle_{\partial\mathcal T_h} &&=0 &\qquad &\forall\mathbf r \in \mathbf V_h,\\
&(\mathrm{div}\,\mathbf q_h,w)_{\mathcal T_h}& &=(f,w)_{\mathcal T_h} & &\forall w\in W_h,\\
\label{eq:BDMeqd}
&\langle\mathbf q_h\cdot\mathbf n,\mu\rangle_{\partial\mathcal T_h\setminus\Gamma} &&=0 & &\forall\mu \in M_h^\circ,\\
\label{eq:BDMeqe}
&\langle \widehat u_h,\mu\rangle_\Gamma &&=\langle u_0,\mu\rangle_\Gamma & &\forall \mu \in M_h^\Gamma.
\end{alignat}
\end{subequations}
\end{framed}
\noindent
A reduced conforming formulation, involving $\mathbf q_h$ and $u_h$ only, can be obtained using $\mathbf V_h^{\mathrm{div}}=\mathbf V_h\cap \mathbf H(\mathrm{div},\Omega)$ as test space in \eqref{eq:BDMeqb} and noticing that \eqref{eq:BDMeqd} is equivalent to $\mathbf q_h\in \mathbf V_h^{\mathrm{div}}$.

\begin{proposition}\label{prop:5.4}
Equations \eqref{eq:BDMeq} have a unique solution.
\end{proposition}

\begin{proof} (This proof is a simple adaptation of the proof of Proposition \ref{prop:3.1}.)
Since $M_h\equiv M_h^\circ\oplus M_h^\Gamma$, we only need to take care of uniqueness of solution. Let then $(\mathbf q_h,u_h,\widehat u_h)\in \mathbf V_h\times W_h\times M_h$ be a solution of 
\begin{subequations}
\begin{alignat}{6}
\label{eq:55a}
&(\kappa^{-1}\mathbf q_h,\mathbf r)_{\mathcal T_h}-(u_h,\mathrm{div}\,\mathbf r)_{\mathcal T_h}+\langle \widehat u_h,\mathbf r\cdot\mathbf n\rangle_{\partial\mathcal T_h} &&=0 &\qquad &\forall\mathbf r \in \mathbf V_h,\\
\label{eq:55b}
&(\mathrm{div}\,\mathbf q_h,w)_{\mathcal T_h}& &=0 & &\forall w\in W_h,\\
\label{eq:55c}
&\langle\mathbf q_h\cdot\mathbf n,\mu\rangle_{\partial\mathcal T_h\setminus\Gamma} &&=0 & &\forall\mu \in M_h^\circ,\\
\label{eq:55d}
&\langle \widehat u_h,\mu\rangle_\Gamma &&=0 & &\forall \mu \in M_h^\Gamma.
\end{alignat}
\end{subequations}
Testing these equations with $(\mathbf q_h,u_h,-\widehat u_h,-\mathbf q_h\cdot\mathbf n)$ and adding the results, we show that $(\kappa^{-1}\mathbf q_h,\mathbf q_h)_{\mathcal T_h}=0$ and hence $\mathbf q_h=\mathbf 0$. Let us now go back to \eqref{eq:55a}, which after integration by parts and localization on a single element yields for all $K\in \mathcal T_h$
\begin{equation}\label{eq:56}
(\nabla u_h,\mathbf r)_K+\langle u_h-\widehat u_h,\mathbf r\cdot\mathbf n\rangle_{\partial K}=0\qquad \forall \mathbf r\in \boldsymbol{\mathcal P}_k(K).
\end{equation}
Let us now construct $\mathbf p\in \boldsymbol{\mathcal P}_k(K)$ satisfying (see \eqref{eq:BDM} and Proposition \ref{prop:5.3}):
\begin{subequations}
\begin{alignat}{4}
(\mathbf p,\mathbf r)_K&=0 &\qquad &\forall \mathbf r\in \mathcal N_{k-2}(K),\\
\label{eq:58b}
\langle \mathbf p\cdot\mathbf n,\mu\rangle_{\partial K}&=\langle u_h-\widehat u_h,\mu\rangle_{\partial K} &&\forall \mu \in \mathcal R_k(\partial K).
\end{alignat}
\end{subequations}
Using this function as the test function in \eqref{eq:56}, we prove that
\begin{equation}\label{eq:59}
0 =(\nabla u_h,\mathbf p)_K +\langle u_h-\widehat u_h,\mathbf p\cdot\mathbf n\rangle_{\partial K}=\langle u_h-\widehat u_h,u_h-\widehat u_h\rangle_{\partial K},
\end{equation}
where we have used $\mu=u_h-\widehat u_h\in \mathcal R_k(\partial K)$ in \eqref{eq:58b} and that $\nabla u_h\in \nabla \mathcal P_{k-1}(K)\subset \boldsymbol{\mathcal P}_{k-2}(K)\subset \mathcal N_{k-2}(K)$. 

(From here on, everything is just a line-by-line copy of the end of the proof of Proposition \ref{prop:3.1}(a), that is, uniqueness of solution of the RT equations.)
The equality \eqref{eq:59} implies that $u_h-\widehat u_h=0$ on $\partial K$ and \eqref{eq:56} shows then (take $\mathbf r=\nabla u_h$) that $u_h\equiv c_K$ in $K$ and $\widehat u_h=u_h\equiv c_K$ on $\partial K$. Since each interior face value of $\widehat u_h$ is reached from different elements, it is easy to see that we have proved that $u_h\equiv c$ and $\widehat u_h\equiv c$. However, equation \eqref{eq:55d} implies that $\widehat u_h=0$ on $\Gamma$, and the proof of uniqueness of solution of \eqref{eq:BDMeq} is thus finished.
\end{proof}

\subsection{Error analysis}\label{sec:5.4}

\paragraph{Energy estimates.} We start by redefining the local projections: we take $\boldsymbol\Pi\mathbf q$ to be the local BDM projection, $\Pi u$ to be the projection on $W_h$ ($\Pi u|_K:=\Pi_{k-1} u$) and $\mathrm P u$ to be (again) the orthogonal projection onto $M_h$. The discrete errors are the same quantities that we defined in \eqref{eq:37}
\[
\boldsymbol\varepsilon_h^q:=\boldsymbol\Pi \mathbf q-\mathbf q_h\in \mathbf V_h, \qquad \varepsilon_h^u:=\Pi u-u_h\in W_h,\qquad \widehat\varepsilon_h^u:=\mathrm Pu-\widehat u_h\in M_h,
\]
and the {\bf error equations} differ from those in \eqref{eq:38}
\begin{alignat*}{6}
&(\kappa^{-1}\boldsymbol\varepsilon_h^q,\mathbf r)_{\mathcal T_h}-(\varepsilon_h^u,\mathrm{div}\,\mathbf r)_{\mathcal T_h}+\langle \widehat\varepsilon_h^u,\mathbf r\cdot\mathbf n\rangle_{\partial\mathcal T_h} &&=(\kappa^{-1}(\boldsymbol\Pi\mathbf q-\mathbf q),\mathbf r)_{\mathcal T_h} &\qquad &\forall\mathbf r \in \mathbf V_h,\\
&(\mathrm{div}\,\boldsymbol\varepsilon_h^q,w)_{\mathcal T_h}& &=0 & &\forall w\in W_h,\\
&\langle\boldsymbol\varepsilon_h^q\cdot\mathbf n,\mu\rangle_{\partial\mathcal T_h\setminus\Gamma} &&=0 & &\forall\mu \in M_h^\circ,\\
&\langle  \widehat\varepsilon_h^u,\mu\rangle_\Gamma &&=0 & &\forall \mu \in M_h^\Gamma,
\end{alignat*}
only in the fact that the spaces and projections have been redefined. The energy estimate \eqref{eq:40}
\begin{equation}\label{eq:60}
\|\boldsymbol\Pi\mathbf q-\mathbf q_h\|_{\kappa^{-1}}=\|\boldsymbol\varepsilon_h^q\|_{\kappa^{-1}}\le \| \boldsymbol\Pi \mathbf q-\mathbf q\|_{\kappa^{-1}}
\end{equation}
is proved in exactly the same way.

\paragraph{A note on the energy estimate.} For purely diffusive problems, the estimate \eqref{eq:60}, together with the approximation properties of the BDM projection (especifically Proposition \ref{prop:5.2}(b)) yields  optimal convergence $\|\mathbf q-\mathbf q_h\|_\Omega \lesssim h^{k+1}$. However, for problems with a reaction term
\[
\mathrm{div}\,\mathbf q+c\, u=f,
\]
the same comments we made in Section \ref{sec:4.5} still hold, and we can prove again
\begin{equation}\label{eq:61}
\|\boldsymbol\varepsilon_h^q\|_{\kappa^{-1}}^2+|\varepsilon_h^u|_c^2\le \|\boldsymbol\Pi\mathbf q-\mathbf q\|_{\kappa^{-1}}^2+|\Pi u-u|_c^2.
\end{equation}
This estimate now implies that $\|\mathbf q-\mathbf q_h\|_\Omega \lesssim h^k$, due to the influence of the lower order polynomial degree of the space $W_h$.

\paragraph{More estimates.} Using
\[
\mathrm{div}\,\boldsymbol\xi=\varepsilon_h^u, \qquad \|\boldsymbol\xi\|_{1,\Omega}\le C \|\varepsilon_h^u\|_\Omega
\]
we can repeat the arguments of Section \ref{sec:3.3} (with the BDM projection now, and taking advantage again of the commutativity property), to reproof \eqref{eq:44}
\begin{equation}\label{eq:62}
\| \Pi u-u_h\|_\Omega=\|\varepsilon_h^u\|_\Omega \lesssim \| \mathbf q-\mathbf q_h\|_{\kappa^{-1}}.
\end{equation}
Taking 
\[
\mathbf r\in \boldsymbol{\mathcal P}_k(K)\qquad \mathbf r\cdot\mathbf n=\widehat\varepsilon_h^u, \qquad \|\mathbf r\|_K \lesssim  h_K^{1/2}\|\widehat\varepsilon_h^u\|_{\partial K}
\]
we can also prove \eqref{eq:47} for BDM
\begin{equation}\label{eq:63}
\|\mathrm P u-\widehat u_h\|_h=
\|\widehat\varepsilon_h^u\|_h \lesssim \|\varepsilon_h^u\|_\Omega + h \|\mathbf q-\mathbf q_h\|_{\kappa^{-1}}.
\end{equation}

\paragraph{Duality estimates.}
Once again, we consider the dual problem
\begin{alignat*}{4}
\kappa^{-1} \boldsymbol\xi -\nabla \theta &=0 &\qquad &\mbox{in $\Omega$},\\
\mathrm{div}\,\boldsymbol\xi &=\varepsilon_h^u &&\mbox{in $\Omega$},\\
\theta &=0 & &\mbox{on $\Gamma$},
\end{alignat*}
and  assume the following {\bf regularity hypothesis}: there exists $C$ such that
\[
\| \boldsymbol\xi\|_{1,\Omega}+ \|\theta\|_{2,\Omega}\le C \|\varepsilon_h^u\|_\Omega.
\]
The arguments of Section \ref{sec:3.4} can be repeated and
\[
\|\varepsilon_h^u\|_\Omega \lesssim h\, (  \|\mathbf q_h-\mathbf q\|_{\kappa^{-1}}  + \| f-\Pi f\|_\Omega),
\]
follows with the same proof. For $k\ge 2$, we can do slighly better when we bound
\[
|(f-\Pi f,\theta-\Pi \theta)_\Omega|\lesssim h^2 \|f-\Pi f\|_\Omega |\theta|_{2,\Omega}\lesssim h^2 \|f-\Pi f\|_\Omega \|\varepsilon_h^u\|_\Omega.
\]
(This estimate does not work for $k=1$, since then $\Pi$ is the projection on piecewise constants and cannot deliver the $h^2$ term.)
The general case can then be presented as
\begin{equation}\label{eq:64}
\|\varepsilon_h^u\|_\Omega \lesssim h\,   \|\mathbf q_h-\mathbf q\|_{\kappa^{-1}}  + h^{\min\{k,2\}}\| f-\Pi f\|_\Omega.
\end{equation}

\paragraph{Optimal convergence.} Approximation of the projections used in the projection-based analysis can be summarized (for smoothest solutions) as
\[
\|\boldsymbol\Pi\mathbf q-\mathbf q\|_\Omega \lesssim h^{k+1}, \qquad \|\Pi u-u\|_\Omega\lesssim h^k, \qquad \| \mathrm P u-u\|_h\lesssim h^{k+1}.
\]
With everything in our favor, the BDM equations \eqref{eq:BDM} provide the following error estimates
\begin{framed}
\begin{alignat*}{6}
\|\mathbf q-\mathbf q_h\|_\Omega &\lesssim h^{k+1}  & \|\boldsymbol\Pi\mathbf q-\mathbf q_h\|_\Omega&\lesssim h^{k+1} &\qquad & \mbox{(see \eqref{eq:60})}\\
\| u-u_h\|_\Omega &\lesssim h^{k} &\| \Pi u-u_h\|_\Omega &\lesssim h^{k+\min\{k,2\}} &\qquad &\mbox{(see \eqref{eq:64})}\\
\| u-\widehat u_h\|_h &\lesssim h^{k+1} & \|\mathrm P u-\widehat u_h\|_h &\lesssim h^{k+\min\{k,2\}} && \mbox{(see \eqref{eq:64} and \eqref{eq:63})}\\
\|\mathbf q\cdot\mathbf n-\mathbf q_h\cdot\mathbf n\|_h &\lesssim h^{k+1} \qquad & \|\boldsymbol\Pi\mathbf q\cdot\mathbf n-\mathbf q_h\cdot\mathbf n\|_h &\lesssim h^{k+1}
\end{alignat*}
\end{framed}
\noindent
Let it be noted that when there is a reaction term in the equation, convergence for $\mathbf q_h$ is subject to the estimate \eqref{eq:61} and reduced to $h^k$. This bound is dragged to all the superconvergence estimates.

\section{The Hybridizable Discontinuous Galerkin method}

In this section we show how the spaces of RT and BDM can be balanced to have equal polynomial degree. Stability will be restored using a discrete stabilization (not penalization) function.  This is how local quantities of RT, BDM, and HDG methods compare. Note that there is no natural finite element structure for $\mathbf q_h$, where we can recognize boundary and integral d.o.f. Instead, we will have a projection that integrates $(\mathbf q_h,u_h)$ in the same structure.
\[
\begin{tabular}{|c|c|c|c|c|}
\hline 
degree$\phantom{\Big|}$ & $\mathbf q_h$ & $u_h$ & boundary d.o.f. & internal d.o.f.\\
\hline
$k\ge 0\phantom{\Big|}$ & $\mathcal{RT}_k(K)=\boldsymbol{\mathcal P}_k(K)+\mathbf m\tildeP_k(K)$ &  $\mathcal P_k(K)$& $\mathcal R_k(\partial K)$& $\boldsymbol{\mathcal P}_{k-1}(K)$\\
$k\ge 1\phantom{\Big|}$ & $\boldsymbol{\mathcal P}_k(K)$ & $\mathcal P_{k-1}(K)$ & $\mathcal R_k(\partial K)$ & $\mathcal N_{k-2}(K)$\\
$k\ge 0 \phantom{\Big|}$ & $\boldsymbol{\mathcal P}_k(K)$ & $\mathcal P_k(K)$ & N.A. & N.A.
\\
\hline
\end{tabular}
\]
Let us start with some small talk. The Hybridizable Discontinuous Galerkin method can be understood as a further development of the Local Discontinuous Galerkin method, one of the many DG schemes covered in the celebrated {\em framework-style} paper of Arnold, Brezzi, Cockburn and Marini \cite{ArBrCoMa:2001}. While the trail of papers is not entirely obvious, premonitions of what was about to happen can be found in the treatment of hybridized mixed methods by Bernardo Cockburn and Jay Gopalakrishnan \cite{CoGo:2004}. Some time later, this fructified in another long framework-style paper of the previous authors and Raytcho Lazarov \cite{CoGoLa:2009}, setting the bases for a full development of HDG methods. Cockburn and collaborators have been pushing the limits of applicability of HDG ideas to many problems in continuum mechanics and physics. The analysis, as will be shown here, is based on a particular definition of a projection tailored to the HDG equations: its first occurrence was due to Cockburn, Gopalakrishnan and myself in \cite{CoGoSa:2010}.

\subsection{The HDG method}

For $k\ge 0$, consider the spaces
\[
\mathbf V_h:=\prod_{K\in \mathcal T_h}\boldsymbol{\mathcal P}_k(K), \qquad W_h:=\prod_{K\in \mathcal T_h}\mathcal P_k(K), \qquad M_h:=\prod_{e\in \mathcal E_h}\mathcal P_k(e),
\]
and the subspace decomposition $M_h=M_h^\circ \oplus M_h^\Gamma$. Consider also the {\bf stabilization function}
\[
\tau \in \prod_{K\in \mathcal T_h}\mathcal R_0(\partial K), \qquad 
\tau \ge 0\, \qquad \tau|_{\partial K}\neq 0 \quad \forall K.
\]
\begin{framed}
\noindent
We look for
\begin{subequations}\label{eq:HDGeq}
\begin{equation}
(\mathbf q_h,u_h,\widehat u_h)\in \mathbf V_h \times W_h \times M_h,
\end{equation}
satisfying
\begin{alignat}{6}
\label{eq:HDGeqb}
&(\kappa^{-1}\mathbf q_h,\mathbf r)_{\mathcal T_h}-(u_h,\mathrm{div}\,\mathbf r)_{\mathcal T_h}+\langle \widehat u_h,\mathbf r\cdot\mathbf n\rangle_{\partial\mathcal T_h} &&=0 &\qquad &\forall\mathbf r \in \mathbf V_h,\\
\label{eq:HDGeqc}
&(\mathrm{div}\,\mathbf q_h,w)_{\mathcal T_h}+\langle \tau(u_h-\widehat u_h),w\rangle_{\partial \mathcal T_h}& &=(f,w)_{\mathcal T_h} & &\forall w\in W_h,\\
\label{eq:HDGeqd}
&\langle\mathbf q_h\cdot\mathbf n+\tau(u_h-\widehat u_h),\mu\rangle_{\partial\mathcal T_h\setminus\Gamma} &&=0 & &\forall\mu \in M_h^\circ,\\
\label{eq:HDGeqe}
&\langle \widehat u_h,\mu\rangle_\Gamma &&=\langle u_0,\mu\rangle_\Gamma & &\forall \mu \in M_h^\Gamma.
\end{alignat}
\end{subequations}
\end{framed}

\paragraph{Some comments.} Equations \eqref{eq:HDGeqb} and \eqref{eq:HDGeqc} are local, given the fact that the spaces are discontinuous. The first of them is the same equation (with different spaces) as in the RT and BDM method. If $\tau$ were to be zero (this is not allowed in our choice of spaces), equation \eqref{eq:HDGeqc} would be the same equation that we had in RT and BDM. Note that, after integration by parts, we can also write \eqref{eq:HDGeqc} as
\[
-(\mathbf q_h,\nabla w)_{\mathcal T_h}+\langle \mathbf q_h\cdot\mathbf n+\tau(u_h-\widehat u_h),w\rangle_{\partial\mathcal T_h}=(f,w)_{\mathcal T_h}\qquad \forall w\in W_h,
\]
where the numerical flux
\begin{equation}\label{eq:66}
\widehat{\mathbf q}_h\cdot\mathbf n:=\mathbf q_h\cdot\mathbf n+\tau (u_h-\widehat u_h)\in \mathcal R_k(\partial K)\qquad \forall K,
\end{equation}
makes an appearance. Equation \eqref{eq:HDGeqd} imposes that this numerical flux is `single-valued' on all internal faces (actually, normal components cancel each other), so that the numerical flux $\widehat{\mathbf q}_h\cdot\mathbf n$ can be identified with an element of $M_h$.

\begin{proposition}
Equations \eqref{eq:HDGeq} have a unique solution.
\end{proposition}

\begin{proof}
(What follows is a slight adaptation of the proofs of Propositions \ref{prop:3.1}(a) (RT) and \ref{prop:5.4} (BDM) to the HDG equations.) We only need to show that any solution $(\mathbf q_h,u_h,\widehat u_h)\in \mathbf V_h\times W_h\times M_h$ of the homogeneous equations
\begin{alignat*}{6}
&(\kappa^{-1}\mathbf q_h,\mathbf r)_{\mathcal T_h}-(u_h,\mathrm{div}\,\mathbf r)_{\mathcal T_h}+\langle \widehat u_h,\mathbf r\cdot\mathbf n\rangle_{\partial\mathcal T_h} &&=0 &\qquad &\forall\mathbf r \in \mathbf V_h,\\
&(\mathrm{div}\,\mathbf q_h,w)_{\mathcal T_h}+\langle \tau(u_h-\widehat u_h),w\rangle_{\partial \mathcal T_h}& &=0 & &\forall w\in W_h,\\
&\langle\mathbf q_h\cdot\mathbf n+\tau(u_h-\widehat u_h),\mu\rangle_{\partial\mathcal T_h\setminus\Gamma} &&=0 & &\forall\mu \in M_h^\circ,\\
&\langle \widehat u_h,\mu\rangle_\Gamma &&=0 & &\forall \mu \in M_h^\Gamma,
\end{alignat*}
vanishes. Testing these equations with $(\mathbf q_h,u_h,-\widehat u_h,-\mathbf q_h\cdot\mathbf n-\tau(u_h-\widehat u_h))$ and adding the results, we easily prove that
\[
(\kappa^{-1}\mathbf q_h,\mathbf q_h)_{\mathcal T_h}+\langle \tau(u_h-\widehat u_h),u_h-\widehat u_h\rangle_{\partial\mathcal T_h}=0
\]
and therefore $\mathbf q_h=\mathbf 0$, $\tau(u_h-\widehat u_h)=0$ (we have used that $\tau\ge 0$) and
\begin{equation}\label{eq:67}
(\nabla u_h,\mathbf r)_K+\langle u_h-\widehat u_h,\mathbf r\cdot\mathbf n\rangle_{\partial K}=0\quad \forall \mathbf r\in \boldsymbol{\mathcal P}_k(K)\quad \forall K.
\end{equation}
In particular, we have
\[
\langle u_h-\widehat u_h,\mathbf r\cdot\mathbf n\rangle_{\partial K}=0\quad \forall \mathbf r\in \boldsymbol{\mathcal P}_k^\bot(K).
\]
This implies (Lemma \ref{lemma:2.2}) that $u_h-\widehat u_h=v$ on $\partial K$, where $v\in \mathcal P_k^\bot(K)$. However, since $\tau(u_h-\widehat u_h)=0$ and $\tau$ is at least positive in one face of $K$, then (by Lemma \ref{lemma:2.1}(a)) necessarily $v=0$ and thus $u_h-\widehat u_h=0$ on $\partial K$. Testing then \eqref{eq:67} with $\mathbf r=\nabla u_h$, we show that $u_h\equiv c_K$ on $K$ and $u_h=\widehat u_h\equiv c_K$ on $\partial K$. Proceeding as in the proof of Proposition \ref{prop:3.1}, we show that $u_h=0$ and $\widehat u_h=0$.
\end{proof}

\subsection{The HDG projection}

The analysis of the HDG method will follow the same pattern we have employed in the analysis of RT and BDM. We start by defining a tailored projection onto the discrete spaces that will be used to write error equations that mimic those of the hybridizable mixed methods. As opposed to the two separate projections for $\mathbf q$ and $u$ that were used in RT and BDM, here the projection will be defined for the pair $(\mathbf q,u)$. However, we will still denote $(\Pi^{\mathrm{HDG}} \mathbf q,\Pi^{\mathrm{HDG}} u)$, as if these projections were defined separately: correct, but cumbersome, notation would express these elements as components of a single operator.

\begin{framed}
\noindent{\bf The HDG projection.} Given sufficiently smooth $(\mathbf q,u):K\to \mathbb R^d\setminus\mathbb R$, we define
\[
(\boldsymbol\Pi^{\mathrm{HDG}}\mathbf q,\Pi^{\mathrm{HDG}}u):=(\boldsymbol\Pi^{\mathrm{HDG}}_q(\mathbf q,u),\Pi^{\mathrm{HDG}}_u(\mathbf q,u))\in \boldsymbol{\mathcal P}_k(K)\times \mathcal P_k(K)
\]
as the solution to the equations
\begin{subequations}\label{eq:HDG}
\begin{alignat}{4}
(\boldsymbol\Pi^{\mathrm{HDG}}\mathbf q,\mathbf r)_K &=(\mathbf q,\mathbf r)_K &\qquad &\forall \mathbf r\in \boldsymbol{\mathcal P}_{k-1}(K),\\
\label{eq:HDGb}
(\Pi^{\mathrm{HDG}} u,v)_K &=(u,v)_K & & \forall v\in \mathcal P_{k-1}(K),\\
\label{eq:HDGc}
\langle\boldsymbol\Pi^{\mathrm{HDG}}\mathbf q\cdot\mathbf n+\tau \Pi^{\mathrm{HDG}}u,\mu\rangle_{\partial K}&=\langle\mathbf q\cdot\mathbf n+\tau\, u,\mu\rangle_{\partial K} & &\forall\mu \in \mathcal R_{k}(K).
\end{alignat} 
\end{subequations}
\end{framed}

\begin{proposition}[Definition of the HDG projection]
Equations \eqref{eq:HDG} are uniquely solvable.
\end{proposition}

\begin{proof}
We first remark that
\[
 \dimm \boldsymbol{\mathcal P}_k(K)+\dimm \mathcal P_k(K)\\
 =  \dimm \boldsymbol{\mathcal P}_{k-1}(K)+\dimm \mathcal P_{k-1}(K)+\dimm \mathcal R_k(\partial K),
\]
and therefore we only need to show uniqueness. Let then $(\mathbf q,u)\in \boldsymbol{\mathcal P}_k(K)\times \mathcal P_k(K)$ satisfy 
\begin{subequations}\label{eq:111}
\begin{alignat}{4}
(\mathbf q,\mathbf r)_K &=0 &\qquad &\forall \mathbf r\in \boldsymbol{\mathcal P}_{k-1}(K),\\
(u,v)_K &=0 & & \forall v\in \mathcal P_{k-1}(K),\\
\label{eq:111c}
\langle\mathbf q\cdot\mathbf n+\tau  u,\mu\rangle_{\partial K}&=0 & &\forall\mu \in \mathcal R_{k}(K).
\end{alignat} 
\end{subequations}
Then $\mathbf q\in \boldsymbol{\mathcal P}_k^\bot(K)$ and $u\in \mathcal P_k^\bot(K)$. Testing \eqref{eq:111c} with $u|_{\partial K}$ and using Lemma \ref{lemma:2.2}, we prove that
\[
0 =\langle\mathbf q\cdot\mathbf n+\tau\, u,u\rangle_{\partial K}=\langle\tau\, u,u\rangle_{\partial K}=\langle\tau^{1/2} u,\tau^{1/2} u\rangle_{\partial K}
\]
and therefore $\tau^{1/2} u=0$ (we have used here that $\tau\ge 0$). We can now test \eqref{eq:111c} with $\mathbf q\cdot\mathbf n$ to prove that $\mathbf q\cdot\mathbf n=0$ on $\partial K$. By Lemma \ref{lemma:2.1}(b), it follows that $\mathbf q=0$. On the other hand, $\tau u=0$ on $\partial K$ and we have assumed that $\tau>0$ in at least one face of $K$. Lemma \ref{lemma:2.1}(a) proves then that $u=0$.
\end{proof}

\paragraph{Weak commutativity.} For general $(\mathbf q,u)$ and $v\in \mathcal P_k(K)$,
\begin{alignat*}{4}
(\mathrm{div}\,\boldsymbol\Pi^{\mathrm{HDG}}\mathbf q,v)_K &=\langle\boldsymbol\Pi^{\mathrm{HDG}}\mathbf q\cdot\mathbf n,v\rangle_{\partial K}-(\boldsymbol\Pi^{\mathrm{HDG}}\mathbf q,\nabla v)_K\\
&=\langle \mathbf q\cdot\mathbf n-\tau(\Pi^{\mathrm{HDG}} u-u),v\rangle_{\partial K}-(\mathbf q,\nabla v)_K &\qquad & \mbox{(by \eqref{eq:HDG})}\\
&=(\mathrm{div}\,\mathbf q,v)_K-\langle \tau(\Pi^{\mathrm{HDG}} u-u),v\rangle_{\partial K},
\end{alignat*}
which can be rewritten as
\begin{equation}\label{eq:commHDG}
(\mathrm{div}\,\boldsymbol\Pi^{\mathrm{HDG}}\mathbf q,v)_K+\langle\tau \,\Pi^{\mathrm{HDG}} u,v\rangle_{\partial K}=(\mathrm{div}\,\mathbf q,v)_K+\langle\tau\, u,v\rangle_{\partial K}\qquad \forall v\in \mathcal P_k(K).
\end{equation}
Compare this result with the clean commutativity properties of the RT \eqref{eq:commRT} and BDM \eqref{eq:commBDM}  projections.

\paragraph{Change to the reference element.} Let $\widecheck\tau:=|a_K|\tau \circ \mathrm F_K|_{\partial\widehat K}$. Consider then the projection $(\widehat{\boldsymbol\Pi}^{\mathrm{HDG}},\widehat\Pi^{\mathrm{HDG}})$ associated to the stabilization function $\widecheck\tau$. It is then easy to show that
\begin{equation}\label{eq:71}
(\widehat{\boldsymbol\Pi^{\mathrm{HDG}}\mathbf q},\widehat{\Pi^{\mathrm{HDG}} u}) =(\widehat{\boldsymbol\Pi}^{\mathrm{HDG}}\qq,\widehat\Pi^{\mathrm{HDG}} \widehat u).
\end{equation}

\paragraph{Decoupling of the equations.} Let us first use $\mu=w|_{\partial K}$ in \eqref{eq:HDGc}, where $w\in \mathcal P_k^\bot(K)$. Then 
\begin{alignat*}{6}
\langle \tau (\Pi^{\mathrm{HDG}}u-u),w\rangle_{\partial K} &=\langle (\mathbf q-\boldsymbol\Pi^{\mathrm{HDG}}\mathbf q)\cdot\mathbf n,w\rangle_{\partial K}
\\
&=(\mathrm{div}\,\mathbf q-\mathrm{div}\,\boldsymbol\Pi^{\mathrm{HDG}}\mathbf q,w)_K+(\mathbf q-\boldsymbol\Pi^{\mathrm{HDG}}\mathbf q,\nabla w)_K\\
&=(\mathrm{div}\,\mathbf q,w)_K. \qquad\qquad\qquad  \mbox{(by \eqref{eq:HDGb} and since $w\in\mathcal P_k^\bot(K)$)}
\end{alignat*}
Therefore, the solution of \eqref{eq:HDG} also satisfies
\begin{subequations}\label{eq:72}
\begin{alignat}{4}
\label{eq:72a}
(\Pi^{\mathrm{HDG}}u,w)_K &=(u,w)_K &\qquad & \forall w\in \mathcal P_{k-1}(K),\\
\label{eq:72b}
\langle \tau\Pi^{\mathrm{HDG}}u,w\rangle_{\partial K} &=\langle \tau u,w\rangle_{\partial K}+(\mathrm{div}\,\mathbf q,w)_K, &&\forall w\in \mathcal P_k^\bot(K),
\end{alignat}
and
\begin{alignat}{4}
\label{eq:72c}
(\boldsymbol\Pi^{\mathrm{HDG}}\mathbf q,\mathbf r)_K&=(\mathbf q,\mathbf r)_K &\quad &\forall \mathbf r\in \boldsymbol{\mathcal P}_{k-1}(K),\\
\label{eq:72d}
\langle\boldsymbol\Pi^{\mathrm{HDG}}\mathbf q\cdot\mathbf n,\mu\rangle_{\partial K\setminus e}&=\langle\mathbf q\cdot\mathbf n,\mu\rangle_{\partial K\setminus e}+\langle \tau(u-\Pi^{\mathrm{HDG}}u),\mu\rangle_{\partial K\setminus e} && \forall \mu \in \mathcal R_k(\partial K\setminus e),
\end{alignat}
where $e$ is any face of $\partial K$.
\end{subequations}
Note that equations \eqref{eq:72a}--\eqref{eq:72b} are uniquely solvable by Lemma \ref{lemma:2.1}(a) and show that $\Pi^{\mathrm{HDG}}u$ depends on $u$ and $\mathrm{div}\,\mathbf q$. Equations \eqref{eq:72c} and \eqref{eq:72d} are also uniquely solvable, as follows from a comment at the end of the proof of Lemma \ref{lemma:2.1}(b). 

\paragraph{The single face HDG method.} A particular choice of the stabilization function $\tau$ was given in \cite{CoDoGu:2008}. It consists of choosing one particular $e_K\in \mathcal E(K)$ and taking $\tau_{\partial K}>0 $ in $e_K$ and $\tau_K\equiv 0$ in $\partial K\setminus e$. This shows (take $e=e_K$ in \eqref{eq:72d}) that the vector part of the HDG projection is completely decoupled from the scalar part for the SF--HDG case, and that it does not depend on $\tau$.

\subsection{Estimates for the HDG projection}\label{sec:6.3}

\paragraph{Notation.} For some forthcoming arguments, it will be useful to isolate the face
\[
\widehat e\in \mathcal E(\widehat K), \qquad \widehat e\subset \{\mathbf x\in \mathbb R^d\,:\, \mathbf x\cdot(1,\ldots,1)=1\}.
\]
From this moment on, {\em the symbol $\lesssim$ will include independence of the parameters $\tau$ as well.}

\begin{proposition}[Estimate on the reference domain -- Part I]\label{prop:6.2}
Given $u$, $f$, and $0\neq \widecheck\tau\in \mathcal R_0(\partial\widehat K)$, such that $\widecheck\tau\ge 0$ and $\widecheck\tau_\circ:=\widecheck\tau|_{\widehat e}>0$, we define $\widehat\Pi u \in \mathcal P_k(\widehat K)$ by solving the equations
\begin{subequations}\label{eq:73}
\begin{alignat}{4}
\label{eq:73a}
(\widehat\Pi u,w)_{\widehat K}&=(u,w)_{\widehat K} &\qquad &\forall w\in \mathcal P_{k-1}(\widehat K),\\
\label{eq:73b}
\langle\widecheck\tau \widehat\Pi u,w\rangle_{\partial\widehat K}&=\langle\widecheck\tau u,w\rangle_{\partial\widehat K}+(f,w)_{\widehat K}&&\forall w\in \mathcal P_k^\bot(\widehat K).
\end{alignat}
\end{subequations}
Then
\begin{eqnarray}
\label{eq:74}
\| u\|_{\widehat K} &\le & \widecheck\tau_\circ^{-1} \big( \|\widecheck\tau\|_{L^\infty} \| u\|_{1,\widehat K}+\| f\|_{\widehat K}\big),\\
\label{eq:75}
\| u-\widehat\Pi u\|_{\widehat K} &\le &\widecheck\tau_\circ^{-1}\big( \|\widecheck\tau\|_{L^\infty} | u|_{k+1,\widehat K}+|f|_{k,\widehat K}\big). 
\end{eqnarray}
\end{proposition}

\begin{proof}
Let $\delta:=\widehat\Pi u-\widehat\Pi_k u$ ($\widehat\Pi_k$ is the projection onto $\mathcal P_k(\widehat K)$), and note that $\delta\in \mathcal P_k^\bot(\widehat K)$ by \eqref{eq:73a}. Then
\begin{alignat*}{4}
\|\delta\|_{\widehat K}^2 &\lesssim \|\delta\|_{\widehat e}^2 &\qquad & \mbox{(conseq of Lemma \ref{lemma:2.1}(a))}\\
&=\widecheck\tau_\circ^{-1}\langle\widecheck\tau \delta,\delta\rangle_{\widecheck e}&&(\widecheck\tau|_{\widecheck e}=\widecheck\tau_\circ)\\
&\le \widecheck\tau_\circ^{-1}\langle\widecheck\tau\delta,\delta\rangle_{\partial\widehat K} &&(\widecheck\tau\ge 0)\\
&=\widecheck\tau_\circ^{-1}\big( \langle\widecheck\tau (u-\widehat\Pi_k u),\delta\rangle_{\partial K} + (f,\delta)_{\widehat K}\big) && (\delta\in \mathcal P_k^\bot \mbox{ and \eqref{eq:73b}}))\\
&\le \widecheck\tau_\circ^{-1} \big( \|\widecheck\tau\|_{L^\infty} \| u-\widehat\Pi_k u\|_{\partial\widehat K}\|\delta\|_{\partial\widehat K}+\| f\|_{\widehat K}\|\delta\|_{\widehat K}\big)\\
&\lesssim \widecheck\tau_\circ\big(  \|\widecheck\tau\|_{L^\infty}\|u-\widehat\Pi_k u\|_{\partial\widehat K}+\|f\|_{\widehat K}\big) \|\delta\|_{\widehat K} && \mbox{(finite dimensions)}\\
&\lesssim \widecheck\tau_\circ^{-1} \big( \|\widecheck\tau\|_{L^\infty} \| u-\widehat\Pi_k u\|_{1,\widehat K} +\|f\|_{\widehat K})\|\delta\|_{\widehat K} &&\mbox{(trace theorem)}\\
&\lesssim \widecheck\tau_\circ^{-1} \big( \|\widecheck\tau\|_{L^\infty} \| u\|_{1,\widehat K} +\|f\|_{\widehat K})\|\delta\|_{\widehat K}. &&\mbox{(finite dimensions)}
\end{alignat*}
Therefore
\[
\| \widehat\Pi u\|_{\widehat K}\le \|\widehat \Pi_k u\|_{\widehat K}+\|\delta\|_{\widehat K}\lesssim \| u\|_{\widehat K}+ \widecheck\tau_\circ^{-1} \big( \|\widecheck\tau\|_{L^\infty}\| u\|_{1,\widehat K}+\|f\|_{\widehat K}\big),
\]
and \eqref{eq:74} is thus proved. At the same time, note that we can substitute \eqref{eq:73b} by
\[
\langle\widecheck\tau \widehat\Pi u,w\rangle_{\partial\widehat K}=\langle\widecheck\tau u,w\rangle_{\partial\widehat K}+(f-\widehat\Pi_{k-1}f,w)_{\widehat K}\qquad \forall w\in \mathcal P_k^\bot(\widehat K).
\]
Returning to our previous argument, we have
\[
\|\widehat\Pi u-\widehat\Pi_k u\|_{\widehat K}\lesssim \widecheck\tau_\circ^{-1} \big( \|\widecheck\tau\|_{L^\infty}\|u-\widehat\Pi_k u\|_{1,\widehat K}+\| f-\widehat\Pi_{k-1} f\|_{\widehat K}\big),
\]
and \eqref{eq:75} follows from a compactness (Bramble-Hilbert style) argument.
\end{proof}

\begin{proposition}[Estimate on the reference domain -- Part II]\label{prop:6.3}
Given $\varepsilon$, $\mathbf q$, and $0\neq \widecheck\tau\in \mathcal R_0(\partial\widehat K)$, $\widecheck\tau\ge 0$, we define $\widehat{\boldsymbol\Pi}\mathbf q\in\boldsymbol{\mathcal P}_k(\widehat K)$ by solving the equations
\begin{subequations}
\begin{alignat}{4}
\label{eq:76a}
(\widehat{\boldsymbol\Pi}\mathbf q,\mathbf r)_{\widehat K} &=(\mathbf q,\mathbf r)_{\widehat K}&\qquad&\forall\mathbf r \in \boldsymbol{\mathcal P}_{k-1}(\widehat K),\\
\label{eq:76b}
\langle\widehat{\boldsymbol\Pi}\mathbf q\cdot\widehat{\mathbf n},\mu\rangle_{\partial\widehat K\setminus\widehat e} &=\langle\mathbf q\cdot\widehat{\mathbf n}+\widecheck\tau \varepsilon,\mu\rangle_{\partial\widehat K\setminus\widehat e} &&\forall \mu\in \mathcal R_k(\partial\widehat K\setminus\widehat e).
\end{alignat}
\end{subequations}
Then
\begin{eqnarray}
\label{eq:77}
\| \widehat{\boldsymbol\Pi}\mathbf q\|_{\widehat K} &\lesssim & \|\mathbf q\|_{1,\widehat K}+ \|\widecheck\tau\|_{L^\infty(\partial\widehat K\setminus\widehat e)} \| \varepsilon\|_{\partial\widehat K},\\
\label{eq:78}
\|\mathbf q-\widehat{\boldsymbol\Pi}\mathbf q\|_{\widehat K} &\lesssim & |\mathbf q|_{k+1,\widehat K}+ \|\widecheck\tau\|_{L^\infty(\partial\widehat K\setminus\widehat e)}\| \varepsilon\|_{\partial\widehat K}.
\end{eqnarray}
\end{proposition}

\begin{proof}
The stability estimate \eqref{eq:77} follows from a simple finite dimensional argument. To prove \eqref{eq:78} we compare with the componentwise $L^2$ projection onto $\boldsymbol{\mathcal P}_k(\widehat K)$. Since $\boldsymbol\delta:=\widehat{\boldsymbol\Pi}\mathbf q-\widehat{\boldsymbol\Pi}_k\mathbf q\in \boldsymbol{\mathcal P}_k^\bot(\widehat K)$, we can use an argument based on Lemma \ref{lemma:2.1}(b) to bound
\begin{alignat*}{4}
\| \boldsymbol\delta\|_{\widehat K}^2 &\lesssim \|\boldsymbol\delta\cdot\widehat{\mathbf n}\|_{\partial\widehat K\setminus\widehat e}^2\\
&=\langle \boldsymbol\delta\cdot\widehat{\mathbf n},\boldsymbol\delta\cdot\widehat{\mathbf n}\rangle_{\partial\widehat K\setminus\widehat e}\\
&=\langle \mathbf q\cdot\widehat{\mathbf n}-\widehat{\boldsymbol\Pi}_k\mathbf q\cdot\widehat{\mathbf n},\boldsymbol\delta+\widecheck\tau\varepsilon,\boldsymbol\delta\rangle_{\partial\widehat K\setminus\widehat e} &\qquad &\mbox{(by \eqref{eq:76b})}\\
&\le \big( \|\mathbf q\cdot\widehat{\mathbf n}-\widehat{\boldsymbol\Pi}_k\mathbf q\cdot\widehat{\mathbf n}\|_{\partial\widehat K}+\|\widecheck\tau\|_{L^\infty(\partial\widehat K\setminus\widehat e)} \|\varepsilon\|_{\partial\widehat K}\big)\|\boldsymbol\delta\|_{\partial\widehat K}\\
&\lesssim \big( \|\mathbf q-\widehat{\boldsymbol\Pi}_k\mathbf q\|_{1,\widehat K}++\|\widecheck\tau\|_{L^\infty(\partial\widehat K\setminus\widehat e)} \|\varepsilon\|_{\partial\widehat K}\big)\|\boldsymbol\delta\|_{\widehat K}. &&\mbox{(trace thm and finite dim)}
\end{alignat*}
The result then follows from a compactness argument.
\end{proof}

\begin{proposition}[Estimates for the HDG projection]
Given $\mathbf q, u$, and $0\neq \tau\in \mathcal R_0(\partial K)$, $\tau \ge 0$,
\begin{subequations}
\begin{alignat}{4}
\| u-\Pi^{\mathrm{HDG}} u\|_K &\lesssim h_K^{k+1}\big( |u|_{k+1,K}+\tau_{\max}^{-1} |\mathrm{div}\,\mathbf q|_{k,K}\big),\\
\| \mathbf q-\boldsymbol\Pi^{\mathrm{HDG}}\mathbf q\|_K &\lesssim h_K^{k+1} \big(|\mathbf q|_{k+1,K}+\tau^\star |u|_{k+1,K}\big),
\end{alignat}
\end{subequations}
with $\tau_{\max}:=\|\tau\|_{L^\infty}$ and $\tau^\star:=\|\tau\|_{L^\infty(\partial K\setminus e)}$, where $\tau|_e=\tau_{\max}$.
\end{proposition}

\begin{proof}
The estimate for $u$ follows from Proposition \ref{prop:6.2} and a scaling argument. Note first that by \eqref{eq:71} we can study the error on the reference element. Doing as in \eqref{eq:72}, we have
\begin{alignat*}{4}
(\widehat\Pi^{\mathrm{HDG}}\widehat u,w)_{\widehat K} &=(\widehat u,w)_{\widehat K} &\qquad & \forall w\in \mathcal P_{k-1}(\widehat K),\\
\langle \widecheck\tau\widehat\Pi^{\mathrm{HDG}}\widehat u,w\rangle_{\partial \widehat K} &=\langle \widecheck\tau \widehat u,w\rangle_{\partial \widehat K}+(\widehat{\mathrm{div}}\,\widehat{\mathbf q},w)_{\widehat K} &&\forall w\in \mathcal P_k^\bot(\widehat K),\\
(\widehat{\boldsymbol\Pi}^{\mathrm{HDG}}\widehat{\mathbf q},\mathbf r)_{\widehat K}&=(\widehat{\mathbf q},\mathbf r)_{\widehat K} &\qquad &\forall \mathbf r\in \boldsymbol{\mathcal P}_{k-1}(\widehat K),\\
\langle\widehat{\boldsymbol\Pi}^{\mathrm{HDG}}\widehat{\mathbf q}\cdot\mathbf n,\mu\rangle_{\partial \widehat K\setminus \widehat e}&=\langle\widehat{\mathbf q}\cdot\mathbf n,\mu\rangle_{\partial \widehat K\setminus \widehat e}+\langle \widecheck\tau(\widehat u-\widehat \Pi^{\mathrm{HDG}}\widehat u),\mu\rangle_{\partial \widehat K\setminus \widehat e} && \forall \mu \in \mathcal R_k(\partial \widehat K\setminus \widehat e),
\end{alignat*}
where the transformation $\mathrm F_K:\widehat K\to K$ is chosen so that  $\mathrm F_K(\widehat e)=e$, where $e\in \mathcal E(K)$ is such that $\tau|_e=\tau_{\max}$. We apply first Proposition \ref{prop:6.2} with 
\begin{equation}\label{eq:80}
f=\widecheck{\mathrm{div}\,\mathbf q}=\widehat{\mathrm{div}}\,\widehat{\mathbf q}=|J_K| \widehat{\mathrm{div}\,\mathbf q},\quad \widecheck\tau_\circ \approx \tau_{\max} h_K^{d-1} \approx \|\widecheck\tau\|_{L^\infty},\quad \|\widecheck\tau\|_{L^\infty(\partial\widehat K\setminus\widehat e)} \lesssim \tau^\star h_K^{d-1}.
\end{equation}
Then, 
\begin{alignat*}{4}
\| u-\Pi^{\mathrm{HDG}} u\|_K & \approx h_K^{\frac{d}2} \| \widehat u-\widehat\Pi^{\mathrm{HDG}}\widehat u\|_{\widehat K} &\qquad &\mbox{(by \eqref{eq:6} and \eqref{eq:71})}\\
&\lesssim h_K^{\frac{d}2}\widecheck\tau_\circ^{-1}\big(\|\widecheck\tau\|_{L^\infty} | \widehat u|_{k+1,\widehat K}+|\widehat{\mathrm{div}}\,\widehat{\mathbf q}|_{k,\widehat K}\big) && \mbox{(by Proposition \ref{prop:6.2})}\\
&\lesssim h_K^{\frac{d}2} |\widehat u|_{k+1,\widehat K}+ \tau_{\max}^{-1}h_K^{1-\frac{d}2} |J_K| |\widehat{\mathrm{div}\,\mathbf q}|_{k,\widehat K}&&\mbox{(by \eqref{eq:80})}\\
&\lesssim h_K^{\frac{d}2} |\widehat u|_{k+1,\widehat K}+ \tau_{\max}^{-1}h_K^{1+\frac{d}2} |\widehat{\mathrm{div}\,\mathbf q}|_{k,\widehat K}&&\mbox{(by \eqref{eq:5})}\\
&\lesssim h_K^{k+1} \big(|u|_{k+1,K}+\tau_{\max}^{-1}|\mathrm{div}\,\mathbf q|_{k,K}\big), &&\mbox{(by \eqref{eq:8})}
\end{alignat*}
and consequently
\begin{alignat*}{4}
\| \widehat u-\Pi^{\mathrm{HDG}}\widehat u\|_{\partial\widehat K}&\lesssim \|\widehat u-\Pi^{\mathrm{HDG}}\widehat u\|_{1,\widehat K} &\qquad &\mbox{(trace theorem)}\\
&\le \|\widehat u-\widehat\Pi_k \widehat u\|_{1,\widehat K}+\|\widehat\Pi_k\widehat u-\widehat\Pi^{\mathrm{HDG}}\widehat u|_{1,\widehat K} \\
& \lesssim |\widehat u|_{k+1,\widehat K}+\|\widehat\Pi_k\widehat u-\widehat\Pi^{\mathrm{HDG}}\widehat u\|_{\widehat K}&&\mbox{(compactness and finite dim.)}\\
&\lesssim |\widehat u|_{k+1,\widehat K}+\|\widehat u-\widehat\Pi^{\mathrm{HDG}}\widehat u\|_{\widehat K}\\
&\lesssim h_K^{-\frac{d}2} h_K^{k+1}\big( |u|_{k+1,K}+\tau_{\max}^{-1}|\mathrm{div}\,\mathbf q|_{k,K}\big).
\end{alignat*}
We then apply Proposition \ref{prop:6.3} with $\varepsilon:=\widehat u-\widehat\Pi^{\mathrm{HDG}}\widehat u$, so that
\begin{alignat*}{4}
\|\mathbf q-\Pi^{\mathrm{HDG}}\mathbf q\|_K &\approx h_K^{1-\frac{d}2} \|\widehat{\mathbf q}-\widehat\Pi^{\mathrm{HDG}}\widehat{\mathbf q}\|_{\widehat K} &\qquad &\mbox{(by \eqref{eq:6} and \eqref{eq:71})}
\\
&\lesssim h_K^{1-\frac{d}2} \big(|\widehat{\mathbf q}|_{k+1,\widehat K}+ \|\widecheck\tau\|_{L^\infty(\partial\widehat K\setminus\widehat e)}\| \widehat u-\Pi^{\mathrm{HDG}}\widehat u\|_{\partial\widehat K}\big) &&\mbox{(by Proposition \ref{prop:6.3})}\\
&\lesssim h_K^{k+1} |\mathbf q|_{k+1,K}+ \tau^\star h_K^{\frac{d}2}\| \widehat u-\Pi^{\mathrm{HDG}}\widehat u\|_{\partial\widehat K}&&\mbox{(by \eqref{eq:8} and \eqref{eq:80})}\\
&\lesssim h_K^{k+1}\big(|\mathbf q|_{k+1,K}+\tau^\star\|u|_{k+1,K}+\tau^\star\,\tau_{\max}^{-1}|\mathrm{div}\,\mathbf q|_{k,K}\big)\\
&\lesssim h_K^{k+1}\big( |\mathbf q|_{k+1,K}+\tau^\star|u|_{k+1,K}\big).
\end{alignat*}
This finishes the proof.
\end{proof}

\paragraph{An important observation.} The entire analysis holds if we change $\tau$ by $-\tau$ in the definition of the projection. This is equivalent to changing the orientation of the normal vector and, as such, to a simple change of signs in some terms in the right-hand sides of the decoupled problems \eqref{eq:72}.

\subsection{Error analysis: energy arguments}

\paragraph{Error equations.} We start by redefining the local projections: we take $(\boldsymbol\Pi\mathbf q,\Pi u)$ to be the local HDG projection and $\mathrm P u$ to be (again) the orthogonal projection onto $M_h$. The discrete errors are the same quantities that we defined in \eqref{eq:37}
\[
\boldsymbol\varepsilon_h^q:=\boldsymbol\Pi \mathbf q-\mathbf q_h\in \mathbf V_h, \qquad \varepsilon_h^u:=\Pi u-u_h\in W_h,\qquad \widehat\varepsilon_h^u:=\mathrm Pu-\widehat u_h\in M_h.
\]
We will also consider the error in the fluxes:
\begin{alignat*}{4}
\widehat\varepsilon_h^q &:= \boldsymbol\Pi \mathbf q\cdot\mathbf n+\tau(\Pi u-\mathrm P u)-\big(\mathbf q_h\cdot\mathbf n+\tau(u_h-\widehat u_h)\big)\\
&= \boldsymbol\varepsilon_h^q\cdot\mathbf n+\tau (\varepsilon_h^u-\widehat\varepsilon_h^u)\\
&=\mathrm P(\mathbf q\cdot\mathbf n)-\big(\mathbf q_h\cdot\mathbf n+\tau(u_h-\widehat u_h)\big) &\qquad &\mbox{(see \eqref{eq:HDGc})}\\
&=\mathrm P(\mathbf q\cdot\mathbf n)-\widehat{\mathbf q}_h\cdot\mathbf n. &&\mbox{(see \eqref{eq:66})}
\end{alignat*}
This is how HDG projections and HDG equations interact:
\begin{subequations}\label{eq:81}
\begin{alignat}{6}
\label{eq:81a}
&(\kappa^{-1}\boldsymbol\Pi\mathbf q,\mathbf r)_{\mathcal T_h}-(\Pi  u,\mathrm{div}\,\mathbf r)_{\mathcal T_h}+\langle \mathrm P  u,\mathbf r\cdot\mathbf n\rangle_{\partial\mathcal T_h} &&=(\kappa^{-1}(\boldsymbol\Pi\mathbf q-\mathbf q),\mathbf r)_{\mathcal T_h} &\qquad &\forall\mathbf r \in \mathbf V_h,\\
\label{eq:81b}
&(\mathrm{div}\,\boldsymbol\Pi\mathbf q,w)_{\mathcal T_h}+\langle \tau (\Pi u-\mathrm P u),w\rangle_{\partial\mathcal T_h}& &=(f,w)_{\mathcal T_h} & &\forall w\in W_h,\\
&\langle\boldsymbol\Pi \mathbf q\cdot\mathbf n+\tau(\Pi u-\mathrm P u),\mu\rangle_{\partial\mathcal T_h\setminus\Gamma} &&=0 & &\forall\mu \in M_h^\circ,\\
&\langle  \mathrm P u,\mu\rangle_\Gamma &&=\langle u_0,\mu\rangle_\Gamma & &\forall \mu \in M_h^\Gamma.
\end{alignat}
\end{subequations}
Note that we have used the weak commutativity property \eqref{eq:commHDG} in \eqref{eq:81b}. The {\bf error equations} are the difference between the latter and the HDG equations \eqref{eq:HDG}:
\begin{subequations}\label{eq:82}
\begin{alignat}{6}
&(\kappa^{-1}\boldsymbol\varepsilon_h^q,\mathbf r)_{\mathcal T_h}-(\varepsilon_h^u,\mathrm{div}\,\mathbf r)_{\mathcal T_h}+\langle \widehat\varepsilon_h^u,\mathbf r\cdot\mathbf n\rangle_{\partial\mathcal T_h} &&=(\kappa^{-1}(\boldsymbol\Pi\mathbf q-\mathbf q),\mathbf r)_{\mathcal T_h} &\qquad &\forall\mathbf r \in \mathbf V_h,\\
&(\mathrm{div}\,\boldsymbol\varepsilon_h^q,w)_{\mathcal T_h}+\langle\tau(\varepsilon_h^u-\widehat\varepsilon_h^u),w\rangle_{\partial\mathcal T_h}& &=0 & &\forall w\in W_h,\\
&\langle\boldsymbol\varepsilon_h^q\cdot\mathbf n+\tau(\varepsilon_h^u-\widehat\varepsilon_h^u),\mu\rangle_{\partial\mathcal T_h\setminus\Gamma} &&=0 & &\forall\mu \in M_h^\circ,\\
\label{eq:82d}
&\langle  \widehat\varepsilon_h^u,\mu\rangle_\Gamma &&=0 & &\forall \mu \in M_h^\Gamma.
\end{alignat}
\end{subequations}
Once again, these equations replicate faithfully the error equations for mixed methods: taking $\tau=0$, we obtain again the error equations for RT \eqref{eq:36} and BDM (Section \ref{sec:5.4}), with new polynomial spaces and projections though.

\paragraph{Energy estimate.}
Testing equations \eqref{eq:82} with $(\boldsymbol\varepsilon_h^u,\varepsilon_h^u,-\widehat\varepsilon_h^u,-\widehat\varepsilon_h^q)$ and adding the results we obtain an {\bf energy identity}
\begin{equation}\label{eq:83}
(\kappa^{-1}\boldsymbol\varepsilon_h^q,\boldsymbol\varepsilon_h^q)_{\mathcal T_h}+\langle \tau(\varepsilon_h^u-\widehat\varepsilon_h^u),\varepsilon_h^u-\widehat\varepsilon_h^u\rangle_{\partial\mathcal T_h}=(\kappa^{-1}(\boldsymbol\Pi\mathbf q-\mathbf q),\boldsymbol\varepsilon_h^q)_{\mathcal T_h},
\end{equation}
and a corresponding {\bf energy estimate} that uses a parameter dependent seminorm
\[
|\mu|_\tau:=\langle \tau\,\mu,\mu\rangle_{\partial\mathcal T_h}^{1/2}=\Big(\sum_{K\in \mathcal T_h}\langle \tau\,\mu,\mu\rangle_{\partial K}\Big)^{1/2},
\]
\begin{framed}
\begin{equation}\label{eq:84}
\|\boldsymbol\varepsilon_h^q\|_{\kappa^{-1}}^2+|\varepsilon_h^u-\widehat\varepsilon_h^u|_\tau^2\le \| \boldsymbol\Pi \mathbf q-\mathbf q\|_{\kappa^{-1}}^2.
\end{equation}
\end{framed}

\paragraph{An estimate for the flux.}
As we saw in Section \ref{sec:3.2} (right before proving \eqref{eq:HH1}), we can bound
\[
h_K^{\frac12}\|\mathbf p\cdot\mathbf n\|_{\partial K}\lesssim \|\mathbf p\|_K\qquad \forall\mathbf p\in \boldsymbol{\mathcal P}_k(K),
\]
and therefore
\[
\|\widehat\varepsilon_h^q\|_h \le \|\boldsymbol\varepsilon_h^q\cdot\mathbf n\|_h+\|\tau(\varepsilon_h^u-\widehat\varepsilon_h^u)\|_h\lesssim \|\boldsymbol\varepsilon_h^q\|_\Omega + \max_{K\in \mathcal T_h}h_K\|\tau\|_{L^\infty(\partial K)}\, |\varepsilon_h^u-\widehat\varepsilon_h^u|_\tau,
\]
which together with the energy estimate \eqref{eq:84} yield our second estimate for HDG
\begin{framed}
\begin{equation}\label{eq:D1}
\| \widehat\varepsilon_h^q\|_h \lesssim \big(1+\max_{K\in \mathcal T_h} h_K\|\tau\|_{L^\infty(\partial K)}\big)\|\boldsymbol\Pi\mathbf q-\mathbf q\|_{\kappa^{-1}}.
\end{equation}
\end{framed}

\paragraph{Bound for $\widehat\varepsilon_h^u$.} If $k\ge 1$, the can locally lift the value $\widehat\varepsilon_h^u$ using the BDM lifting of Proposition \ref{prop:5.3}. The arguments used to prove \eqref{eq:63} are still valid, since they rely on the existence of a local lifting to the test space $\mathbf V_h$ and on the first error equation. We thus have
\begin{framed}
\begin{equation}\label{eq:D2}
\| \widehat\varepsilon_h^u\|_h \lesssim \|\varepsilon_h^u\|_\Omega+h \|\mathbf q-\mathbf q_h\|_{\kappa^{-1}} \qquad \mbox{for $k\ge 1$.}
\end{equation}
\end{framed}
\noindent This argument will guarantee superconvergence of $\widehat u_h$ to $\mathrm Pu$ whenever $u_h$ superconverges to $\Pi u$. This will be the goal of the next section.

\subsection{Error analysis: duality arguments}

\paragraph{Estimates by duality arguments.} In order to avoid some lengthy computations that have appeared in previous treatments of the duality arguments, we will use the more systematic approach of Section \ref{sec:4.5}. The first step is the consideration of a dual problem:
\begin{subequations}\label{eq:86}
\begin{alignat}{4}
\label{eq:86a}
\kappa^{-1} \boldsymbol\xi -\nabla \theta &=0 &\qquad &\mbox{in $\Omega$},\\
-\mathrm{div}\,\boldsymbol\xi &=\varepsilon_h^u &&\mbox{in $\Omega$},\\
\theta &=0 & &\mbox{on $\Gamma$}.
\end{alignat}
\end{subequations}
Because the balance of signs between $\boldsymbol\xi$ and $\theta$ has changed, we will call $(\boldsymbol\Pi\boldsymbol\xi,\Pi\theta)$ to the HDG projection corresponding to $-\tau$ (see the last comment of Section \ref{sec:6.3}). 
We now write some equations satisfied by the projections, namely what does for equations \eqref{eq:81} for problem \eqref{eq:86})
\begin{alignat*}{6}
&(\kappa^{-1}\boldsymbol\Pi \boldsymbol\xi,\mathbf r)_{\mathcal T_h}+(\Pi  \theta,\mathrm{div}\,\mathbf r)_{\mathcal T_h}-\langle \mathrm P  \theta,\mathbf r\cdot\mathbf n\rangle_{\partial\mathcal T_h} &&=(\kappa^{-1}(\boldsymbol\Pi\boldsymbol\xi-\boldsymbol\xi),\mathbf r)_{\mathcal T_h} &\qquad &\forall\mathbf r \in \mathbf V_h,\\
&-(\mathrm{div}\,\boldsymbol\Pi\boldsymbol\xi,w)_{\mathcal T_h}+\langle \tau (\Pi \theta-\mathrm P \theta),w\rangle_{\partial\mathcal T_h}& &=(\varepsilon_h^u,w)_{\mathcal T_h} & &\forall w\in W_h,\\
&\langle \boldsymbol\Pi \boldsymbol\xi \cdot\mathbf n-\tau(\Pi \theta-\mathrm P \theta),\mu\rangle_{\partial\mathcal T_h\setminus\Gamma} &&=0 & &\forall\mu \in M_h^\circ,\\
&\langle  \mathrm P \theta,\mu\rangle_\Gamma &&=0 & &\forall \mu \in M_h^\Gamma.
\end{alignat*}
Note how the second equation --the weak commutativity property--, and the third equation --the action of the projection on faces-- have changed signs because of the fact the we are using $-\tau$ instead of $\tau$.
We next go ahead and test with the errors of the solution to the HDG equation. We are going to align everything in a careful way, since we want to add by columns instead of by rows:
\begin{alignat*}{8}
&(\kappa^{-1}\boldsymbol\Pi \boldsymbol\xi,\boldsymbol\varepsilon_h^q)_{\mathcal T_h}&&+(\Pi  \theta,\mathrm{div}\,\boldsymbol\varepsilon_h^q)_{\mathcal T_h}&&-\langle \mathrm P  \theta,\boldsymbol\varepsilon_h^q\cdot\mathbf n\rangle_{\partial\mathcal T_h\setminus\Gamma} &&=(\kappa^{-1}(\boldsymbol\Pi\boldsymbol\xi-\boldsymbol\xi),\boldsymbol\varepsilon_h^q)_{\mathcal T_h},\\
&-(\mathrm{div}\,\boldsymbol\Pi\boldsymbol\xi,\varepsilon_h^u)_{\mathcal T_h}&&+\langle \tau \Pi \theta,\varepsilon_h^u\rangle_{\partial\mathcal T_h}&&-\langle\tau\mathrm P \theta,\varepsilon_h^u\rangle_{\partial\mathcal T_h\setminus\Gamma}& &=\|\varepsilon_h^u\|_\Omega^2,\\
&\langle\boldsymbol\Pi\boldsymbol\xi\cdot\mathbf n,\widehat\varepsilon_h^u\rangle_{\partial\mathcal T_h}&&-\langle\tau\Pi\theta,\widehat\varepsilon_h^u\rangle_{\partial\mathcal T_h}&&+\langle\tau\mathrm P\theta,\widehat\varepsilon_h^u\rangle_{\partial\mathcal T_h\setminus\Gamma} &&=0.
\end{alignat*}
In between, we have used that $\widehat\varepsilon_h^u=0$ on $\Gamma$ (this was the fourth error equation \eqref{eq:82d}) and $\mathrm P\theta=0$ on $\Gamma$. We now sume these three equalities, but organize terms by column:
\begin{alignat*}{4}
\|\varepsilon_h^u\|_\Omega^2+(\kappa^{-1}(\boldsymbol\Pi\boldsymbol\xi-\boldsymbol\xi),\boldsymbol\varepsilon_h^q)_{\mathcal T_h}=&(\kappa^{-1}\boldsymbol\Pi \boldsymbol\xi,\boldsymbol\varepsilon_h^q)_{\mathcal T_h}-(\mathrm{div}\,\boldsymbol\Pi\boldsymbol\xi,\varepsilon_h^u)_{\mathcal T_h}+\langle\boldsymbol\Pi\boldsymbol\xi\cdot\mathbf n,\widehat\varepsilon_h^u\rangle_{\partial\mathcal T_h}\\
&+(\Pi  \theta,\mathrm{div}\,\boldsymbol\varepsilon_h^q)_{\mathcal T_h}+\langle \tau \Pi \theta,\varepsilon_h^u\rangle_{\partial\mathcal T_h}-\langle\tau\Pi\theta,\widehat\varepsilon_h^u\rangle_{\partial\mathcal T_h}\\
&-\langle \mathrm P  \theta,\boldsymbol\varepsilon_h^q\cdot\mathbf n\rangle_{\partial\mathcal T_h\setminus\Gamma} 
-\langle\tau\mathrm P \theta,\varepsilon_h^u\rangle_{\partial\mathcal T_h\setminus\Gamma}+\langle\tau\mathrm P\theta,\widehat\varepsilon_h^u\rangle_{\partial\mathcal T_h\setminus\Gamma}\\
=&(\boldsymbol\Pi \boldsymbol\xi,\kappa^{-1}\boldsymbol\varepsilon_h^q)_{\mathcal T_h}-(\mathrm{div}\,\boldsymbol\Pi\boldsymbol\xi,\varepsilon_h^u)_{\mathcal T_h}+\langle\boldsymbol\Pi\boldsymbol\xi\cdot\mathbf n,\widehat\varepsilon_h^u\rangle_{\partial\mathcal T_h}\\
&+(\Pi  \theta,\mathrm{div}\,\boldsymbol\varepsilon_h^q)_{\mathcal T_h}+\langle  \Pi \theta,\tau(\varepsilon_h^u-\widehat\varepsilon_h^u)\rangle_{\partial\mathcal T_h}\\
&-\langle \mathrm P  \theta,\boldsymbol\varepsilon_h^q\cdot\mathbf n+\tau(\varepsilon_h^u-\widehat\varepsilon_h^u)\rangle_{\partial\mathcal T_h\setminus\Gamma}\\
=&(\boldsymbol\Pi\boldsymbol\xi,\kappa^{-1}(\boldsymbol\Pi\mathbf q-\mathbf q))_{\mathcal T_h},
\end{alignat*}
where we have used the error equations \eqref{eq:82}. What is left is a simple reorganization of terms in the above equality:
\begin{alignat*}{4}
\|\varepsilon_h^u\|_\Omega^2 &=(\boldsymbol\Pi\boldsymbol\xi,\kappa^{-1}(\boldsymbol\Pi\mathbf q-\mathbf q))_{\mathcal T_h}-(\boldsymbol\Pi\boldsymbol\xi-\boldsymbol\xi,\kappa^{-1}(\boldsymbol\Pi\mathbf q-\mathbf q_h))_{\mathcal T_h}\\
&=(\boldsymbol\Pi\boldsymbol\xi-\boldsymbol\xi, \kappa^{-1}(\mathbf q_h-\mathbf q))_{\mathcal T_h}+(\boldsymbol\xi,\kappa^{-1}(\boldsymbol\Pi\mathbf q-\mathbf q))_{\mathcal T_h} &&\mbox{(add and substract $\boldsymbol\xi$)}\\
&=(\boldsymbol\Pi\boldsymbol\xi-\boldsymbol\xi,\kappa^{-1}(\mathbf q_h-\mathbf q))_{\mathcal T_h}+(\nabla\theta,\boldsymbol\Pi\mathbf q-\mathbf q)_{\mathcal T_h} &&\mbox{(by equation \eqref{eq:86a})}\\
&=(\boldsymbol\Pi\boldsymbol\xi-\boldsymbol\xi,\kappa^{-1}(\mathbf q-\mathbf q_h))_{\mathcal T_h}+(\nabla\theta-\boldsymbol\Pi_{k-1}\nabla\theta,\boldsymbol\Pi\mathbf q-\mathbf q)_{\mathcal T_h}.
\end{alignat*}
Let us write this as an inequality:
\begin{alignat*}{4}
\|\varepsilon_h^u\|_\Omega^2\le & \|\boldsymbol\Pi\boldsymbol\xi-\boldsymbol\xi\|_{\mathcal T_h}\|\kappa^{-1/2}\|_{L^\infty} \|\mathbf q_h-\mathbf q\|_{\kappa^{-1}}+\|\nabla\theta-\boldsymbol\Pi_{k-1}\theta\|_{\mathcal T_h}\| \kappa^{1/2}\|_{L^\infty} \|\boldsymbol\Pi\mathbf q-\mathbf q\|_{\kappa^{-1}}\\
\lesssim &\big( \|\boldsymbol\Pi\boldsymbol\xi-\boldsymbol\xi\|_{\mathcal T_h}+\|\nabla\theta-\boldsymbol\Pi_{k-1}\theta\|_{\mathcal T_h}\|) \|\boldsymbol\Pi\mathbf q-\mathbf q\|_{\kappa^{-1}}.
\end{alignat*}
Assuming regularity
\[
\| \boldsymbol\xi\|_{1\Omega}+\|\theta\|_{2,\Omega}\le C_{\mathrm{reg}}\|\varepsilon_h^u\|_\Omega
\]
for the solution of \eqref{eq:86}, the above argument leads to 
\begin{framed}
\begin{equation}\label{eq:D3}
\|\varepsilon_h^u\|_\Omega \lesssim h^{\min\{k,1\}} \|\boldsymbol\Pi\mathbf q-\mathbf q\|_{\kappa^{-1}},
\end{equation}
\end{framed}
\noindent and hence to superconvergence when $k\ge 1$. For $k=0$, no regularity of the dual problem is needed. 

\paragraph{Wrap up paragraph.} The previous estimates together already studied of the HDG projection and of the projection $\mathrm P$, give the following table of convergence orders for smooth solutions.
\begin{framed}
\begin{alignat*}{6}
\|\mathbf q-\mathbf q_h\|_\Omega &\lesssim h^{k+1}  & \|\boldsymbol\Pi\mathbf q-\mathbf q_h\|_\Omega&\lesssim h^{k+1} &\qquad & \mbox{(see \eqref{eq:84})}\\
\| u-u_h\|_\Omega &\lesssim h^{k+1} &\| \Pi u-u_h\|_\Omega &\lesssim h^{k+1+\min\{k,1\}} &\qquad &\mbox{(see \eqref{eq:D3})}\\
\| u-\widehat u_h\|_h &\lesssim h^{k+1} & \|\mathrm P u-\widehat u_h\|_h &\lesssim h^{k+2} && \mbox{($k\ge 1$ -- see \eqref{eq:D2})}\\
\|\mathbf q\cdot\mathbf n-\widehat{\mathbf q}_h\cdot\mathbf n\|_h &\lesssim h^{k+1} \qquad & \|\mathrm P(\mathbf q\cdot\mathbf n)-\widehat{\mathbf q}_h\cdot\mathbf n\|_h &\lesssim h^{k+1}&& \mbox{(see \eqref{eq:D1})}
\end{alignat*}
\end{framed}

\section*{Acknowledgments}

I was introduced to Mixed Finite Elements  by Salim Meddahi. Years later, I learned about the Discontinuous Galerkin Method from Gabriel Gatica. Finally I got into the Hybridizable Discontinuous Galerkin Method through their creator, Bernardo Cockburn. The three of them have been great collaborators and even better friends for many years. What I now know about the topic, I know from them. I also want to thank Johnny Guzm\'an, for having helped me understand many tricky details about the analysis of mixed methods. My research is partially funded by the NSF--DMS 1216356 grant.

\bibliographystyle{abbrv}
\bibliography{referencesFEM}

\end{document}